\documentclass[a4paper,12pt]{amsart}

\usepackage[english]{babel}
\usepackage[english]{translator}
\usepackage[T1]{fontenc}
\usepackage[utf8]{inputenc}

\usepackage{amsmath}
\usepackage{amssymb}
\usepackage{amsfonts}
\usepackage{amstext}
\usepackage{amsthm}
\usepackage{bbm}

\usepackage{graphicx}
\usepackage{color}
\usepackage{subfigure}
\usepackage{float}

\usepackage{marvosym}
\usepackage{enumerate}
\usepackage{hyperref}
\usepackage[
  hmarginratio={1:1},     
  vmarginratio={1:1},     
  textwidth=460pt,        
  heightrounded,          
  bottom=3.65cm,
  top=3.65cm,
]{geometry}

\hypersetup{
  pdffitwindow=false,
  pdfhighlight=/O,
  pdfnewwindow,
  colorlinks=true,
  citecolor=red,             
  linkcolor=blue,            
  menucolor=blue,            
  urlcolor=blue,             
  pdfpagemode=UseOutlines,
  bookmarksnumbered=true,
  linktocpage,
  pdftitle={Nonlinear Stability of Rotating Waves in Parabolic Systems},
  pdfsubject={Nonlinear Stability of Rotating Waves in Parabolic Systems},
  pdfauthor={Wolf-J"urgen Beyn, Christian D"oding},
  pdfkeywords={},
  pdfcreator={pdflatex},
  pdfproducer={LaTeX with hyperref}
}

\newcommand{\D}{\mathcal{D}}              
\newcommand{\C}{\mathbb{C}}               
\newcommand{\R}{\mathbb{R}}               
\newcommand{\Z}{\mathbb{Z}}                
\newcommand{\N}{\mathbb{N}}                
\newcommand{\F}{\mathcal{F}}              
\renewcommand{\S}{\mathcal{S}}              
\renewcommand{\Re}{\mathrm{Re}\,}          
\renewcommand{\Im}{\mathrm{Im}\,}          
\renewcommand{\L}{\mathcal{L}}             

\newcommand{\one}{\mathbbm{1}}             

\renewcommand{\ker}{\mathcal{N}}
\newcommand{\ran}{\mathcal{R}}

\newcommand{\vek}[2]{\begin{pmatrix} #1 \\ #2 \end{pmatrix}}
\renewcommand{\u}{\mathbf{u}}
\renewcommand{\v}{\mathbf{v}}
\newcommand{\w}{\mathbf{w}}
\renewcommand{\F}{\mathcal{F}}
\newcommand{\G}{\mathcal{G}}
\renewcommand{\r}{\mathbf{r}}
\newcommand{\st}{\mathfrak{s}}

\makeatletter
\renewcommand{\@secnumfont}{\bfseries}
\makeatother
\makeatletter
  \def\section{\@startsection{section}{1}%
    \z@{.7\linespacing\@plus\linespacing}{.5\linespacing}%
    {\normalfont\LARGE\bfseries}}
\makeatother
\makeatletter
\def\@seccntformat#1{%
  \protect\textup{%
    \protect\@secnumfont
    \expandafter\protect\csname format#1\endcsname 
    \csname the#1\endcsname
    \protect\@secnumpunct
  }%
}

\newcommand{\sect}
{
  \setcounter{equation}{0}
  \setcounter{figure}{0}
  \section
}

\theoremstyle{definition}
\newtheorem{definition}{Definition}[section]
\newtheorem{assumption}[definition]{Assumption}

\newtheorem{remark}[definition]{Remark}

\theoremstyle{plain}
\newtheorem{theorem}[definition]{Theorem}
\newtheorem{lemma}[definition]{Lemma}

\allowdisplaybreaks

\begin{document}
\title[Stability of Traveling Oscillating Fronts \\ in Complex Ginzburg Landau Equations]{Stability of Traveling Oscillating Fronts in Complex Ginzburg Landau Equations}
\setlength{\parindent}{0pt}
\vspace*{0.4cm}
\begin{center}
\normalfont\LARGE\bfseries{\shorttitle}
\vspace*{12pt}
\end{center}

\begin{center}
Wolf-J{\"u}rgen Beyn\footnotemark[1] and Christian D{\"o}ding{\footnotemark[2]${}^{,}$\footnotemark[3]} \\
\vspace{12pt}
October 25, 2021
\end{center}

\footnotetext[1]{Department of Mathematics, Bielefeld University, 33501 Bielefeld, Germany, \\ e-mail: \textcolor{blue}{beyn@math.uni-bielefeld.de}, phone: \textcolor{blue}{+49 (0)521 106 4798}.}
\footnotetext[2]{Department of Mathematics, Ruhr-University Bochum, 44801 Bochum, Germany, \\ e-mail: \textcolor{blue}{christian.doeding@rub.de}, phone: \textcolor{blue}{+49 (0)234 32 19876}.}
\footnotetext[3]{This work is an extended version of parts of the author's
  PhD Thesis \cite{Doeding}.}

\vspace{12pt}
\noindent
\begin{center}
\begin{minipage}{0.8\textwidth}
  {\small
    \textbf{Abstract.}
    Traveling oscillating fronts (TOFs) are specific waves of the form
    $U_\star (x,t) = e^{-i \omega t} V_\star(x - ct)$ with a profile $V_{\star}$
    which decays at $- \infty$ but approaches a nonzero limit at $+\infty$.
    TOFs usually appear in complex Ginzburg Landau equations of the type $U_t = \alpha U_{xx} + G(|U|^2)U$. In this paper we prove a theorem on 
    the asymptotic stability of TOFs, where we allow the
    initial perturbation to be the  sum of an exponentially localized part  and
    a front-like part which approaches a small but nonzero limit at $+ \infty$.
    The underlying assumptions guarantee that
    the operator, obtained from linearizing about the TOF in  a co-moving and
    co-rotating frame, has essential spectrum touching the imaginary axis
    in a quadratic fashion and that further isolated eigenvalues are bounded away
    from the imaginary axis. The basic idea of the proof is to consider
    the problem in an extended phase space which couples the wave dynamics
    on the real line to the ODE dynamics at $+ \infty$. Using slowly
    decaying exponential weights, the framework allows to derive appropriate
     resolvent estimates, semigroup techniques, and Gronwall estimates.
  }
\end{minipage}
\end{center}

\vspace{12pt}

\vspace{12pt}

\sect{Introduction}
\label{sec1}
In this paper we consider complex-valued semilinear parabolic equations of the form
\begin{align} \label{Evo}
	U_t = \alpha U_{xx} + G(|U|^2)U, \quad x \in \R,\, t \ge 0
\end{align}
with nonlinearity $G:\R \rightarrow \C$ and diffusion coefficient $\alpha \in \C$, $\Re \alpha > 0$. If the nonlinearity $G$ is a linear resp. a quadratic polynomial over $\C$ then \eqref{Evo} leads to the cubic resp. the quintic complex Ginzburg Landau equation. Evolution equations of the form \eqref{Evo} admit the propagation of various types of waves which oscillate in time and which either
have a front profile or which are periodic in space like
wave trains, see \cite{SandstedeScheel04}, \cite{Saarloos}. We are interested in the stability behavior
of a special class of solutions which we call traveling oscillating fronts (TOFs).
A TOF is a solution of \eqref{Evo} of the form
\begin{align*}
	U_\star (x,t) = e^{-i \omega t} V_\star(x - ct)
\end{align*}
with a profile $V_\star: \R \rightarrow \C$ satisfying the asymptotic property
\begin{align*}
	\lim_{\xi \rightarrow -\infty} V_\star(\xi) = 0, \quad \lim_{\xi \rightarrow +\infty} V_\star(\xi) = V_{\infty}
\end{align*}
for some $V_{\infty} \in \C$, $V_{\infty} \neq 0$. The parameters $\omega,c \in \R$ are called the frequency and the velocity of the TOF.

Figure \ref{ExampleTOF}  shows a  typical TOF obtained by
simulating the quintic complex Ginzburg-Landau equation
\begin{align} \tag{QCGL} \label{QCGL}
  U_t = \alpha U_{xx} + \beta_1 U + \beta_3 |U|^2 U + \beta_5 |U|^4U, \quad x \in \R,\, t \ge 0
  \end{align}
with an initial function $U(\cdot,0)$ of sigmoidal shape.
\begin{figure}[h!]
\centering
\begin{minipage}[t]{0.45\textwidth}
\centering
\includegraphics[scale=0.45]{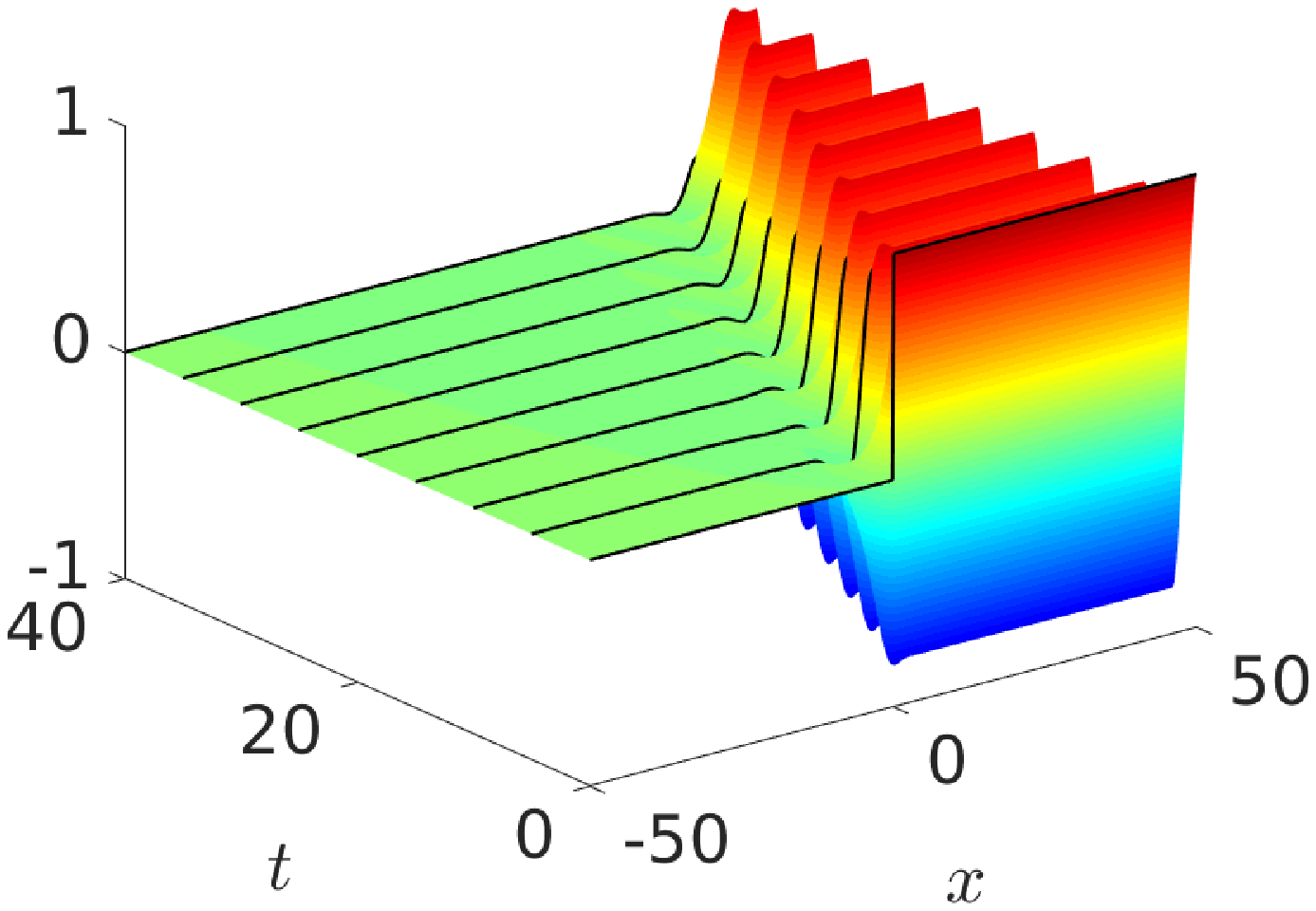}
\end{minipage}
\begin{minipage}[t]{0.45\textwidth}
\centering
\includegraphics[scale=0.45]{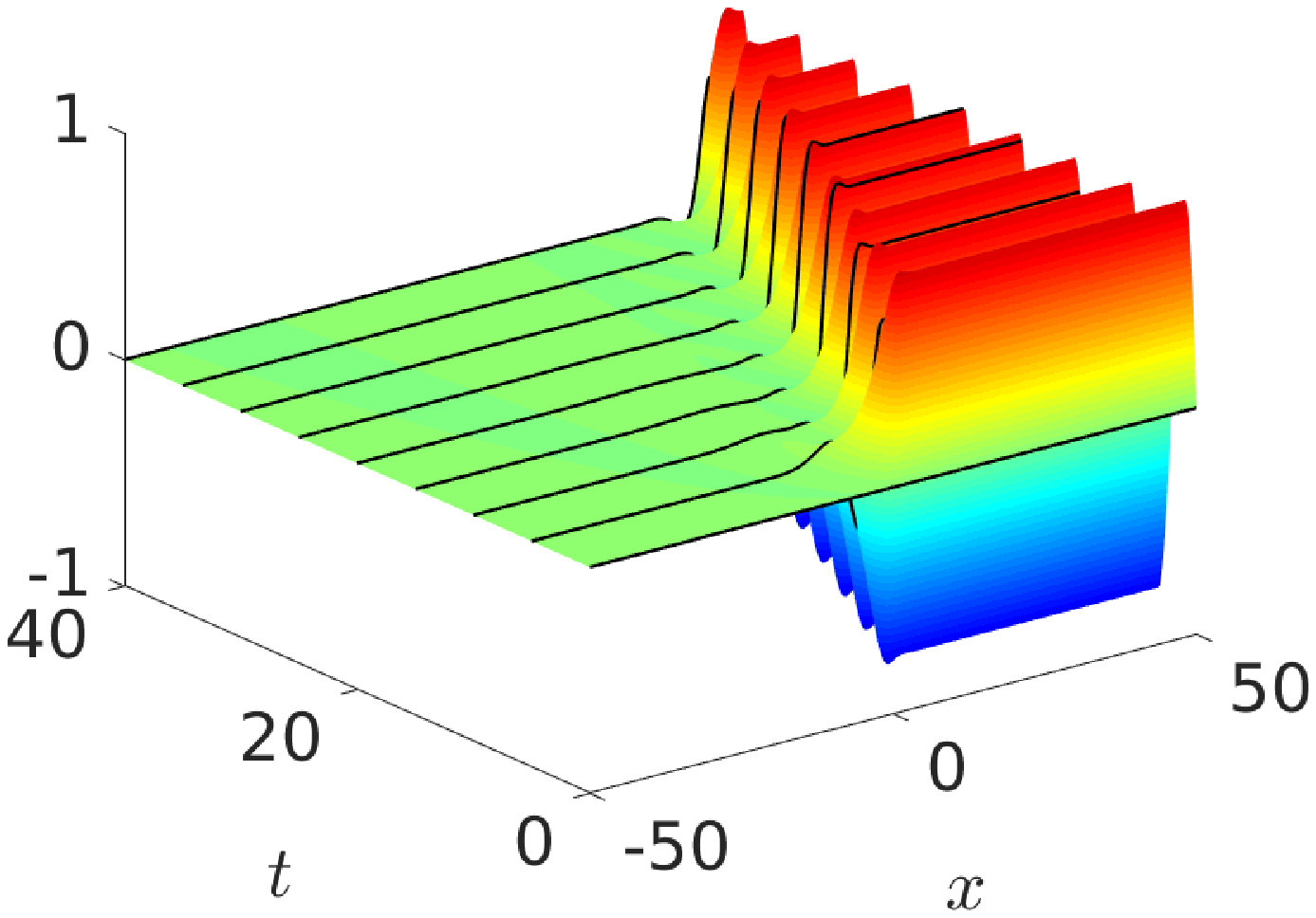}
\end{minipage}

\begin{minipage}[t]{0.45\textwidth}
\centering
\includegraphics[scale=0.45]{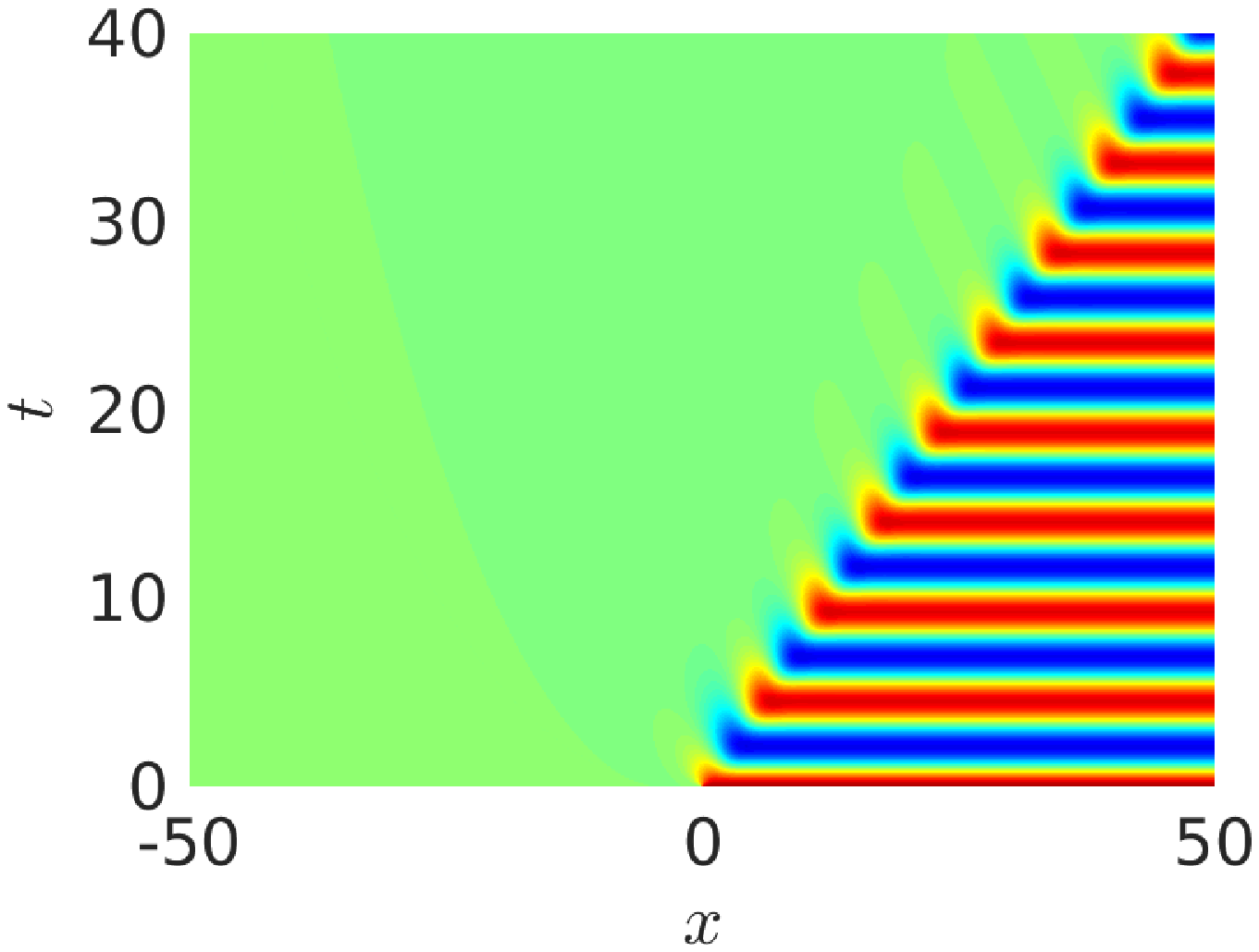}
\end{minipage}
\begin{minipage}[t]{0.45\textwidth}
\centering
\includegraphics[scale=0.45]{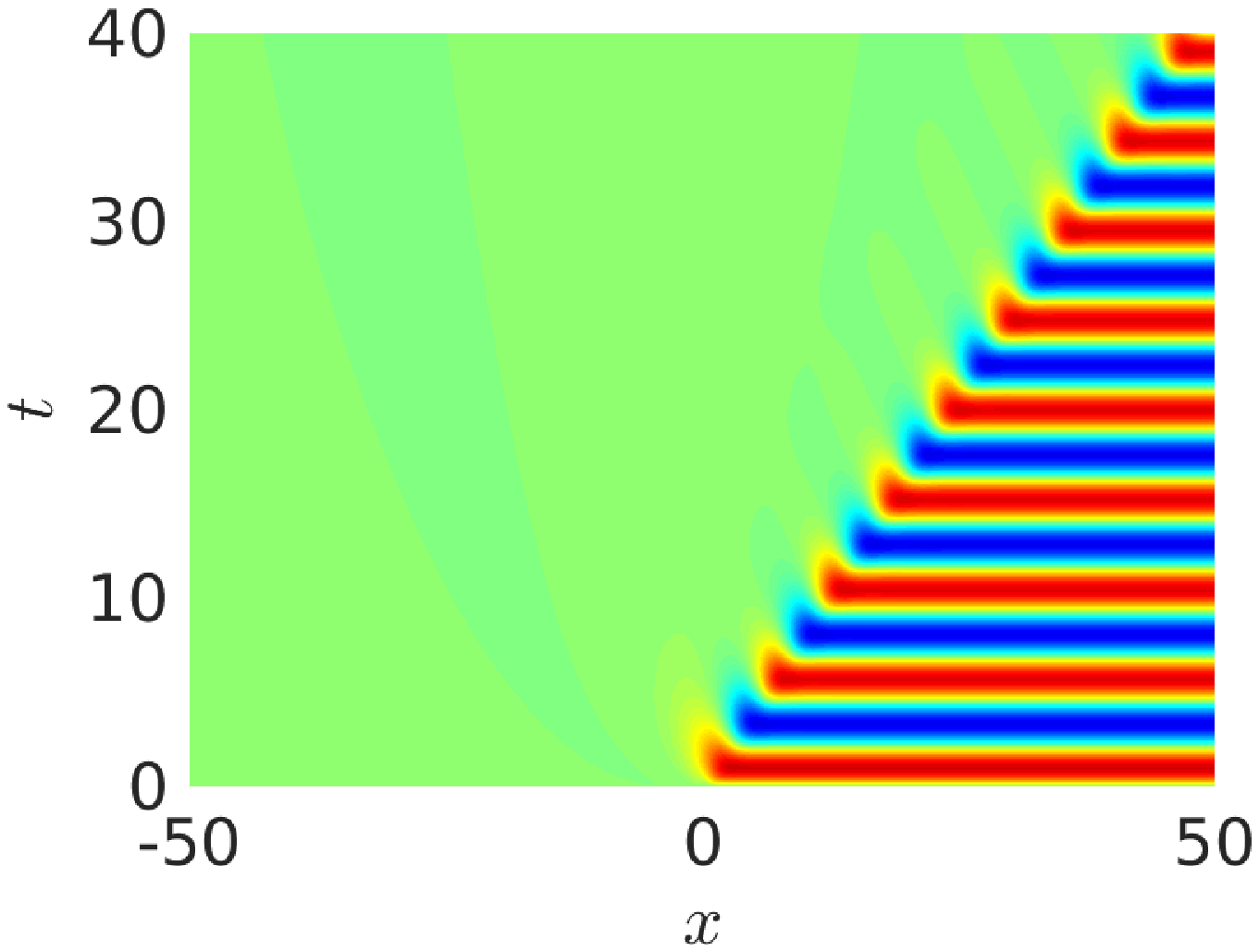}
\end{minipage}
\caption{Numerical simulation of a TOF in \eqref{QCGL} with parameters $\alpha = 1 + \tfrac{i}{2}$, $\beta_3 = 1 + i$, $\beta_5 = -1 + i$ and $\beta_1 = -0.1$. Real part (left) and imaginary part (right).} \label{ExampleTOF}
\end{figure}

We aim at sufficient conditions under which a TOF is
nonlinearly stable with asymptotic phase in suitable function spaces.
As initial perturbations we allow functions which can be decomposed
into an exponentially localized part and a front-like part which
perturbs the limit at $+ \infty$. There are two main difficulties
to overcome: first, the operator obtained by linearizing about the TOF has
essential spectrum touching the imaginary axis at zero in a quadratic
way. Second, the perturbation at infinity prevents the use of standard
Sobolev spaces for the linearized operator. The first difficulty will
be overcome by exponential weights which shift the essential spectrum
to the left, while the second difficulty is handled by analyzing stability
in an extended phase space which couples the dynamics on the real line
to the dynamics at $+ \infty$.

In the following we give a more technical outline of the setting
and our basic assumptions, and we provide an overview of the following
sections.
Our results will be stated for the two-dimensional real-valued system equivalent to \eqref{Evo}. Setting $U =u_1 + iu_2$, $u_j(x,t) \in \R$, $\alpha = \alpha_1 + i \alpha_2$, $\alpha_j \in \R$ and $G = g_1 + ig_2$ with $g_j:\R \rightarrow \R$ the equivalent real-valued parabolic system reads
\begin{align} \label{rEvo}
	u_t = Au_{xx} + f(u), \quad x \in \R,\, t \ge 0,
\end{align} 
where
\begin{equation} \label{SystemDef}
	A = 
	\begin{pmatrix}
		\alpha_1 & -\alpha_2 \\
		\alpha_2 & \alpha_1
	\end{pmatrix}, 
	\quad f(u) = g(|u|^2)u, \quad g(\cdot) = 
	\begin{pmatrix}
		g_1(\cdot) & -g_2(\cdot) \\
		g_2(\cdot) & g_1(\cdot)
	\end{pmatrix}.
\end{equation}
A TOF $U_{\star}=u_{\star,1}+ i u_{\star,2}$  of \eqref{Evo} then corresponds to
a solution $u_\star=(u_{\star,1},u_{\star,2})^{\top}$ of \eqref{rEvo} of the form
\begin{align*}
  u_\star(x,t) = R_{-\omega t} v_\star(x-ct),\quad
  R_{\theta}=\begin{pmatrix} \cos \theta & -\sin \theta \\
\sin \theta & \cos \theta \end{pmatrix},
\end{align*}
where $R_{\theta}$ denotes rotation by the angle $\theta \in \R$.  The profile
$v_\star: \R \rightarrow \R^2$ satisfies $V_{\star}=v_{\star,1}+i v_{\star,2}$ and
\begin{align} \label{rasymp}
	\lim_{\xi \rightarrow -\infty} v_\star (\xi) = 0, \quad \lim_{\xi \rightarrow +\infty} v_\star (\xi) = v_\infty,
\end{align}
where $V_{\infty}=v_{\infty,1}+ i v_{\infty,2}$ and $v_{\infty}\in \R^2$, $v_{\infty} \neq 0$. The vector $v_{\infty}$ is called the asymptotic rest-state
and the specific solution $R_{-\omega t}v_{\infty}$ of \eqref{rEvo}
is called a bound state.
Figure \ref{Figure:TOF3D}
shows a typical profile of a TOF in $(x,u_1,u_2)$-space.

\begin{figure}[h!]
\centering
\begin{minipage}[t]{0.9\textwidth}
\centering
\includegraphics[scale=1]{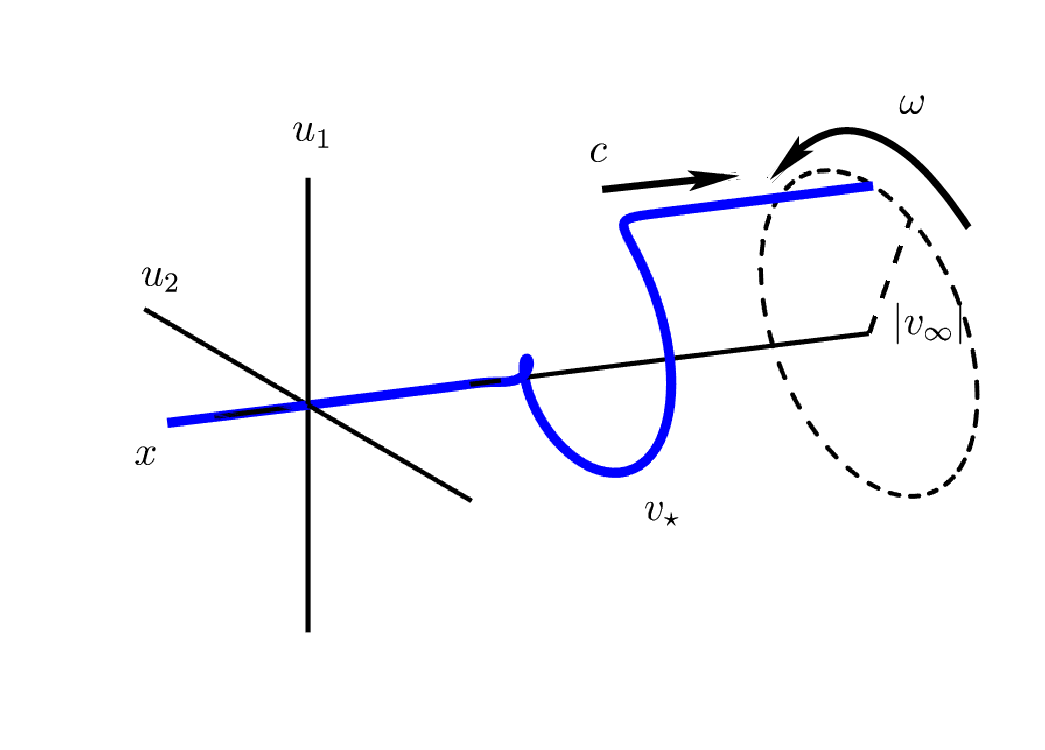}
\caption{Traveling oscillating front.} \label{Figure:TOF3D}
\end{minipage}
\end{figure}

For the stability analysis it is natural to transform \eqref{rEvo} into a co-moving and co-rotating frame, i.e. we set $u(x,t) = R_{-\omega t} v(\xi,t)$, $\xi = x-ct$ and find that
$v$ solves the equation
\begin{align} \label{comovsys}
  v_t = A v_{\xi\xi} + cv_\xi + S_\omega v + f(v), \quad \xi \in \R,\, t \ge 0, \quad
  S_{\omega} := \begin{pmatrix} 0 & -\omega \\ \omega & 0 \end{pmatrix}.
\end{align} 
 The time-independent profile $v_\star$ becomes a stationary solution of \eqref{comovsys}, i.e. it solves the ODE
\begin{align} \label{statcomovsys}
	0 = A v_{xx} + c v_x + S_\omega v + f(v), \quad x \in \R.
\end{align}
From the asymptotic property \eqref{rasymp} one concludes
(see Lemma \ref{lem:asym}) that the rest-state satisfies
\begin{equation*}
	g(|v_\infty|^2) = -S_\omega, \quad \lim_{x \rightarrow \pm \infty} v_{\star}'(x) = 0, \quad \lim_{x \rightarrow \pm \infty} v_{\star}''(x) = 0.
\end{equation*}

Since the right-hand side of  \eqref{comovsys} is equivariant with respect
to translations and multiplication by rotations, the TOFs always come in families, i.e. the system \eqref{comovsys} has a two-dimensional continuum
\begin{equation} \label{relequi}
  \mathcal{O}(v_\star) := \{ R_\theta v_\star(\cdot - \tau): (\theta,\tau) \in S^1 \times \R\}, \quad S^1=\R / 2 \pi \Z
\end{equation}
of stationary solutions. In the language of abstract evolution equations this
is a relative equilibrium, see e.g. \cite{Lauterbach}, \cite{Fiedleretal}, \cite{KapitulaPromislow}.
In the following we study the long time behavior of the solution $v$ of the initial-value problem
\begin{align} \label{perturbsys}
	v_t = Av_{xx} + cv_x + S_\omega v + f(v), \quad v(0) = v_\star + u_0,
\end{align}  
where  the initial perturbation $u_0$ is assumed to be small in a suitable sense.

Let us first state some basic assumptions on the system \eqref{rEvo}, \eqref{SystemDef}.

\begin{assumption} \label{A1}
The coefficient $\alpha$ and the function $g$ satisfy
\begin{align}
	& \alpha_1 > 0, \quad  g  \in C^3(\R,\R^{2,2}), \quad g_1(0)  < 0.  \label{A1a}
\end{align}
\end{assumption}
\begin{assumption} \label{A2}
  
    There exists a traveling oscillating front solution $u_\star$ of \eqref{rEvo} with profile $v_\star \in C^2_b (\R,\R^2)$, speed $c > 0$, frequency $\omega \in \R$ and asymptotic rest-state $v_\infty = (|v_\infty|,0)^\top \in \R^2$ such that
\begin{align*}
	g_1'(|v_\infty|^2) < 0.
\end{align*}
\end{assumption}
Note that the special form of $v_{\infty}$ selects just a specific equilibrium
from the orbit \eqref{relequi}.
Both assumptions guarantee that we have a stable equilibrium at
$\xi=-\infty$ and a circle of stable equilibria at $\xi=\infty$ when spatial
derivatives are ignored in \eqref{comovsys}.
Further conditions will
be imposed on the spectrum of the linearized operator
\begin{align} \label{LatL}
	Lv = Av_{xx} + cv_x + S_\omega v + Df(v_\star) v
\end{align}
in suitable function spaces.
In view of \eqref{relequi} we expect the linearization $L$ from \eqref{LatL}
to have a two-dimensional kernel. In addition, it turns out that the essential
spectrum of $L$,  touches the imaginary axis at the origin when considered in the function space $L^2(\R,\R^2)$. Thus there is no spectral gap between the
zero eigenvalue und the remaining spectrum, so that standard approaches to conclude nonlinear stability do not apply; see \cite{Henry},
\cite{KapitulaPromislow}, \cite{Sandstede02}.

We overcome this problem by two devices.
First, we impose the following condition 
\begin{assumption}(Spectral Condition) \label{A4}
The  diffusion coefficient $\alpha=\alpha_1+i \alpha_2$ satisfies
\begin{align*}
	\alpha_2 g_2'(|v_\infty|^2) + \alpha_1 g_1'(|v_\infty|^2) < 0.
\end{align*}
\end{assumption}
Note that Assumption \ref{A4} follows from Assumptions \ref{A1}, \ref{A2}
if $\alpha_2=0$. Moreover, we will show that Assumption \ref{A4} guarantees
 the essential spectrum to have negative quadratic contact
with the imaginay axis.
Second, we use Lebesgue and Sobolev spaces with exponential weight
\begin{align} \label{eta}
	\eta(x) = e^{\mu \sqrt{x^2 + 1}}, \quad \mu \ge 0.
\end{align}
Any sufficiently small  $\mu > 0$ will be enough to shift the essential
spectrum to the left and allow for a stability result.
Weights of this or similar type frequently appear in stability analyses,
see e.g. \cite{Zelik}, \cite[Ch.3.1.1]{KapitulaPromislow}, \cite{ghazaryan},
\cite{Kapitula94}. We note that the stability statement w.r.t. the $L^{\infty}$-norm  in the
second part of \cite[Theorem 7.2]{Kapitula94} comes closest to our
results. There a perturbation argument for the case of a positive definite
matrix $A$ in \eqref{comovsys} is employed and  the resulting Evans
function \cite{Alexander} is analyzed. This leads to more restrictive
conditions on the coefficients of the system and on the initial data.

We finish the introduction with a brief outline of the contents of
the following sections. In section \ref{sec2} we complete the basic
assumptions \ref{A1}, \ref{A2}, \ref{A4} by eigenvalue conditions
for the operator \eqref{LatL} and we state our main results in more
technical terms. The approach  of the profile towards
its rest states is shown to be exponential, and stability with asymptotic phase
is stated in weighted $H^1$-spaces. We also explain the main idea of
the proof which incorporates the dynamics of \eqref{comovsys} at $\xi = \infty$
into an extended evolutionary system, see \eqref{CP}, \eqref{F}.
In Section \ref{sec3} we discuss in detail
the Fredholm properties of the operator $L$ and its extended version
in weighted spaces and we derive resolvent estimates.
These form the basis for obtaining detailed estimates of the associated
(extended) semigroup in Section \ref{sec4}. The subsequent section \ref{sec5}
is devoted to the decomposition of the dynamics into the motion within
the underlying two-dimensional symmetry group and within a
codimension-two function space. Section \ref{sec6} then provides sharp
estimates for the resulting remainder terms. Then a local existence theorem and
a Gronwall estimate complete the proof in Section \ref{sec7}.

Let us finally mention that the techniques of this paper can be used
to prove that the general method of freezing
(\cite{BeynThummler04}, \cite{Rottmann-Matthes12}, \cite{BOR14})
works successfully for the parabolic equation \eqref{rEvo} with two underlying symmetries, see \cite{Doeding}.

\sect{Assumptions and main results}
\label{sec2}
As a preparation for the subsequent analysis we specify the approach of
a TOF towards its rest states.
 \begin{lemma} \label{lem:asym}
 Let $v_\star \in C^2_b(\R,\R^2)$ be the profile of a traveling oscillating front of \eqref{rEvo} with speed $c>0$, frequency $\omega \in \R$ and asymptotic rest-state $v_\infty \in \R^2\backslash \{ 0 \}$. Moreover, suppose $\Re \alpha > 0$ and $g  \in C(\R,\R^{2,2})$. Then the following holds:
 \begin{align*}
	g(|v_\infty|^2) = -S_\omega, \quad \lim_{x \rightarrow \pm \infty} v_{\star}'(x) = 0, \quad \lim_{x \rightarrow \pm \infty} v_{\star}''(x) = 0.
 \end{align*}
 \end{lemma}
Using Assumption \ref{A1} and \ref{A2} one can conclude that the convergence in Lemma \ref{lem:asym} of the profile $v_\star$ and its derivatives is exponentially fast.
\begin{theorem} \label{decay}
  Let Assumption \ref{A1} and \ref{A2} be satisfied and let $v_{\star}$ be given
  as in Lemma \ref{lem:asym}. Then $v_\star \in C^5_b(\R,\R^2)$ holds and
  there are constants $K,\mu_\star > 0$ such that
\begin{align*}
	| v_\star(x) - v_\infty | + |v'_\star(x)| + |v''_\star(x)| + |v'''_\star(x)| & \le K e^{-\mu_\star x} \quad \forall\, x \ge 0, \\
	| v_ \star(x) | + |v'_\star(x)| + |v''_\star(x)| + |v'''_\star(x)| & \le K e^{\mu_\star x} \quad \forall\, x \le 0.
\end{align*}
\end{theorem}
In the Appendix we give the proof of Lemma \ref{lem:asym} and the main steps of
the proof of Theorem \ref{decay}.

With the weight $\eta$ given by \eqref{eta}, let us introduce the weighted $L^2$ space
\begin{align*}
  L^2_\eta(\R,\R^n) := \{ v \in L^2(\R,\R^n): \eta v \in L^2(\R,\R^n) \},
  \quad (u,v)_{L^2_{\eta}}:= (\eta u, \eta v)_{L^2}
\end{align*}
and the associated weighted Sobolev spaces defined for $\ell \in \N$ by
\begin{align*} 
	H^\ell_\eta(\R,\R^n)  &:= \{ v \in L^2_\eta(\R,\R^n) \cap H^\ell_{\mathrm{loc}}(\R,\R^n): \partial^k v \in L^2_\eta(\R,\R^n),\, 1 \le k \le \ell \}, \\
	\| v \|_{H^\ell_\eta}^2& := \sum_{k=0}^\ell \| \partial^k v \|_{L^2_\eta}^2.
\end{align*}

Let us note that Theorem \ref{decay} ensures $v_{\star}^{(j)} \in H^{3-j}_{\eta}(\R,\R^2)$
for   $0 \le \mu < \mu_{\star}$ and $j=1,2,3$.
However, 
the profile $v_\star$ of a TOF does not decay to zero as $x \rightarrow \infty$,
and, moreover, we expect
the limit $\rho(t)=\lim_{x \rightarrow \infty}v(x,t)$ of a solution of \eqref{comovsys} to still move with time.
Therefore, the idea is to include an ODE for the dynamics of
$\rho(t)$ into the overall system.
Formally taking the limit $x \rightarrow \infty$ in \eqref{comovsys} and
assuming $v_x(x,t), v_{xx}(x,t) \rightarrow 0$ as $x \rightarrow \infty$ we obtain for $\rho$  the ODE
\begin{align} \label{phaseODE}
	\rho'(t) = S_\omega\rho(t) + f(\rho(t)).
\end{align}
Note that $v_\infty$ is a stationary solution of \eqref{phaseODE} due to
Lemma \ref{lem:asym}, and, by equivariance, there is a whole circle of equilibria
$\{R_{-\theta}v_{\infty}: \theta \in S^1\}$. Next we choose a template function
\begin{align*}
  \hat{v}(x) := \tfrac{1}{2} \tanh(\hat{\mu}x) + \tfrac{1}{2},
  \quad 0< 2 \hat{\mu} \le \mu_{\star}.
\end{align*}
The rate $\hat{\mu}$ has been chosen such that the approach toward the limits as
$x \to \pm \infty$ is weaker than for the derivatives of the solution in Theorem \ref{decay}. Such a choice is not strictly necessary but will avoid some
technicalities in the following.
If $0<\mu < 2\hat{\mu}$ we conclude $v_{\star}- \hat{v} v_{\infty} \in H^2_{\eta}(\R,\R^2)$ and we
also expect the solution $v$ of \eqref{comovsys}
to satisfy $v(\cdot,t) - \hat{v} \rho(t) \in H^2_\eta(\R,\R^2)$, i.e. to lie in an
affine linear space with a time dependent offset given by $\rho$. 
Therefore,  we introduce the Hilbert space
\begin{align*}
  X_\eta := \Big\{ (v,\rho)^\top:\, v: \R \rightarrow \R^2,\, \rho \in \R^2,\, v-\rho\hat{v} \in L^2_\eta(\R,\R^2) \Big\}  
\end{align*}
with inner product $\big( (u,\rho)^{\top},
  (v,\zeta)^{\top} \big)_{X_{\eta}}=(\rho,\zeta)+(u-\rho \hat{v},v-\zeta \hat{v})_{L^2_{\eta}}$. Similarly, we define the smooth analog
\begin{align*}
	X^\ell_\eta := \left\{ (v,\rho)^\top \in X_\eta: v \in H^\ell_{\mathrm{loc}},\partial^k v \in L^2_\eta,\, 1 \le k \le \ell \right\}, \quad \ell \in \N_0
\end{align*}
with the norm	given by
$\left\| (v,\rho)^\top \right\|_{X^\ell_\eta}^2 := |\rho|^2 + \| v - \rho \hat{v}\|_{L^2_\eta}^2 + \sum_{k=1}^{\ell} \|\partial^k v\|_{L^2_\eta}^2$. We further set
$Y_\eta := X^2_\eta$ and denote the elements of $X^{\ell}_\eta$ by bold letters, for example,
\begin{align*} \v = (v,\rho)^\top , \quad \v_\star = (v_\star,v_\infty)^\top,
  \quad \v_0= (u_0,\rho_0)^{\top}.
\end{align*}
As noted above, Theorem \ref{decay} implies $v_\star \in v_\infty \hat{v} + H^2_\eta$ and thus $\v_\star \in Y_\eta$.
Instead of \eqref{perturbsys}, we consider the extended Cauchy problem on
$X_\eta$ 
\begin{align} \label{CP}
	\v_t = \F(\v), \quad \v(0) = \v_\star + \v_0,
\end{align}
where $\F$ is a semilinear operator given by
\begin{align} \label{F}
  \F : Y_\eta \rightarrow X_\eta, \quad \vek{v}{\rho} = \v \mapsto \F (\v) =
  \vek{Av_{xx} + cv_x + S_\omega v + f(v)}{S_\omega \rho + f(\rho)}.
\end{align}
With these settings, $\v_{\star}$ becomes a stationary solution of
\eqref{CP}, and our task is to prove its nonlinear stability with asymptotic
phase. For this purpose, let us extend the group action induced by rotation
and translation of elements from $L^2_{\eta}$ to $X_\eta$ as follows:
\begin{align} \label{groupaction}
  a(\gamma): X_\eta \rightarrow X_\eta, \quad \vek{v}{\rho} \mapsto a(\gamma)\vek{v}{\rho} = \vek{R_{-\theta} v(\cdot - \tau)}{R_{-\theta} \rho},\quad
  \gamma =(\theta,\tau) \in \G := S^1 \times \R.
\end{align}
The operator $\F$ from \eqref{F} is then equivariant w.r.t. the group action,
i.e. $\F(a(\gamma) \v) = a(\gamma) \F(\v)$ for all $\gamma \in \G$ and
$u \in Y_{\eta}$. Further a metric on $\G$ is given by
\begin{align*}
  d_{\G} (\gamma_1,\gamma_2) = |\gamma_1 - \gamma_2|_{\G}, \quad |\gamma|_{\G} := \min_{k \in \Z} |\theta - 2\pi k| + |\tau|, \quad \gamma = (\theta,\tau)
  \in \G.
\end{align*}

Finally,  we collect the assumptions on the linearized operator
$L: \D(L)=H^2 \subset L^2 \rightarrow L^2$ from
\eqref{LatL}. The operator $L$  will turn out to be
  closed and  densely defined. We denote its resolvent set by
\begin{align*}
	\mathrm{res}(L) := \{ s \in \C: sI-L: \D(L) \rightarrow L^2 \text{ is bijective} \} 
\end{align*}
and its spectrum by $\sigma(L) = \C \backslash  \mathrm{res}(L)$. 
The further subdivision of the spectrum into the essential spectrum and the point spectrum varies in the literature (see the five different notions in \cite{EdmundsEvans}).
We use the following definition (see  $\sigma_{\mathrm{e},4}(L)$ in \cite[Ch.I.4,IX.1]{EdmundsEvans} or \cite[Ch.3]{KapitulaPromislow} and note the slight
deviation from \cite{Henry},\cite{Kato}):
\begin{align} \label{eq2:defessential}
	\sigma_{\mathrm{pt}}(L) := \{ s \in \sigma(L): sI - L \text{ is Fredholm of index } 0 \}, \quad
  \sigma_{\mathrm{ess}}(L) := \sigma(L) \backslash \sigma_{\mathrm{pt}}(L). 
\end{align}
When we insert the translates $v_{\star}(\cdot- \tau)$ from \eqref{relequi}
into the stationary
equation \eqref{comovsys} and differentiate with respect to $\tau$ we
obtain that the nullspace $\ker(L)$ of $L$ contains at least
$v'_{\star}\in H^2$.
The following condition requires that there are no 
(generalized) eigenfunctions and that  eigenvalues from the point
spectrum lie strictly to the left of the imaginary axis.
\begin{assumption}[Eigenvalue Condition] \label{A5} $ $ \\
    There exists $\beta_E > 0$ such that $\Re s < - \beta_E$ holds
  for all $s \in \sigma_{\mathrm{pt}}(L)$. Moreover,
  \begin{align}
    \label{eq3:zerosimple}
	\dim \ker (L) = \dim \ker (L^2)=1.
\end{align}
\end{assumption}
Recall $v'_{\star}\in \ker(L)$ so that \eqref{eq3:zerosimple} implies
$\ker(L) = \mathrm{span}\{ v'_{\star} \}$. In Theorem \ref{thm4.17} below we will
see that $L:H^2 \to L^2 $ is not Fredholm, hence $0$ belongs to
$\sigma_{\mathrm{ess}}(L)$ and not to $\sigma_{\mathrm{pt}}(L)$.
For this reason we wrote condition \eqref{eq3:zerosimple} explicitly in
terms of nullspaces, and $\Re s < -\beta_E$ for $s\in \sigma_{\mathrm{pt}}(L)$
is no contradiction for $s=0$.

Differentiating the equation \eqref{comovsys} for the stationary continuum
\eqref{relequi} with respect to the first group variable $\theta\in S^1$
produces a second `eigenfunction' $ S_1 v_{\star}$ which, however, does
not belong to $\D(L)=H^2$. But this eigenfunction will appear for the 
extended operator obtained by linearizing $\F$ from
\eqref{F} at $\v_{\star}$:
\begin{align} \label{calL}
	\L_\eta : Y_\eta \rightarrow X_\eta, \quad \vek{v}{\rho} \mapsto \L_\eta \vek{v}{\rho} = \vek{Av_{xx} + cv_x + S_\omega v + Df(v_\star) v}{S_\omega \rho + Df(v_\infty)\rho}.
\end{align}
The subindex $\eta$ indicates that the operator $\L_{\eta}$ depends
on the weight through its domain and range. We further write $\L=\L_1$ in case $\mu = 0$, $\eta \equiv 1$ and introduce
\begin{align*}
	E_\omega := S_\omega + Df(v_\infty)
\end{align*} 
for the second component of the operator.
 In Section \ref{sec3} we prove the following result for the point spectrum of the operator $\L_{\eta}$ defined in \eqref{calL}.
\begin{lemma} \label{lemma4.18}
  Let Assumption \ref{A1}, \ref{A2}, \ref{A4} and \ref{A5} be satisfied. Then there exists a constant $\mu_1\in (0, 2 \hat{\mu})$
  such that the following holds for all weight functions \eqref{eta} with 
  $0 < \mu \le \mu_1$ :
  \begin{enumerate}[i)]
    \item  The eigenvalue $0$ belongs to $\sigma_{\mathrm{pt}}(\L_\eta)$ and has 
    geometric and algebraic multiplicity $2$, more precisely, 
\begin{align} \label{eq2:eigenfunctions}
  	\ker(\L_\eta^2)= \ker (\L_\eta) =&\, \mathrm{span} \{\varphi_1, \varphi_2 \}  , \quad
\varphi_1 =  (v'_{\star}, 0)^{\top}, \quad\varphi_2= (S_1 v_\star, S_1 v_\infty)^{\top}.
\end{align}
\item There exists some $\beta_1=\beta_1(\mu)>0$ such that all eigenvalues $s \in \sigma_{\mathrm{pt}}(\L_\eta) \setminus \{0\}$
    satisfy $\Re s < - \beta_1<0$.
  \end{enumerate}
  \end{lemma}

Now we are in a position to formulate the main result.
\begin{theorem} \label{Theorem4.10}
  Let Assumption \ref{A1}, \ref{A2}, \ref{A4} and \ref{A5} be satisfied and let
  $\eta$ be given by \eqref{eta}. Then there exists $\mu_0 > 0$ such that for
  every $\mu \in (0,\mu_0)$ there are constants
  $\varepsilon_0(\mu), \beta(\mu), K(\mu), C_\infty(\mu) > 0$ so that the
  following statements hold. For all initial perturbations $\v_0 \in Y_\eta$
  with $\| \v_0 \|_{X_\eta^1} < \varepsilon_0$  the equation \eqref{CP} has a
  unique global solution
  $\v \in C((0,\infty), Y_\eta) \cap C^1([0,\infty),X_\eta)$ which can be
    represented as
\begin{align*}
	\v(t) = a(\gamma(t)) \v_\star + \w(t), \quad t \in [0,\infty)
\end{align*}
for suitable functions $\gamma \in C^1([0,\infty),\G)$ and $\w \in C((0,\infty),Y_\eta) \cap C^1([0,\infty), X_\eta)$. Further, there exists
an asymptotic phase $\gamma_\infty = \gamma_\infty(\v_0) \in \G$ such that
\begin{align*}
	\| \w(t) \|_{X^1_\eta} + |\gamma(t) - \gamma_\infty|_\G & \le K e^{-\beta t} \| \v_0 \|_{X^1_\eta}, \quad 	|\gamma_\infty|_\G \le C_\infty \| \v_0 \|_{X^1_\eta}.
\end{align*}
\end{theorem}
This leads to corresponding stability statements for a TOF of the equations
\eqref{perturbsys} and \eqref{rEvo}, respectively. For simplicity, we state
the result in an informal way  under the
assumptions of Theorem \ref{Theorem4.10} for the extended version of \eqref{rEvo},
i.e.
\begin{equation} \label{evoextended}
\begin{aligned}
  u_t & = Au_{xx} + f(u), \quad u(\cdot,0)= v_{\star} + u_0,\\
  r_t & = f(r), \quad r(0)= v_{\infty}+ \rho_0.
\end{aligned}
\end{equation}
Initial perturbations $\v_0=(u_0,\rho_0)$ are assumed to be small in the sense that
\begin{align*}
  \| \v_0\|_{X^1_{\eta}}^2 &= \|u_0- \rho_0 \hat{v}\|_{L^2_{\eta}}^2 + \|\partial_x u_0 \|
  _{L^2_{\eta}}^2 + |\rho_0|^2 \le \varepsilon_0^2.
\end{align*}
Then the system \eqref{evoextended} has a unique solution $\u = (u,r)
\in C((0,\infty), Y_\eta) \cap C^1([0,\infty),X_\eta)$ and there exist
  functions $(\theta,\tau)\in C^1([0,\infty), S^1 \times \R)$ and
    a value $(\theta_{\infty},\tau_{\infty})\in S^1 \times \R$ such that
    for all $t \ge 0$
    \begin{equation*}
    \begin{aligned}
      &  \|u(\cdot,t)-R_{-\omega t -\theta(t)}v_{\star}(\cdot - c t -\tau(t))
      - \hat{v}(r(t)-R_{-\omega t -\theta(t)}v_{\infty})\|_{L^2_{\eta}}
      +\|\partial_x u(\cdot,t)\|_{L^2_{\eta}}\\
      &+ |r(t)-R_{-\omega t -\theta(t)}v_{\infty}| +|\tau(t)-\tau_{\infty}|+ |\theta(t)-\theta_{\infty}| \\
     & \le  K e^{-\beta t} \| \v_0 \|_{X^1_\eta}.
    \end{aligned}
    \end{equation*}
    Note the detailed expression for the asymptotic behavior of $\lim_{x \to \infty}u(x,t)$ as $t \to \infty$.

\sect{Spectral analysis of the linearized operator}
\label{sec3}
In this section we study the spectrum of the linear operator $\L_\eta$ from
\eqref{calL} and estimate solutions of the resolvent equation
\begin{align} \label{resolvGl}
	(sI - \L_\eta)\v = \r, \quad s \in \C,\, \r= (r,\zeta)^{\top} \in X_\eta.
\end{align}
In the first step we derive resolvent estimates for solutions
$\v \in Y_\eta$ of \eqref{resolvGl} when $|s|$ is large and $s$ lies in the
exterior of some sector opening to the left. The approach is based on
energy estimates from \cite{Kreiss}, \cite{KreissLorenz}.
\begin{lemma} \label{aprioriest}
  Let Assumption \ref{A1} and \ref{A2} be satisfied and let $\mu_2 \in (0,2\hat{\mu})$.
  Then there exist
  constants $\varepsilon_0, R_0, C> 0$ such that the following properties
  hold for all $0 \le \mu \le \mu_2$ . The operator $\L_\eta:Y_\eta \subset X_\eta \rightarrow X_\eta$ is  closed and densely defined in $X_{\eta}$.  For all
\begin{align} \label{Omega0}
	s \in \Omega_0 := \left\{ s \in \C: |s| \ge R_0, |\arg(s)| \le \frac{\pi}{2} + \varepsilon_0 \right\}
\end{align}
the equation \eqref{resolvGl} with $\v \in Y_\eta$ and $\r \in X_\eta$ implies
\begin{align}
	& |s| \| \v \|_{X_\eta}^2 + \| v_x \|_{L_\eta^2}^2 \le \frac{C}{|s|} \| \r \|_{X_\eta}^2 \label{resolest1} \\
	& |s|^2 \| \v \|_{X_\eta}^2 + |s| \| v_x \|_{L_\eta^2}^2 + \|v_{xx}\|_{L_\eta^2}^2 \le C \| \r \|_{X_\eta}^2. \label{resolest2}
\end{align}
In addition, if $\r \in X_\eta^1$ and $\v \in X^3_\eta$ then
\begin{align} \label{resolest3}
	|s|^2 \| \v \|_{X^1_\eta}^2 + |s| \| v_{xx} \|_{L_\eta^2}^2 + \|v_{xxx}\|_{L_\eta^2}^2 \le C \| \r \|_{X^1_\eta}^2.
\end{align}
\end{lemma}

\begin{proof}
  For the proof let us abbreviate $C_{\star} := Df(v_\star)$, $C_\infty := Df(v_\infty)$ and 
  $(\cdot,\cdot) = (\cdot,\cdot)_{L_\eta^2(\R,\R^2)}$.
  From Theorem \ref{decay} and $\mu \le \mu_2 < \mu_{\star}$ we find for $(v,\rho)^{\top} \in X_{\eta}$ 
  \begin{equation} \label{eq3:basic}
    \begin{aligned}
    \|C_{\star}v - C_{\infty}\rho \hat{v}\|_{L^2_{\eta}} \le &\, \|C_{\star}\|_{L^{\infty}}
    \|v - \rho \hat{v}\|_{L^2_{\eta}}+
    \|(C_{\star}-C_{\infty})\rho \hat{v}\|_{L^2_{\eta}} \\
    \|(C_{\star}-C_{\infty})\rho \hat{v}\|_{L^2_{\eta}}^2 = & \,
    \|(C_{\star}-C_{\infty})\rho \hat{v}\|_{L^2_{\eta}(\R_-)}^2 +
    \|(C_{\star}-C_{\infty})\rho \hat{v}\|_{L^2_{\eta}(\R_+)}^2 \\
    \le &\, \frac{ e^{2\mu} \|C_{\star}- C_{\infty}\|_{L^{\infty}}^2}{2\hat{\mu}- \mu} |\rho|^2
    +\int_0^{\infty}\eta^2(x) |Df(v_{\star}(x))-Df(v_{\infty})|^2 dx |\rho|^2
    \\
    \le & \, K_C |\rho|^2.
    \end{aligned}
    \end{equation}
  From this one infers that the operator $\L_\eta: Y_\eta \rightarrow X_\eta$ is bounded.
  Next, we note that \eqref{resolest2} implies the closedness of $\L_\eta$. For this purpose, let $\{ \v_n \}_{n \in \N} \subset Y_\eta$ be given with
  $\v_n \rightarrow \v$ in $X_\eta$ and $\L_\eta \v_n \rightarrow \w$ in $X_\eta$.
  Pick $s_0 \in \Omega_0$ with $|s_0| \ge 1$. Then \eqref{resolest2} yields
\begin{align*}
	\| \v_n - \v_m \|_{Y_\eta}^2 & \le |s_0|^2 \| \v_n - \v_m \|_{X_\eta}^2 + |s_0| \| v_{n,x} - v_{m,x} \|_{L^2_\eta}^2 + \| v_{n,xx} - v_{m,xx} \|_{L^2_\eta}^2 \\
	& \le C_1 \| s_0( \v_n - \v_m) - \L_\eta\v_n - \L_\eta \v_m \|_{X_\eta}^2 \rightarrow 0 ,\quad n,m \rightarrow \infty.
\end{align*}
Thus, $\{ \v_n \}_{n \in \N}$ is a Cauchy sequence in $Y_\eta$ and there is $\tilde{\v} \in Y_\eta$ with $\v_n \rightarrow \tilde{\v}$ in $Y_\eta$. We conclude $\v = \tilde{\v} \in Y_\eta$ and $\v_n \rightarrow \v$ in $Y_\eta$.
Finally, $\L_\eta \v = \w$ follows from the boundedness of $\L_\eta: Y_\eta \rightarrow X_\eta$ and the estimate
\begin{align*}
	\| \L_\eta \v - \w \|_{X_\eta} & \le \| \L_\eta(\v - \v_n) \|_{X_\eta} + \| \L_\eta \v_n - \w \|_{X_\eta} \rightarrow 0, \quad n \rightarrow \infty.
\end{align*} 
The estimate \eqref{resolest3} follows by differentiating \eqref{resolvGl} w.r.t. $x$ and using \eqref{resolest2}. Therefore, it is left to show \eqref{resolest1} and \eqref{resolest2}. We begin with \eqref{resolest1}. For this purpose,
let $s \in \Omega_0$ with $R_0$ and $\varepsilon_0$ still to be determined.
Take the inner product of \eqref{resolvGl} with $\v$ in $X_\eta$ to obtain
\begin{align*}
	& (\v,\r)_{X_\eta} = (\v, (sI-\L_\eta)\v)_{X_\eta} = \left(\vek{v}{\rho}, \vek{sv - Av_{xx} - cv_x - S_\omega v - C_{\star}v}{s \rho - S_\omega \rho - C_\infty \rho}\right)_{X_\eta} \\
	& = \rho^\top (sI - S_\omega- C_\infty) \rho + (v - \rho \hat{v}, sv - Av_{xx} - cv_x - S_\omega v - C_{\star}v - (s \rho - S_\omega \rho - C_\infty \rho) \hat{v}) \\
	& = s \| \v \|_{X_\eta}^2 - \rho^\top S_{\omega}\rho - \rho^\top C_\infty \rho \\
	& \quad - (v- \rho\hat{v}, Av_{xx})_{L_\eta^2} - c(v - \rho\hat{v}, v_x) - (v - \rho\hat{v}, S_\omega (v - \rho\hat{v})) - (v-\rho\hat{v},C_{\star}v- C_\infty \rho\hat{v}).
\end{align*}
Integration by parts yields
\begin{align}
  \label{equ:proof1}
\begin{split}
  & s \| \v \|_{X_\eta}^2 + (v_x-\rho\hat{v}_x,Av_x)_{L_\eta^2} \\
  &  = \rho^\top (S_\omega + C_\infty) \rho -2(\eta'\eta^{-1}(v - \rho \hat{v}), Av_x) + c(v-\rho \hat{v}, v_x)_{L_\eta^2} \\
  & + (v-\rho\hat{v},S_\omega(v -\rho \hat{v}))_{L_\eta^2}  + (v - \rho \hat{v},C_{\star}v -C_\infty \rho\hat{v})_{L_\eta^2} + ( \v , \r)_{X_\eta}.
\end{split}
\end{align}
Further, we use Cauchy-Schwarz and Young's inequality with arbitrary
$\varepsilon_i >0$ and  \eqref{eq3:basic}  to obtain the estimates
\begin{align} \label{equ:proof1a}
\begin{split}
	|(v_x-\rho \hat{v}_x, Av_x)|	& \le | A | \left( \| v_x\|_{L^2_\eta}^2 + \frac{1}{4 \varepsilon_1} \| \rho \hat{v}_x \|_{L^2_\eta}^2 + \varepsilon_1 \| v_x \|_{L^2_\eta}^2 \right) \\
	& = | A | (1+\varepsilon_1) \| v_x \|_{L^2_\eta}^2 + \frac{ e^{2\mu} | A | }{4(2\hat{\mu}-\mu) \varepsilon_1} |\rho|^2,
\end{split}
\end{align}
\begin{align} \label{equ:proof1aa}
\begin{split}
	|(\eta' \eta^{-1}(v-\rho \hat{v}), Av_x)| & \le \frac{ \mu^2 | A |}{4 \varepsilon_2} \| v - \rho \hat{v} \|_{L^2_\eta}^2 + \varepsilon_2 | A |  \| v_x \|_{L^2_\eta}^2,
\end{split}
\end{align}
\begin{align} \label{equ:proof1b}
\begin{split}
	|c(v-\rho \hat{v}, v_x)| \le \frac{|c|}{4 \varepsilon_3} \| v-\rho\hat{v} \|_{L^2_\eta}^2 + |c| \varepsilon_3 \| v_x \|_{L^2_\eta}^2,
\end{split}
\end{align}
\begin{align} \label{equ:proof1c}
	|(v-\rho \hat{v}, S_\omega (v-\rho \hat{v}))| \le |\omega| \| v-\rho \hat{v}\|_{L^2_\eta}^2,
\end{align}
\begin{align} \label{equ:proof1d}
\begin{split}
	|(v-\rho \hat{v}, C_{\star}v-C_\infty \rho \hat{v})| 
	& \le   \left( \| C_{\star}\|_{L^\infty}+ \frac{1}{4 \varepsilon_4} \right) \| v-\rho \hat{v} \|_{L^2_\eta}^2 + K_C \varepsilon_4 |\rho|^2.
\end{split}
\end{align}
Take the absolute value in \eqref{equ:proof1} and use \eqref{equ:proof1a}-\eqref{equ:proof1d} with $\varepsilon_i = 1$ to obtain for some $K_0,K_1 > 0$
\begin{align}
	\label{equ:proof2}
	|s| \| \v \|_{X_\eta}^2 \le K_0 \|v_x\|_{L_\eta^2}^2 + K_1 \| \v \|_{X_\eta}^2 + \| \v \|_{X_\eta} \| \r \|_{X_\eta}.
\end{align}
Next we note that $( v_x - \rho \hat{v}_x , Av_x) = \alpha_1 \| v_x \|_{L^2_\eta}^2 - (\rho \hat{v}_x, Av_x)$ and
\begin{align} \label{equ:proof2a}
	|(\rho \hat{v}_x,Av_x)| \le | A | \| \hat{v}_x\|_{L^2_\eta} |\rho| \| v_x \|_{L^2_\eta} \le \frac{ e^{2 \mu} | A |}{(2\hat{\mu}-\mu) \varepsilon_5}|\rho|^2 + \varepsilon_5 | A| \| v_x \|_{L^2_\eta}^2.
\end{align}
Taking the real part in \eqref{equ:proof1} we obtain by using Cauchy-Schwarz,
Young's inequality and \eqref{equ:proof1aa}-\eqref{equ:proof1d} as well as
\eqref{equ:proof2a} with $\varepsilon_2 = \varepsilon_5 = \frac{\alpha_1}{8| A|}$, $\varepsilon_3 = \frac{\alpha_1}{4 |c|}$, $\varepsilon_4 = 1$ the estimate 
\begin{align*}
	\Re s  \| \v \|_{X_\eta}^2 + \alpha_1 \| v_x \|_{L_\eta^2}^2 & \le \big(\varepsilon_5 | A | + \varepsilon_2 | A | + \varepsilon_3 |c|\big) \| v_x \|_{L^2_\eta}^2 + K_2 \| \v \|_{X_\eta}^2 + \| \v \|_{X_\eta} \| \r \|_{X_\eta}  \\
	& \le \frac{\alpha_1}{2} \| v_x \|_{L^2_\eta}^2 + K_2 \| \v \|_{X_\eta}^2 + \| \v \|_{X_\eta} \| \r \|_{X_\eta}.
\end{align*}
This yields
\begin{align} \label{equ:proof3}
	\Re s \| \v \|_{X_\eta}^2 + \frac{\alpha_1}{2} \| v_x \|_{L_\eta^2}^2 \le K_2 \| \v \|_{X_\eta}^2 + \| \v \|_{X_\eta} \|\r \|_{X_\eta}.
\end{align}
The remaining proof falls naturally into three cases depending on the value of $s\in \Omega_0$. \\
\textbf{Case 1:} $\Re s \ge |\Im s|$, $\Re s > 0$, $|s| \ge 2\sqrt{2} K_2$. \\
We have $0 < \Re s \le |s| \le \sqrt{2} \Re s$. Therefore, using \eqref{equ:proof3} and Young's inequality with $\varepsilon = \frac{\sqrt{2}}{|s|}$, we obtain
\begin{align*}
	\frac{|s|}{\sqrt{2}} \| \v \|_{X_\eta}^2 + \frac{\alpha_1}{2} \| v_x \|_{L_\eta^2}^2 & \le \frac{|s|}{2\sqrt{2}} \| \v \|_{X_\eta}^2 + \| \v \|_{X_\eta} \| \r \|_{X_\eta} \\
	& \le \frac{|s|}{2\sqrt{2}} \| \v \|_{X_\eta}^2 + \frac{|s|}{4\sqrt{2}} \| \v \|_{X_\eta}^2 + \frac{\sqrt{2}}{|s|} \| \r \|_{X_\eta}^2.
\end{align*}
Thus, for a suitable constant $C$
\begin{align*}
	|s| \| \v \|_{X_\eta}^2 + \| v_x \|_{L_\eta^2}^2 \le \frac{C}{|s|} \| \r \|_{X_\eta}^2.
\end{align*}
\textbf{Case 2:} $| \Im s | \ge \Re s \ge 0$. \\
From \eqref{equ:proof3} we have
\begin{align*}
	\| v_x \|_{L_\eta^2} \le \frac{2}{\alpha_1} \left( K_2 \| \v \|_{X_\eta}^2 + \| \v \|_{X_\eta} \| \r \|_{X_\eta} \right).
\end{align*}
Use this in \eqref{equ:proof2} and find a constant $K_3>0$ such that
\begin{align*}
	|s| \| \v \|_{X_\eta}^2 
	& \le K_3 \left( \| \v \|_{X_\eta}^2 + \| \v \|_{X_\eta} \| \r \|_{X_\eta} \right).
\end{align*}
 Take $|s| > 2K_3$ and use Young's inequality with $\varepsilon = |s|^{-1}$
\begin{align*}
	|s| \| \v \|_{X_\eta}^2 & \le \frac{|s|}{2} \| \v \|_{X_\eta}^2 + K_3 \| \v \|_{X_\eta} \| \r \|_{X_\eta} \le \frac{|s|}{2} \| \v \|_{X_\eta}^2 + \frac{|s|}{4} \| \v \|_{X_\eta}^2 + \frac{K_3^2}{|s|} \| \r \|_{X_\eta}^2, 
\end{align*}
hence
\begin{align} \label{equ:proof4}
	|s| \| \v \|_{X_\eta}^2 \le \frac{4 K_3^2}{|s|} \| \r \|_{X_\eta}^2.
\end{align}
Using \eqref{equ:proof3}, \eqref{equ:proof4} and taking $|s|\ge 4K_2$ yields by Young's inequality with $\varepsilon = |s|^{-1}$
\begin{align} \label{equ:proof5}
\begin{split}
  \frac{\alpha_1}{2} \| v_x \|_{L_\eta^2}^2 & \le \frac{|s|}{4} \| \v \|_{X_\eta}^2 + \| \v \|_{X_\eta} \| \r \|_{X_\eta} \le \frac{|s|}{4} \| \v \|_{X_\eta}^2 + \frac{|s|}{4} \| \v \|_{X_\eta}^2 + \frac{1}{|s|} \| \r \|_{X_\eta}^2
 + \frac{K_4}{|s|} \| \r \|_{X_\eta}^2.
\end{split}
\end{align}
Combining \eqref{equ:proof4} and \eqref{equ:proof5} we arrive at the estimate
\eqref{resolest1}.

\textbf{Case 3:} $\Re s \le 0$, $|\Re s | \le \varepsilon_0 | \Im s |$.
Using \eqref{equ:proof2} and \eqref{equ:proof3} yields 
\begin{align*}
	& | \Im s | \| \v \|_{X_\eta}^2 \le |s| \| \v \|_{X_\eta}^2 \le K_0 \|v_x \|_{L_\eta^2}^2 + K_1 \| \v\|_{X_\eta}^2 + \| \v \|_{X_\eta} \| \r \|_{X_\eta} \\
	& \le \frac{2 K_0}{\alpha_1} \left( |\Re s| \| \v \|_{X_\eta}^2 + K_2 \| \v \|_{X_\eta}^2 + \| \v \|_{X_\eta} \| \r \|_{X_\eta} \right) + K_1 \| \v \|_{X_\eta}^2 + \| \v \|_{X_\eta} \| \r \|_{X_\eta}.
\end{align*}
Choose $0 < \varepsilon_0 < \frac{\alpha_1}{4 K_0}$, so that
$\frac{2K_0}{\alpha_1} |\Re s| \le \frac{|\Im s|}{2}$ holds. Then we conclude 
\begin{align*}
	|\Im s| \| \v \|_{X_\eta}^2 \le K_5 \left( \| \v \|_{X_\eta}^2 + \| \v \|_{X_\eta} \| \r \|_{X_\eta} \right).
\end{align*}
Since $|s| \le \sqrt{1+ \varepsilon_0^2} |\Im s|$ we also have
\begin{align*}
	|s| \| \v \|_{X_\eta}^2 \le K_6 \left( \| \v \|_{X_\eta}^2 + \| \v \|_{X_\eta} \| \r \|_{X_\eta} \right).
\end{align*}
Now take $|s| > 2K_6$ and use Young's inequality with $\varepsilon = |s|^{-1}$ to find
\begin{align*}
	|s| \| \v \|_{X_\eta}^2 & \le \frac{|s|}{2} \| \v \|_{X_\eta}^2 + K_6 \| \v \|_{X_\eta} \| \r \|_{X_\eta} 
	 \le \frac{|s|}{2} \| \v \|_{X_\eta}^2 + \frac{|s|}{4} \| \v \|_{X_\eta}^2 + \frac{K_6^2}{|s|}\| \r \|_{X_\eta}^2,
\end{align*}
which yields
\begin{align} \label{equ:proof6}
	|s| \| \v \|_{X_\eta}^2 \le \frac{K_7}{|s|} \| \r \|_{X_\eta}^2.
\end{align}
To complete \eqref{resolest1}, take $|s| \ge 2K_2$ in \eqref{equ:proof3} and use \eqref{equ:proof6},
\begin{align*}
\begin{split}
	\frac{\alpha_1}{2} \| v_x\|_{L_\eta^2}^2 & \le |\Re s| \| \v \|_{X_\eta}^2 + \frac{|s|}{2} \| \v \|_{X_\eta}^2 + \| \v\|_{X_\eta} \| \r \|_{X_\eta} \\
	& \le |s| \| \v \|_{X_\eta}^2 + \frac{|s|}{2} \| \v \|_{X_\eta}^2 + \frac{|s|}{2} \| \v \|_{X_\eta}^2 + \frac{1}{2|s|} \| \r \|_{X_\eta}^2 \\
	& = 2|s| \| \v \|_{X_\eta}^2 + \frac{1}{2 |s|} \| \r \|_{X_\eta}^2 \le \frac{K_8}{|s|} \| \r \|_{X_\eta}^2.
\end{split}
\end{align*}
It remains to prove \eqref{resolest2}. 
The resolvent equation \eqref{resolvGl} implies  the following equation
in $X_\eta$,
\begin{align*}
\vek{-v_{xx}}{0}  = \vek{A^{-1}( -sv + cv_x + S_\omega v + Cv + r )}{A^{-1}( -s\rho + S_\omega \rho + C_\infty \rho + \zeta )}.
\end{align*}
Thus,  with the help of \eqref{eq3:basic} we obtain for $|s| \ge 1$ the estimate
\begin{align*}
	\| v_{xx} \|_{L_\eta^2}^2 & \le K_9 \left( |s|^2 \| \v \|_{X_\eta}^2 + \| v_x \|_{L_\eta^2}^2 + \| \v \|_{X_\eta}^2 + \| \r \|_{X_\eta}^2 \right) \\
	& \le 2K_9 \left( |s|^2 \| \v \|_{X_\eta}^2 + |s| \| v_x \|_{L_\eta^2}^2 + \| \r \|_{X_\eta}^2 \right). 
\end{align*}
When combined with \eqref{resolest1} this proves our assertion.
\end{proof}
In the next step we study the Fredholm property of the operator
$\L_\eta$ in \eqref{calL}. First we consider the operator $L$ from \eqref{LatL} on $L^2_\eta$ and therefore, as in \eqref{calL}, indicate the dependence on the weight $\eta$ by a subindex. So we introduce
\begin{align*}
	L_\eta: H^2_\eta \rightarrow L^2_\eta, \quad v \mapsto Av_{xx} + cv_x + S_\omega v + Df(v_\star)v.
\end{align*} 
Further, let us transform $L_\eta$ into unweighted spaces
\begin{equation} \label{eq3:unweighted}
  L_{[\eta]}:H^2 \to L^2, \quad v \mapsto \eta L_\eta (\eta^{-1}v)=
  Av_{xx} + B(\mu)v_x +C(\mu)v,
\end{equation}
where $q(x) = \sqrt{x^2 +1}$ and
\begin{equation*}
  B(\mu,x) = cI+\frac{2 \mu x}{q(x)}A, \quad
  C(\mu,x)= S_{\omega}+Df(v_{\star}(x)) 
  - \frac{c \mu x}{q(x)}I + \big(\frac{\mu^2 x^2}{q(x)^2} -\frac{\mu}{q(x)}
  +\frac{\mu x^2}{q(x)^3}\big)A.
\end{equation*}
The limits as $x \to \pm \infty$ of these matrices are given by
\begin{equation} \label{eq3:limits}
  B_{\pm,\mu}=cI \pm 2 \mu A, \quad
  C_{\pm,\mu} = S_{\omega}+ \mu^2 A  +
  \begin{cases} -c\mu I + Df(v_{\infty}), & \text{in case}\;\; +, \\
    c\mu I + Df(0), &\text{in case}\; \; -.  \end{cases}
\end{equation}
With these limit matrices we define the piecewise constant operator
\begin{align*}
  L_{[\eta],\infty}:H^2 \to L^2,\quad v \to  Av_{xx}+
  (B_{-,\mu}\one_{\R_-}+B_{+,\mu} \one_{\R_+})
  v_x +(C_{-,\mu}\one_{\R_-}+C_{+,\mu} \one_{\R_+})v.
\end{align*}
The following lemma shows that it is sufficient to analyze the Fredholm properties
of $L_{[\eta],\infty}$.
\begin{lemma} \label{lemma4.14}
  Let Assumption \ref{A1} and \ref{A2} be satisfied, and assume
  $0 \le \mu \le \mu_2$ with $\mu_2$ from Lemma \ref{aprioriest}. Then
  for each $s \in \C$ the following are equivalent:
\begin{enumerate}[i)]
\item The operator $sI-L_{[\eta],\infty}: H^2 \rightarrow L^2$ is  Fredholm of index $k$.
  \item The operator $sI-L_\eta: H^2_\eta \rightarrow L^2_\eta$ is Fredholm of index $k$.
  \item The operator $sI-\L_\eta: Y_\eta \rightarrow X_\eta$ is Fredholm of index $k$.
\end{enumerate}
\end{lemma}
\begin{proof} Both equivalences $(i) \Leftrightarrow (ii)$ and
  $(ii) \Leftrightarrow (iii)$ follow from the invariance of
  the Fredholm index under compact perturbations \cite[Ch.IX]{EdmundsEvans}.
  For the first equivalence note that the multiplication operator
  $v \mapsto m(\cdot) v$ is compact from $H^1$ to $L^2$ if $m \in L^{\infty}$ and
  $\lim_{x \to \pm \infty}m(x) =0$ (\cite[Lemma 4.1]{BeynLorenz}).
  This shows that the Fredholm property transfers from $L_{[\eta],\infty}$
  to $L_{[\eta]}$ (see \eqref{eq3:unweighted}), and thus also to
  $L_\eta:H^2_{\eta} \to L^2_{\eta}$.  
  For the second equivalence use the homeomorphism
  $\mathcal{T}: X_{\eta} \to L_{\eta}^2 \times \R^2$,
  $(v,\rho)^{\top} \to (v - \rho \hat{v},\rho)^{\top}$ and transform $\L_\eta$ into
  the block operator
  \begin{align*}
    \mathcal{T} \L_\eta \mathcal{T}^{-1} = \begin{pmatrix} L_{\eta} & \mathcal{K}\\
     0 & E_{\omega} 
    \end{pmatrix}, \quad \mathcal{K}= \hat{v}_{xx} A+ c \hat{v}_xI +
   \hat{v}( Df(v_{\star})-Df(v_{\infty})).
  \end{align*}
  Since $\mathcal{K}$ is bounded in $L^2_{\eta}$ the result
  follows from the Fredholm bordering lemma (\cite[Lemma 2.3]{Beyn90}).
  \end{proof}
The Fredholm property of $sI - L_{[\eta],\infty}$ can be determined from
the first order system corresponding to $(sI -L_{[\eta],\infty})v=r$, i.e.
$w=(v,v_x)^{\top}$ and
\begin{equation*}
  \begin{aligned}
  (0,r)^{\top}&= w_x - (M_{-,\mu}(s) \one_{\R_-}+ M_{+,\mu}(s)\one_{\R_+})w, \\
  M_{\pm,\mu}(s) & =  \, \begin{pmatrix} 0 & I \\ A^{-1}(sI- C_{\pm,\mu}) &
    - A^{-1}B_{\pm,\mu} \end{pmatrix}.
  \end{aligned}
  \end{equation*}
We define the {\it $\mu$-dependent Fredholm set} by
\begin{align*}
  \Omega_{F}(\mu) = \{ s \in \C : M_{-,\mu}(s) \; \text{and} \;
  M_{+,\mu}(s) \; \text{are hyperbolic} \}
\end{align*}
and denote by $m_{\st}^{\pm}(s,\mu)$ the dimension of the stable subspace of
$M_{\pm,\mu}(s)$ for $s \in \Omega_{F}(\mu)$.  Rewriting the eigenvalue problem for $M_{\pm,\mu}(s)$ in $\C^4$
as a quadratic eigenvalue problem in $\C^2$ shows that
$\Omega_{F}(\mu)= \C \setminus \sigma_{\mathrm{disp},\mu}(\L_\eta)$ holds
for  the $\mu$-dependent dispersion set,
\begin{equation} \label{eq3:dispersionset}
  \begin{aligned}
  \sigma_{\mathrm{disp},\mu}(\L_\eta) = & \, \{s \in \C: \det(sI - D_{\pm}(\nu,\mu))
  =0 \; \; \text{for some}\; \; \nu \in \R \; \; \text{and some sign} \;\;
  \pm \}, \\
  D_{\pm}(\nu,\mu)= & - \nu^2 A + i \nu B_{\pm,\mu} + C_{\pm,\mu}.
  \end{aligned}
\end{equation}
The following Lemma is well-known and appears for example in
\cite[Lemma 3.1.10]{KapitulaPromislow}, \cite{Palmer88}, \cite[Sec. 3]{Sandstede02}.
\begin{lemma} \label{FredholmLetainfty}
  Let Assumption \ref{A1} and \ref{A2} be satisfied, and let
  $0 \le \mu \le \mu_2$ with $\mu_2$ from Lemma \ref{aprioriest}. Then the operator $sI - L_{[\eta],\infty}:H^2 \rightarrow L^2$ is a Fredholm operator if and only if
  $s \in \Omega_F(\mu)$. If $s \in \Omega_F(\mu)$ then the Fredholm index is given by
\begin{align} \label{eq3:indexformula}
	\mathrm{ind}(sI - L_{[\eta],\infty}) =  m_\st^+(s,\mu) - m_\st^-(s,\mu).
\end{align}
\end{lemma}
\begin{remark} An intuitive argument for the formula \eqref{eq3:indexformula} is as follows.
The Fredholm index
measures the degrees of freedom of a linear problem minus the number of constraints. In this case there are $m_{\st}^+(s,\mu)$ forward decaying modes and
$m -m_{\st}^-(s,\mu)$ ($m=\dim(w)=4$)  backward decaying modes, adding up to
$m_{\st}^+(s,\mu)+m -m_{\st}^-(s,\mu)$ degrees of freedom. The condition that
these modes fit together at the origin provides $m$ constraints
which then leads to the index formula \eqref{eq3:indexformula}.
\end{remark}

\begin{figure}[h!]
\begin{minipage}[t]{0.49\textwidth}
\centering
\includegraphics[scale=1]{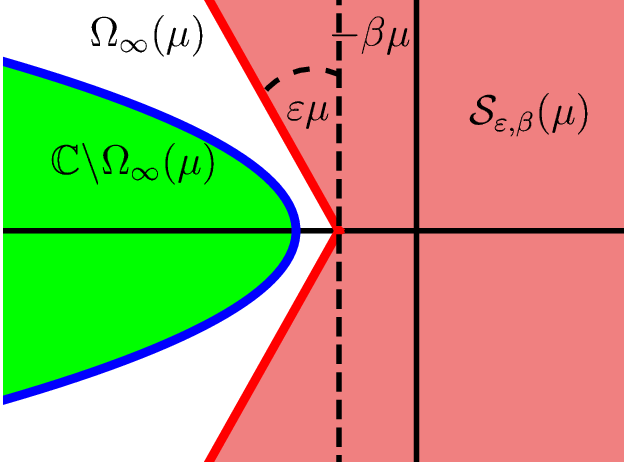}
\end{minipage}
\begin{minipage}[t]{0.49\textwidth}
\centering
\includegraphics[scale=1]{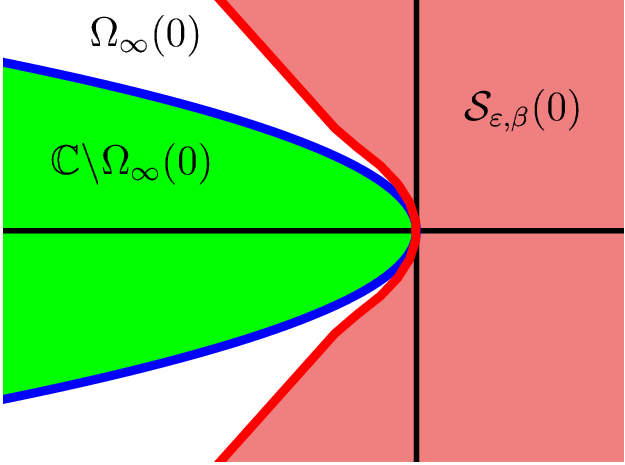}
\end{minipage}
\caption{The sectors $\mathcal{S}_{\varepsilon, \beta}(\mu)$, $\mu > 0$ (left) and $\mathcal{S}_{\varepsilon, \beta}(0)$ (right) from Theorem \ref{thm4.17}.}
\label{sectorials}
\end{figure}

Next we show how the Fredholm index $0$ domain extends into the left half plane.

\begin{theorem} \label{thm4.17}
  Let Assumption \ref{A1}, \ref{A2} and \ref{A4} be satisfied and let $\mu_2$ be given
  by Lemma \ref{aprioriest}. Then there exist constants $\mu_0 \in (0,\mu_2)$ and $ \beta, \varepsilon, \kappa > 0$ such that for each $0\le \mu \le \mu_0$ the open
  set
  $\Omega_F(\mu)$ has a unique connected component $\Omega_\infty(\mu)$  satisfying
  \begin{equation} \label{eq3:sectorinomega}
    \begin{aligned}
      \S_{\varepsilon,\beta}(\mu) :=& \, \big\{ s \in \C: | \arg (s + \beta \mu) | \le \frac{\pi}{2} + \varepsilon \mu \big\} \subset \Omega_\infty(\mu),\quad
      \text{if} \quad \mu>0, \\
      \S_{\varepsilon,\beta}(0) :=& \, \big\{ s \in \C: \Re s \ge -
      \kappa \min( |\Im s|,\beta)^2 + \varepsilon\min(\beta - |\Im s|,0) \big\} \subset \Omega_\infty(0), \quad \text{if}\quad \mu=0.
      \end{aligned}
  \end{equation}
  Moreover, $\Omega_{\infty}(\mu)$ has the properties
\begin{enumerate}[i)]
\item For all $s \in \Omega_\infty(\mu)$ the operator $sI - \L_\eta: Y_\eta \rightarrow X_\eta$ is Fredholm of index $0$.
\item $\sigma_{\mathrm{ess}}(\L_\eta) \subseteq \C \backslash \Omega_\infty(\mu)$.
\end{enumerate}
\end{theorem}

  Figure \ref{Figure:Fred} illustrates the spectral behavior for $\mu=0$, $\L = \L_\eta$ in case of
  the quintic Ginzburg Landau equation \eqref{QCGL} with $\alpha=1$, $\beta_1= -\frac{1}{8}$, $\beta_3=1+i$, $\beta_5= -1 +i$. The dispersion set
  $\sigma_{\mathrm{disp},0}(\L)$ consists of $4$ parabola-shaped curves. They
  originate from
  purely imaginary eigenvalues of $M_{-,0}(s)$ (red) and of 
  $M_{+,0}(s)$ (blue). Note that one of the latter curves has quadratic contact
  with the imaginary axis. The numbers in the connected components
  of $\C\setminus \sigma_{\mathrm{disp},0}(\L)$ denote the Fredholm index as calculated
  from \eqref{eq3:indexformula}. The white components have index $0$ with
  $\Omega_{\infty}(0)$ being the rightmost component. The essential spectrum
  $\sigma_{\mathrm{ess}}(\L)$ (see \eqref{eq2:defessential}) is colored green.
  Every $\mu >0$ sufficiently small shifts the spectrum to the left
  (by $\approx -c \mu$) which allows to inscribe a proper
  sector with tip at $-\beta \mu$ into the unbounded Fredholm $0$ component $\Omega_{\infty}(\mu)$. In case $\mu=0$ the sector is rounded quadratically 
  for $|\Im s| \le \beta$; see Figure \ref{sectorials}.

\begin{figure}[h!]
\centering
\begin{minipage}[t]{0.99\textwidth}
\centering
\includegraphics[scale=1]{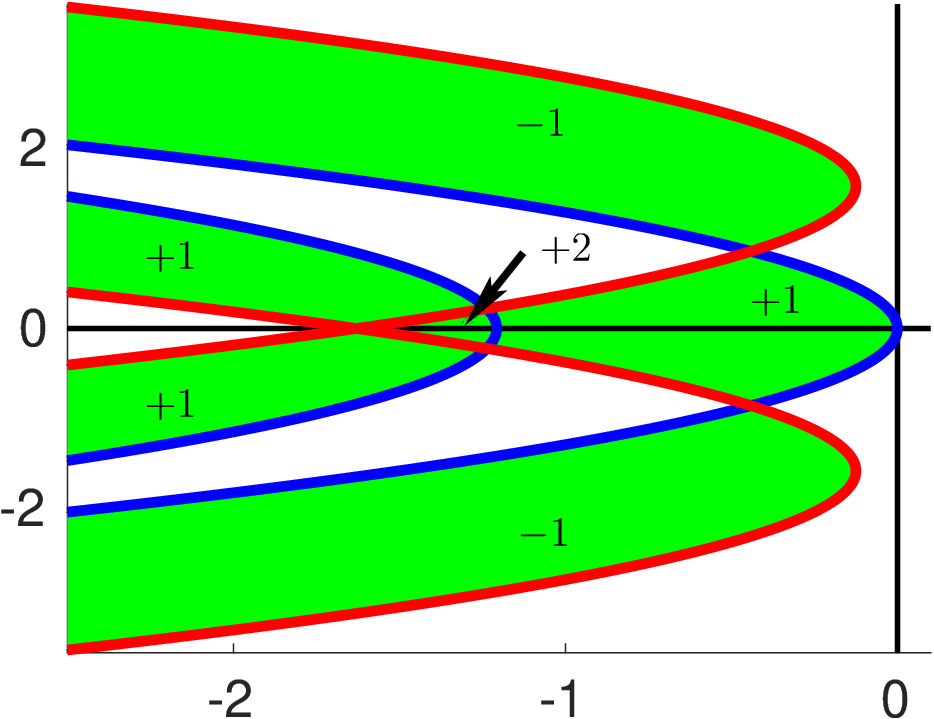}
\caption{Essential spectrum of the linearized operator $\L$ (green) and
  dispersion set (red/blue) for \eqref{QCGL} with
  $\alpha=1$, $\beta_1= -\frac{1}{8}$, $\beta_3=1+i$, $\beta_5= -1 +i$. The numbers in the connected components indicate the Fredholm index of $sI - \L$,
  white regions have Fredholm index $0$.}\label{Figure:Fred}
\end{minipage}
\end{figure}

\begin{proof} We show  $\S_{\varepsilon,\beta}(\mu) \subseteq 
  \C \setminus \sigma_{\mathrm{disp},\mu}(\L_\eta)=\Omega_F(\mu)$, so that \eqref{eq3:sectorinomega}
  follows by taking $\Omega_{\infty}(\mu)$ to be the connected component
  of $\Omega_F(\mu)$ which contains this sector.
  From  \eqref{SystemDef} one finds
  \begin{align} \label{eq3:fderiv}
    Df(v) = g(|v|^2) + 2 g'(|v|^2) \begin{pmatrix} v_1^2 & v_1 v_2 \\
      v_1 v_2 & v_2^2 \end{pmatrix}.
  \end{align}
 From \eqref{eq3:limits},\eqref{eq3:dispersionset} we obtain
  \begin{align*}
    D_{-}(\nu,\mu)= - \nu^2 A + i \nu(cI-2 \mu A) + S_{\omega} + g(0) +
    \mu^2 A + c \mu I.
  \end{align*}
  The dispersion set contains the two curves of eigenvalues $s(\nu,\mu)$ and
  $\overline{s(\nu,\mu)}$ of $D_{-}(\nu,\mu)$ given for $\nu \in \R$ by
  \begin{align*} 
    s(\nu,\mu)= (\mu^2- \nu^2) \alpha_1 + 2 \nu \mu \alpha_2 + g_1(0) + c \mu
    +i \big((\mu^2- \nu^2) \alpha_2 + \nu c - 2 \nu \mu \alpha_1 + \omega +g_2(0)
    \big).
  \end{align*} 
  An elementary discussion shows that $\Re s(\nu,\mu) \le g_1(0) <0$
  for all $\nu \in \R$ and $0 \le \mu \le \frac{c \alpha_1}{|\alpha|^2}$.
  Moreover, for $|\nu|$ large, the values $s(\nu,\mu)$ lie in a sector which has an  opening angle with the negative real axis of at most
  $\arctan(\frac{|\alpha_2|}{\alpha_1}) < \frac{\pi}{2}$
  uniformly for $\mu$ small. Thus there is a sector  $\S_{\varepsilon,\beta}(\mu)$ as above
  which has the two curves in its exterior. Next we obtain from
  \eqref{eq3:limits}, \eqref{eq3:dispersionset}, \eqref{eq3:fderiv}
  and Lemma \ref{lem:asym}
  \begin{equation} \label{eq3:D+formula}
    \begin{aligned}
    D_+(\nu,\mu)= &\, \begin{pmatrix} \delta_1(\nu,\mu) + 2 \rho_1 & - \delta_2(\nu,\mu) \\
      \delta_2(\nu,\mu) + 2 \rho_2 & \delta_1(\nu,\mu) \end{pmatrix}, \quad
    \rho_j = g_j'(|v_{\infty}|^2)|v_{\infty}|^2, j=1,2,  \\
     \delta_1(\nu,\mu)=&\, (\mu^2-\nu^2) \alpha_1 + 2 i \nu \mu  \alpha_1
      + i \nu c- c \mu, \quad \delta_2(\nu,\mu) =
      (\mu^2-\nu^2)\alpha_2 + 2 i\nu \mu \alpha_2.
    \end{aligned}
  \end{equation}
  The eigenvalues are
  \begin{align} \label{eq3:D+eigenvalues}
    s_{\pm}(\nu,\mu) = \delta_1(\nu,\mu) + \rho_1 \pm \big(\rho_1^2 -
    2 \delta_2(\nu,\mu) \rho_2 - \delta_2^2(\nu,\mu)\big)^{1/2}.
  \end{align}
  Consider first the case $\mu=0$: 
  \begin{align*}
    s_{\pm}(\nu,0) = - \alpha_1 \nu^2 + \rho_1 + i \nu c \pm R(\nu)^{1/2},
    \quad R(\nu) = |\rho|^2 - (\rho_2 - \alpha_2 \nu^2)^2.
  \end{align*}
  For $ \nu \to \infty$ we obtain $\Re s_{\pm}(\nu,0) \sim - \alpha_1 \nu^2$
  and $\Im s_{\pm}(\nu,0) \sim \nu c \pm |\alpha_2|\nu^2$, hence large
  values lie in  a sector opening to the left with angle $< \frac{\pi}{2}$.
  Further, Assumption \ref{A2} yields $\Re s_-(\nu,0) \le -\alpha_1 \nu^2 + \rho_1 \le \rho_1<0$ 
  for all $\nu\in \R$ (independently of the sign in front of $R(\nu)$).
  Next we show that $r(\nu)= \Re s_+(\nu,0)$ has a unique global maximum at $\nu=0$, more precisely,
    \begin{equation} \label{eq3:nubehave}
    r(0)=0, \quad   r''(0) < 0  \quad  r'(\nu)\begin{cases} >0, & \nu \in (-\infty,0) \setminus \{-\nu_0\} , \\
      =0, & \nu=0, \\ < 0, & \nu \in (0,\infty)\setminus \{\nu_0\} , \end{cases}
    \end{equation}
    where $\nu_0> 0$  is defined by
    $ |\alpha_2| \nu_0^2 = |\rho| + \mathrm{sgn}(\alpha_2)\rho_2>0$  for
    $\alpha_2 \neq 0$ and 
     $\nu_0=\infty$ in case $\alpha_2=0$. Note that  $|\nu|> \nu_0$ holds
    if and only if $R(\nu) <0$, and $r(\nu)=-\alpha_1 \nu^2+\rho_1$ for $|\nu|> \nu_0$.
    Moreover, if $\alpha_2=0$ then 
     we have $r(\nu)=-\alpha_1 \nu^2$ . Thus \eqref{eq3:nubehave} holds in both cases.
    It remains to consider $r(\nu)$ for $|\nu| < \nu_0$ and $\alpha_2 \neq 0$.
    We obtain $R(0)=|\rho_1|^2$, $r(0)=0$ and 
    \begin{align*}
      r'(\nu) = &\, 2 \nu \big( -\alpha_1 +
      R(\nu)^{-1/2} \alpha_2 (\rho_2- \alpha_2
      \nu^2)\big).
    \end{align*}
    This shows $r'(0)=0$ and 
    $|\rho_1|r''(0)=2(-\alpha_1|\rho_1|+\alpha_2 \rho_2) <0$
    by Assumption \ref{A4}. If $r'$ has a further zero for $|\nu|<\nu_0$
    then this implies $\alpha_2(\rho_2 - \alpha_2 \nu^2) >0$, $\mathrm{sgn}(\alpha_2)= \mathrm{sgn}(\rho_2)$, and  
    $|\alpha|^2(\rho_2- \alpha_2 \nu^2)^2= \alpha_1^2 |\rho|^2$.
    By these sign conditions we have a unique square root given by
\begin{align*}
  |\alpha|( \rho_2 - \alpha_2 \nu^2) = \mathrm{sgn}(\alpha_2) \alpha_1 |\rho|, \quad \text{or} \quad |\alpha||\alpha_2|\nu^2 =
  |\alpha| |\rho_2|  - \alpha_1|\rho|.
\end{align*}
But the last equation has no real solution $\nu$ since our Assumptions
\ref{A1}-\ref{A4} and $\alpha_2 \rho_2 >0$ imply
\begin{align*}
  |\alpha|^2 \rho_2^2-\alpha_1^2 |\rho|^2  =\alpha_2^2 \rho_2^2 -
  \alpha_1^2 \rho_1^2=( \alpha_2 \rho_2-\alpha_1 \rho_1)( \alpha_2 \rho_2+\alpha_1 \rho_1 ) <0.
\end{align*}
Since $r'$ has no further zeros in $(-\nu_0,0)\cup (0,\nu_0)$ the sign of $r'$
is determined by $r''(0)<0$.

Now consider $s_{\pm}(\nu,\mu)$ for small values $\mu>0$. For large values
of 
$|\nu| $ the asymptotic behavior is only slightly modified to $\Re s_{\pm}(\nu,\mu) \sim - \alpha_1 \nu^2$
and $\Im s_{\pm}(\nu,\mu) \sim \nu(c+ 2 \mu \alpha_1) \pm |\alpha_2|\nu^2$,
and we find a proper sector enclosing the dispersion curves for $|\nu|\ge \nu_1$
uniformly for $0\le \mu \le \mu_1$.  The curve
$s_-(\nu,0)$, $|\nu| \le \nu_1$ is bounded away from the imaginary axis, hence
it suffices to consider $s_+(\nu,\mu)$. From \eqref{eq3:D+formula},
\eqref{eq3:D+eigenvalues} one computes the partial derivatives
 $s_{+,\nu}(0,0)=ic$, $s_{+,\mu}(0,0)= - c$, $s_{+,\nu \nu}(0,0)=2 |\rho_1|^{-1}(\rho_1 \alpha_1 + \rho_2 \alpha_2)$ and then by a Taylor expansion
\begin{align} \label{eq3:Taylors+}
  s_+(\nu,\mu) = i c \nu - c \mu + |\rho_1|^{-1}(\rho_1 \alpha_1 + \rho_2 \alpha_2)\nu^2 + \mathcal{O}(|\nu|^3 + |\mu| |\nu| + |\mu|^2 ).
\end{align}
Note that we can estimate $C |\mu| |\nu| \le \frac{C}{4 \delta} \mu^2 + C \delta \nu^2$ and absorb $C \delta$   into the negative pre-factor of $\nu^2$ by taking $\delta$ small, and subsequently absorb $\frac{C}{4 \delta}\mu$ into
the negative pre-factor of $\mu$ by taking $\mu$ small.
Thus we can choose $\beta= \frac{c}{2}$ and determine $\varepsilon,\nu_2,\mu_2>0$ such
that
\begin{align*}
   |\arg (s_+(\nu,\mu) + \beta \mu) | \ge \frac{\pi}{2} + \varepsilon \mu
   \quad \text{for all } \quad |\nu| \le \nu_2, \quad 0 < \mu \le \mu_2.
\end{align*}
Finally, for $\nu_2 \le |\nu| \le \nu_1$, the curve $s_+(\nu,0)$ is bounded
away from the imaginary axis, and our assertion follows by continuity of $s_+$.
For $\mu=0$ the expansion \eqref{eq3:Taylors+} ensures $\Im s_+(\nu,0)= \nu c + \mathcal{O}(|\nu|^3)$ and the existence of some
$\beta>0$ such that
\begin{align*}
  \Re s_+(\nu,0) \le - 2 \kappa |\Im s_+(\nu,0)|^2 \quad \text{for} \quad
  |\Im s_+(\nu,0)| \le \beta, \quad \kappa = \frac{|\rho_1 \alpha_1 +\rho_2 \alpha_2|}{4 c^2 |\rho_1|}.
\end{align*}
The inclusion \eqref{eq3:sectorinomega} then follows for $\varepsilon$
sufficiently small. In addition, we require $\varepsilon \le \kappa \beta$
which implies
\begin{align} \label{eq3:sectormove}
  \S_{\varepsilon,\tilde{\beta}}(0) \subset
  \S_{\varepsilon,\beta}(0) \quad \text{for} \quad
  0 < \tilde{\beta} < \beta.
\end{align}
For the proof of (i) note that the Fredholm index is constant in $\Omega_{\infty}(\mu)$. Therefore, it is enough to 
consider $M_{\pm,\mu}(s)$ for large $s >0$:
\begin{align*}
  \begin{pmatrix} I & 0 \\ 0 & s^{-1/2}I \end{pmatrix} M_{\pm,\mu}(s)
  \begin{pmatrix} I & 0 \\ 0 & s^{1/2} I \end{pmatrix}=
  s^{1/2}\Big[ \begin{pmatrix} 0 & I \\ A^{-1} & 0 \end{pmatrix}
    + s^{-1/2} \begin{pmatrix} 0 & 0 \\ - A^{-1} C_{\pm,\mu} &
      - A^{-1} B_{\pm,\mu} \end{pmatrix} \Big].
\end{align*}
The leading matrix $\begin{pmatrix} 0 & I \\ A^{-1} & 0 \end{pmatrix}$
has a two-dimensonal stable subspace which belongs to the  eigenvalues $- \alpha^{-1/2},-\bar{\alpha}^{-1/2}$,
and a two-dimensional unstable subspace which belongs to $\alpha^{-1/2},\bar{\alpha}^{-1/2}$.
These  subspaces are only slightly perturbed for large $s$, and hence we obtain the 
Fredholm index $0$ from Lemma \ref{FredholmLetainfty}.

We conclude the proof by noting that assertion (ii) follows from  
our previous result and the definition of the essential spectrum,
cf. \eqref{eq2:defessential}.      
\end{proof}
We continue with the
\begin{proof}[Proof of Lemma \ref{lemma4.18}]
  
  Clearly, Theorem \ref{decay} ensures that the functions $\varphi_1$, $\varphi_2$ from \eqref{eq2:eigenfunctions} lie in every $Y_{\eta}$, where $\eta$ is defined by \eqref{eta} and $0 \le \mu < 2 \hat{\mu}$, and
   both are eigenfunctions of $\L_\eta$ defined in 
  \eqref{calL}. Recall that $E_{\omega}=S_{\omega}+ Df(v_{\infty})=2
  \begin{pmatrix} \rho_1 & 0 \\ \rho_2 & 0 \end{pmatrix}$ (see
  \eqref{eq3:D+formula} for $\rho_1,\rho_2$) has the eigenvalues $0$ and $ 2 \rho_1 <0$ with
  eigenvectors $(0,|v_{\infty}|)^{\top}= S_1 v_{\infty}$ and $(\rho_1,\rho_2)^{\top}$, respectively.
  Now consider $(v,\zeta)^{\top} \in Y_{\eta}$ such that $\L_\eta (v,\zeta)^\top=0$.
  From $E_{\omega}\zeta = 0$ we obtain
  $\zeta  = z S_1 v_{\infty}$ for some $z \in \C$. This shows $(u,0):= (v,\zeta)^{\top}
  - z \varphi_2 \in \ker(\L_\eta)$ and $u = y v_{\star}'$ for some $y \in \C$ by
  Assumption \ref{A5}. Therefore, we obtain $(v,\zeta)^{\top}=
  z \varphi_2 + y \varphi_1$.
  Now suppose $\L_\eta (v,\zeta)^{\top}= y \varphi_1 + z \varphi_2$ for some
  $y,z \in \C$. Since $E_{\omega}$ has only simple eigenvalues we
  conclude $z=0$ from the second component and then $y=0$ from Assumption
  \ref{A5}. This proves \eqref{eq2:eigenfunctions}.

  From Theorem \ref{thm4.17} we infer that $\L_\eta: Y_\eta \rightarrow X_\eta$ is Fredholm of index $0$ and hence also $L_\eta:H^2_{\eta}\to L^2_{\eta}$ by Lemma
  \ref{lemma4.14}. Assumption \ref{A5} guarantees that $0$ is a simple
  eigenvalue of $L_\eta:H^2_{\eta}\to L^2_{\eta}$. Note that we have
  $v'_{\star} \in H^2_{\eta}$ and that there is no solution $y \in H^2_{\eta}$ of
  $L_\eta y = v'_{\star}$ since this implies $y \in H^2$ and $y\in \ker(L^2)=\ker(L)$.
  Simple eigenvalues are known to be isolated. This may be seen by applying
  the inverse function theorem to
  \begin{equation*}
    T:H^2_{\eta}\times \R \to L^2_{\eta}\times \R, \quad
    T\begin{pmatrix} v \\ s \end{pmatrix} =
    \begin{pmatrix} L_\eta v - sv \\ (v'_{\star},v-v'_{\star})_{L^2}
    \end{pmatrix} \; \text{at} \; \begin{pmatrix} v \\ s \end{pmatrix}=
    \begin{pmatrix} v'_{\star} \\ 0 \end{pmatrix}.
  \end{equation*}
    Therefore, there exists some $s_0=s_0(\mu)>0$ such that 
    $L_\eta :H^2_{\eta}\to L^2_{\eta}$ has no eigenvalues with $|s|\le s_0$ except
    $s=0$.

Finally, we prove assertion (ii) for
$\beta_0= \min(\frac{\varepsilon s_0}{2}, \beta_E,\tilde{\beta}, |\rho_1|)$ where $\beta_E$ is from  Assumption
\ref{A5} and $\beta$, $\tilde{\beta}$ satisfy \eqref{eq3:sectorinomega}, \eqref{eq3:sectormove} as well as
$\tilde{\beta} < \min(\beta,\frac{\sqrt{3}-1}{2}s_0)$.
Consider $s \in \sigma_{\mathrm{pt}}(\L_\eta) \setminus \{0\}$ with
$\Re s \ge - \beta_0$ and eigenfunction $(v,\zeta)^{\top}\in Y_{\eta}$.
From $\sigma(E_{\omega})=\{0, 2 \rho_1\}$ we obtain $\zeta=0$ and thus $L_\eta v=s v$, $v \in H^2_{\eta}$.  We claim that $s \in \S_{\varepsilon,\tilde{\beta}}(0)$.
For $\Re s \ge 0$ this is obvious, while for
$0 > \Re s \ge -\beta_0 \ge - \frac{s_0}{2}$ this follows from
\begin{align*}
  |\Im s|\ge (s_0^2 - (\Re s)^2)^{1/2} \ge \frac{\sqrt{3}}{2}s_0 > \tilde{\beta}+ \frac{s_0}{2}, \quad \Re s \ge -\frac{\varepsilon s_0}{2} > - \kappa \tilde{\beta}^2 + \varepsilon(\tilde{\beta}- |\Im s|).
\end{align*}
By Theorem \ref{thm4.17},  $s \in \Omega_{\infty}(0)$ and 
  $s I - \L_\eta: Y_\eta \to X_\eta$ with $\mu = 0$ is Fredholm of index $0$.
By Lemma \ref{lemma4.14} the same holds for $sI-L:H^2\to L^2$ and we have
shown $s \in \sigma_{\mathrm{pt}}(L)$. This contradicts Assumption \ref{A5} since $\Re s \ge - \beta_E$.
\end{proof}

So far we determined the spectral properties of the linearized operator $\L_\eta$
and proved spectral stability of the extended system \eqref{CP} posed on the exponentially weighted spaces $X_\eta$ for positive but small $\mu > 0$. In particular, the essential spectrum of $\L_\eta$ is included in the left half plane as well as its point spectrum except the zero eigenvalue which has algebraic and geometric multiplicity $2$, i.e. $\ker (\L_\eta) = \ker (\L_\eta^2) = \mathrm{span} \{ \varphi_1, \varphi_2 \}$. Since $\L_\eta$ is Fredholm of index $0$ the same holds true for the (abstract) adjoint operator $\L_\eta^*: \D(\L_\eta^*) \subset X_\eta \rightarrow X_\eta$ and there are two normalized adjoint eigenfunctions $\psi_1, \psi_2 \in \D(\L^*)$ with
\begin{align} \label{adjointEF}
	\ker (\L_\eta^*) = \mathrm{span} \{ \psi_1,\psi_2 \}, \quad (\psi_i, \varphi_j)_{X_\eta} = \delta_{ij}, \quad i,j = 1,2.
\end{align} 
We define the map
\begin{align} \label{projector}
	P_\eta: X_\eta \rightarrow X_\eta, \quad \v \mapsto (\psi_1,\v)_{X_\eta} \varphi_1 + (\psi_2,\v)_{X_\eta} \varphi_2.
\end{align}
Then $P_\eta$ is a projection onto $\ker (\L_\eta)$ and $X_\eta$ can be decomposed into
\begin{align*}
	X_\eta = \ran (P_\eta) \oplus \ran (I - P_\eta) = \ker (\L_\eta) \oplus \ker(\L_\eta^*)^\bot.
\end{align*}
The subspace $\ker (\L_\eta^*)^\bot$ is invariant under $\L_\eta$ and we introduce
its intersection with the smooth spaces $X^k_\eta$, $k = 1,2$
\begin{align*}
	V_\eta := \ker (\L_\eta^*)^\bot \subset X_\eta, \quad V^k_\eta := V_\eta \cap X^k_\eta.
\end{align*}  

\sect{Semigroup estimates}

In the previous section we studied the spectrum of the linearized operator $\L_\eta$
on exponentially weighted spaces for positive but small $\mu > 0$ and derived a-priori estimates for the resolvent equation \eqref{resolvGl}.
Theorem \ref{thm4.17} shows that there is no essential spectrum
in the Fredholm $0$ component $\Omega_{\infty}(\mu)$ and thus also not in the sector 
$\S_{\varepsilon, \beta} (\mu)$. When combined with Lemma \ref{aprioriest}
we obtain $\Omega_0 \subset \mathrm{res}(\L_\eta)$ for the domain $\Omega_0$
from \eqref{Omega0}. Further,
Lemma \ref{lemma4.18} shows
that the nonzero point spectrum is bounded away from the imaginary axis. 
Thus we conclude from  Lemma \ref{aprioriest} that $\L_\eta$ is a sectorial operator. By the classical semigroup theory,
(see \cite{Henry}, \cite{Miklavcic}, \cite{Pazy}) the operator $\L_\eta$ generates an analytic semigroup
$\{ e^{t\L_\eta} \}_{t \ge 0}$ on $X_\eta$ such that for any $\delta >0$ there exists a constant $C_{\delta}$ with  $\| e^{t\L_\eta} \u \|_{X_\eta} \le C_{\delta} e^{\delta t} \| \u \|_{X_\eta}$, $t\ge 0$. Next we avoid the neutral modes of $\{ e^{t\L_\eta} \}_{t \ge 0}$ and  restrict  $\L_\eta$
to $V_\eta^2$ in order to have exponential decay.

\label{sec4}
\begin{theorem} \label{semigroup}
  Let Assumption \ref{A1}, \ref{A2}, \ref{A4} and \ref{A5} be satisfied and let $0 < \mu \le \min(\mu_0,\mu_1)$ with $\mu_0$ from Theorem \ref{thm4.17} and $\mu_1$ from Lemma \ref{lemma4.18}. Then the linearized operator $\L_\eta: Y_\eta \rightarrow X_\eta$ generates an analytic semigroup $\{ e^{t \L_\eta}\}_{t \ge 0}$ on $X_\eta$. Moreover, there exist $K = K(\mu) \ge 1$ and $\nu = \nu(\mu) > 0$ such that for all $t \ge 0$ and $\w \in V^\ell_\eta$, $\ell = 0,1$ the following estimate holds
\begin{align} \label{semigroupest}
	\| e^{t\L_\eta} \w \|_{X^\ell_\eta} \le K e^{-\nu t} \| \w \|_{X^\ell_\eta}.
\end{align}
\end{theorem}

\begin{proof}
  The first assertion follows by the arguments above. Thus it remains to show the estimate \eqref{semigroupest}. For that purpose, we note that the restriction $\L_{V_\eta}$ of $\L_\eta$ to $V_\eta$ is a closed operator on $V_\eta$ with $\ker(\L_{V_\eta}) = \{ 0 \}$ and $\ran (\L_{V_\eta}) = V_\eta$. Thus $\L_{V_\eta}$ is Fredholm of index $0$ and $0 \notin \sigma ( \L_{V_\eta} )$. Moreover, the projector $P_\eta$ from
  \eqref{projector} commutes with $\L_\eta$ which leads to
  $\mathrm{res}(\L_\eta) \subset \mathrm{res}( \L_{V_\eta} )$.
  Therefore, by Theorem \ref{thm4.17}, Lemma \ref{lemma4.18} and Lemma \ref{aprioriest} we find $\varepsilon = \varepsilon(\mu)$, $\nu = \nu(\mu)$ and a sector $\Sigma_{\varepsilon,\nu} = \{ s \in \C: \mathrm{arg}(s + \nu) \le \frac{\pi}{2} + \varepsilon \} $ such that $\Sigma_{\varepsilon,\nu} \subset \mathrm{res}(\L_{V_\eta})$.
  Further  we can decrease $\varepsilon>0$ and take $R > 0$ sufficiently large so that $\Sigma_{\varepsilon,\nu} \cap \{|s|\ge R\} \subseteq \Omega_0$.
  From Lemma \ref{aprioriest} and the fact that the resolvent is bounded in a compact
  subset of the resolvent set we then find a constant $C = C(\mu) >0$ such that for all $\w \in V_\eta^{\ell}$ and $\ell = 0,1$ the following holds
\begin{align*}
	\| (sI - \L_{V_\eta})^{-1} \w\|_{X^\ell_\eta} = \| (sI - \L_\eta)^{-1} \w\|_{X^\ell_\eta} \le C \| \w \|_{X^\ell_\eta} \quad & \forall s \in \Sigma_{\varepsilon, \nu} \cap \{ |s| \le R \}, \\
	\| (sI - \L_{V_\eta})^{-1} \w\|_{X^\ell_\eta} = \| (sI - \L_\eta)^{-1} \w\|_{X^\ell_\eta} \le \frac{C}{|s|} \| \w \|_{X^\ell_\eta} \quad & \forall s \in \Sigma_{\varepsilon, \nu} \cap \{ |s| > R \}.
\end{align*}
Therefore, $\L_{V_\eta}$ is a sectorial operator on $V_\eta$ and
the representation of the semigroup
\begin{align*}
  e^{t \L_{V_\eta}}= \int_{\Gamma_{\varepsilon,\nu}} (zI-\L_{V_\eta})^{-1} e^{t z} dz, \quad
  \Gamma_{\varepsilon,\nu} = \{- \nu + r \exp(\mathrm{sgn}(r)i(\tfrac{\pi}{2}+ \varepsilon)) :
  r \in \R \}
\end{align*}
leads in the standard way to the exponential estimate
\begin{align*}
	\| e^{t\L_{\eta}} \w \|_{X^\ell_\eta} = \| e^{t\L_{V_\eta}} \w \|_{X^\ell_\eta} \le K e^{-\nu t} \| \w \|_{X^\ell_\eta}, \quad \w \in V^\ell_\eta,\; \ell=0,1.
\end{align*}
\end{proof}

\sect{Decomposition of the dynamics}
\label{sec5}
The nonlinear operator $\F$ on the right hand side of \eqref{CP} is equivariant w.r.t. the group action $a(\gamma)$ from \eqref{groupaction} of the group $\G = S^1 \times \R$. Every element $\gamma$ of the group can be represented by an angle $\theta$ and a shift $\tau$. The composition $\circ: \G \times \G \rightarrow \G$ of two elements $\gamma_1,\gamma_2 \in \G$ is given by
$
	\gamma_1 \circ \gamma_2 = (\theta_1 + \theta_2 \, \text{mod} \, 2\pi, \tau_1 + \tau_2)$
and the inverse map $\gamma \mapsto \gamma^{-1}$ by $\gamma^{-1} = (-\theta \, \text{mod} \, 2\pi, -\tau)$. Both maps are smooth and $\G$ is a two dimensional $C^\infty$-manifold. An atlas of the group $\G$ is given by the two (trivial) charts $(U,\chi)$ and $(\tilde{U}, \tilde{\chi})$ defined by
\begin{alignat*}{2}
	& U = \{ \gamma = (\theta \, \text{mod} \, 2\pi, \tau) \in \G: \theta \in (-\pi,\pi), \tau \in \R \}, && \quad \chi: U \rightarrow \R^2, \; \gamma \mapsto \chi(\gamma) = (\theta, \tau), \\
    & \tilde{U} = \{ \gamma  = (\theta \, \text{mod} \, 2\pi, \tau) \in \G: \theta \in (0,2\pi), \tau \in \R \}, && \quad \tilde{\chi}: \tilde{U} \rightarrow \R^2, \; \gamma \mapsto \tilde{\chi}(\gamma) = (\theta, \tau).
  \end{alignat*}
We will always work with the chart $\chi$ since the arguments for $\tilde{\chi}$
will be almost identical.
Next we show smoothness of the group action $a(\cdot)\v$ in $\G$ depending on the regularity of $\v$.
\begin{lemma} \label{lemmagroup}
  The group action $a:\G \rightarrow GL[X_\eta]$, $\gamma \mapsto a(\gamma)$ from \eqref{groupaction} is a homomorphism and $a(\gamma)Y_\eta = Y_\eta$, $\gamma \in \G$.
  For $\v \in X_\eta$ the map $a(\cdot) \v: \G \rightarrow X_\eta$ is continuous
  and for $\v \in Y_\eta$ it is continuously differentiable.
  For $\gamma \in U$, $\chi(\gamma) = z$ the derivative applied to $h = (h_1,h_2)^\top \in \R^2$ is given by
\begin{align*}
	(a(\cdot)\v \circ \chi^{-1})'(z)h = -h_1 a(\gamma) \mathbf{S}_1 \v - h_2 a(\gamma) \v_x,
\end{align*}
where $\mathbf{S}_1 \v = (S_1v,S_1\rho)^\top$, $\v_x = (v_x, 0)^\top$ for $\v = (v,\rho)^\top$.
\end{lemma}

The proof of Lemma \ref{lemmagroup} is straightforward and will be given in the Appendix. It is based on well known properties of translation and rotation on (weighted) Lebesgue and Sobolev spaces. Next recall the Cauchy problem \eqref{CP} with perturbed initial data
\begin{align*}
	\u_t = \F(\u), \quad \u(0) = \v_\star + \v_0.
\end{align*}
We follow the approach in  \cite{BeynLorenz}, \cite{Henry} and decompose the
dynamics of the solution into a motion along the group orbit $\{a(\gamma)\v_\star:\gamma \in \G\}$ of the wave and into a perturbation $\w$ in the space $V_\eta$. We use local coordinates in $U$ and write the solution $\u(t)$ as
\begin{align} \label{decomp}
	\u(t) = a(\gamma(t)))\v_\star + \w(t), \quad \gamma(t) = \chi^{-1}(z(t)) \in U,\,\w(t) \in V_\eta
\end{align}
for $t \ge 0$. Thus $z$ describes the local coordinates of the motion on the group orbit $\mathcal{O}(\v_\star)$ given by $\gamma$ in the chart $(U,\chi)$ and $\w \in V_\eta$ is the difference of the solution to the group orbit in $V_\eta = \ker(\L_\eta^*)^\bot$. It turns out that the decomposition is unique as long as the solution stays in a small neighborhood of the group orbit and $\gamma$ stays in $U$. This will be guaranteed by taking sufficiently small initial perturbations $\v_0$. Let $P_\eta$ be the projector  onto $\ker(\L_\eta)$ from \eqref{projector} and recall from \eqref{eq2:eigenfunctions} that  $\ker(\L_\eta)$ is spanned
by the eigenfunctions $\varphi_2=\mathbf{S}_1 \v_\star$ and $\varphi_1=\v_{\star,x}$. Following \cite{BeynLorenz} we define
\begin{align} \label{Pi}
	\Pi_\eta: \chi(U) \subset \R^2 & \rightarrow \ker(\L_\eta), \quad z \mapsto P_\eta ( a(\chi^{-1}(z)) \v_\star - \v_\star).
\end{align}
For simplicity  of notation we frequently replace $\chi^{-1}(z)$ by $\gamma$
where $\gamma$ is always taken in our working chart $(U,\chi)$. The next
lemma uses $\Pi_\eta$ to show uniqueness of the decomposition \eqref{decomp} in a neighborhood of $\v_\star$.

\begin{lemma} \label{lemmatrafo}
  Let Assumption \ref{A1}, \ref{A2}, \ref{A4} and \ref{A5} be satisfied and let $\mu_1$
  be given by Lemma \ref{lemma4.18}. Then for all $0 < \mu \le \mu_1$ there is a zero neighborhood $W=W(\mu) \subset \chi(U)$ such that the map $\Pi_\eta: W \rightarrow \ker(\L_\eta)$ from \eqref{Pi} is a local diffeomorphism. Moreover, there is a zero neighborhood
  $V=V(\mu) \subset \chi(U) \times V_\eta$ such that the transformation
\begin{align*}
	T_\eta: V & \rightarrow X_\eta, \quad  (z, \w) \mapsto a(\chi^{-1}(z)) \v_\star - \v_\star + \w 
\end{align*}
is a diffeomorphism with the solution of $T_\eta(z,\w) = \v$ given by
\begin{align} \label{Tsolution}
	z = \Pi_\eta^{-1} ( P_\eta \v), \quad \w = \v + \v_\star - a(\chi^{-1}(z)) \v_\star.
\end{align}
\end{lemma}

\begin{proof}
  Since $0 < \mu \le \mu_1$ the projector $P_\eta$ and $\Pi_\eta$ are well defined. By Lemma \ref{lemmagroup} the group action $a$ is continuously differentiable
  and so is $\Pi_\eta$. Further, $\Pi_\eta(0) = 0$ and its derivative is given
  by $D\Pi_\eta (0)y = -y_1\varphi_1 - y_2 \varphi_2$, $y \in \R^2$ where $\varphi_1$, $\varphi_2$. Therefore, $D\Pi_\eta(0)$ is invertible on $\ker (\L_\eta)$
  and the first assertion is a consequence of the inverse function theorem.
  By the same arguments, $T_\eta$ is continuously differentiable,
  $T_\eta(0,0) = 0$ and its derivative is given by
  $DT_\eta(0,0) = \small{
    \begin{pmatrix}D\Pi_\eta(0)  & I_{V_{\eta}\to X_{\eta}} \end{pmatrix} }:\R^2 \times V_{\eta} \to X_{\eta}$ which is again invertible. Hence $T_\eta: V \rightarrow X_\eta$ is a diffeomorphism on a zero neighborhood $V \subset \chi(U) \times V_\eta$. Finally, applying
  $P_\eta$ to $T_\eta(z,\w) = \v$ yields $z = \Pi_\eta^{-1} ( P_\eta \v )$ while
  the second equation in \eqref{Tsolution} follows from the definition of $T_{\eta}$.
\end{proof}

Consider a smooth solution $\u(t)$, $t \in [0,t_\infty)$ of \eqref{CP} which stays close to the profile of the TOF. In particular, assume that
  $\u(t) - \v_\star, t \in [0,t_{\infty})$ lies in the region where $T_\eta^{-1}$ exists by Lemma \ref{lemmatrafo}. Then there are unique $z(t) \in \chi(U)$ and $\w(t) = (w(t), \zeta(t))^\top \in V_\eta$ for $t \in [0,t_\infty)$ such that
\begin{align*}
	\u(t) - \v_\star = T_\eta ( z(t), \w(t) ) \quad \forall t \in [0,t_\infty),
\end{align*}
and \eqref{decomp} holds. Taking the initial condition from \eqref{CP} into account yields for $t = 0$
\begin{align*}
	\v_\star + \v_0 = \u(0) = a(\chi^{-1}(z(0))) \v_\star + \w(0),
\end{align*}
which leads to $\v_0 = T_\eta(z(0), \w(0))$. Therefore, the initial conditions for $z,\w$ are given by
\begin{align} \label{initcond}
	z(0) = \Pi_\eta^{-1} ( P_\eta \v_0)=: z_0, \quad \w(0) = \v_0 + \v_\star - a(\chi^{-1}(z(0))) \v_\star =: \w_0.
\end{align}
Now we write the angular and translational components of $z$ explicitly as $z(t) = (\theta(t),\tau(t))$. 
We insert the decomposition \eqref{decomp} into \eqref{CP} and obtain 
\begin{align*}
	0 = \u_t - \L_\eta \u  = \frac{d}{dt} a(\chi^{-1}(z))\v_\star + \w_t & - a(\gamma)\vek{A v_{\star,xx} + c v_{\star,x} + S_\omega v_\star}{ S_\omega v_\infty  } \\
	& - \vek{A w_{xx} + c w_x + S_\omega w}{ S_\omega \zeta  } - \vek{f(R_{\theta} v_\star(\cdot - \tau) + w)}{f(R_{\theta} v_\infty + \zeta)}.
\end{align*}
Using the equivariance of $\F$  and the derivative of the group action from Lemma \ref{lemmagroup}, leads to 
\begin{align} \label{eq2}
	\w_t =  \L_\eta \w  -  a(\chi^{-1}(z))\varphi_1 \theta_t - a(\chi^{-1}(z))\varphi_2 \tau_t + r^{[f]}( z, \w)
\end{align}
where the remainder $r^{[f]}(z,\w)$ is given for $z=(\theta,\tau)$ and $\w \in V_{\eta}$ by
\begin{align*}
	r^{[f]}( z , \w ) := \vek{f(R_{\theta} v_\star(\cdot - \tau) + w)}{f(R_{\theta} v_\infty + \zeta)} - \vek{f(R_{\theta} v_\star(\cdot - \tau))}{f(R_{\theta} v_\infty)} - \vek{Df(v_\star)w}{Df(v_\infty)\zeta}.
\end{align*}
Let us apply the projector $P_\eta$ to \eqref{eq2} and use $\w(t) \in V_\eta$, $t \in [0,t_\infty)$  and  $P_\eta(\w_t - \L_\eta \w)=0$ to obtain the equality
\begin{align} \label{eq3}
	0 & = P_\eta r^{[f]}( z, \w ) - P_\eta a(\chi^{-1}(z))\varphi_1 \theta_t - P_\eta a(\chi^{-1}(z)) \varphi_2 \tau_t.
\end{align}
The next lemma shows that equation \eqref{eq3} can be written as an explicit ODE for $z = (\theta,\tau)$.

\begin{lemma} \label{Lemma4.5}
  Let Assumption \ref{A1}, \ref{A2}, \ref{A4} and \ref{A5} be satisfied and let $\mu_1$
  be given by  Lemma \ref{lemma4.18}. Then for all $0 < \mu \le \mu_1$ the map
\begin{align*}
	S_\eta(z): \R^2 & \rightarrow \ker(\L_\eta), \quad	y \mapsto P_\eta a(\chi^{-1}(z)) \varphi_1 y_1 + P_\eta a(\chi^{-1}(z)) \varphi_2 y_2
\end{align*}
is continuous, linear and continuously differentiable w.r.t. $z\in (-\pi,\pi) \times \R$. Moreover, there is a zero neighborhood $V=V(\mu) \subset \R^2$  such that $S_\eta(z)^{-1}$ exists for all $z \in V$ and depends continuously on $z$.
\end{lemma}

\begin{proof}
  Since $0 < \mu \le \mu_1$ the projector $P_\eta$ and the map $S_\eta(z)$ are well defined. Moreover, $S_\eta(z)$ is linear and continuous. Once more the smoothness of the group action, cf. Lemma \ref{lemmagroup}, implies that $S_\eta(z)$ is continuously differentiable w.r.t. $z$. Take $\w \in \ker (\L_\eta) =
  \mathrm{span}\{ \varphi_1, \varphi_2\}$ and recall the adjoint eigenfunctions $\psi_1, \psi_2$ from \eqref{adjointEF}. We form the inner products in $X_\eta$ of the equation $S_\eta(z) y = \w$, $y \in \R^2$ with the adjoint eigenfunctions:
\begin{align*}
	M(z) y = \vek{(\psi_1, \w )}{(\psi_2, \w )}, \quad M(z) = \begin{pmatrix}
	(\psi_1, P_\eta a(\chi^{-1}(z)) \varphi_1 ) & (\psi_1, P_\eta a(\chi^{-1}(z)) \varphi_2 ) \\
	( \psi_2, P_\eta a(\chi^{-1}(z)) \varphi_1 ) & ( \psi_2, P_\eta a(\chi^{-1}(z)) \varphi_2 )
	\end{pmatrix}.
\end{align*}
Now $M(0) = I$ and $M(z)$ depends continuously on $z$. Then there exists a zero neighborhood $V \subset \R^2$ such that $M(z)$ is invertible and its inverse depends continuously on $z$. Finally, we obtain for $S_\eta(z)^{-1}$ the representation
\begin{align*}
	S_\eta (z)^{-1} \w = M(z)^{-1} \vek{(\psi_1, \w )}{(\psi_2, \w )},
\end{align*}
which proves our assertion.
\end{proof}

As a consequence of Lemma \ref{Lemma4.5} we obtain from \eqref{eq3} and \eqref{initcond} the $z$-equation
\begin{align} \label{gamma1}
	z_t = r^{[z]} ( z, \w), \quad z(0) = \Pi_\eta^{-1}(P_\eta \v_0),
\end{align}
where $r^{[z]}$ is given by
\begin{align} \label{gammaremainder}
	 r^{[z]}(z,\w):= S_\eta(z)^{-1} P_\eta r^{[f]} (z, \w).
\end{align}
This equation describes the motion of the solution projected onto the group orbit $\mathcal{O}(v_\star)$. The last step is to apply the projector $(I-P_\eta)$ to \eqref{eq2} and using \eqref{gamma1} to obtain the equation for the offset $\w$ from the group orbit:
\begin{align*}
	\w_t &  = \L_\eta \w + (I-P_\eta)r^{[f]} ( z, \w ) - (I-P_\eta)(a(\cdot)\v_\star \circ \chi^{-1} )(z) S_\eta(z)^{-1}P_\eta r^{[f]}( z, \w ) \\
	& =: \L_\eta \w + r^{[w]} ( z, \w )
\end{align*}
with the remainder $r^{[w]}$ given by
\begin{align} \label{wremainder}
	r^{[w]}(z, \w) := \Big((I-P_\eta) - (I-P_\eta)(a(\cdot)\v_\star \circ \chi^{-1} )(z) S_\eta(z)^{-1}P_\eta \Big) r^{[f]} ( z, \w).
\end{align}
Finally, the fully transformed system including initial values for $\w$ and $z$ reads as
\begin{alignat}{2} 
	\w_t & = \L_\eta \w + r^{[w]} ( z, \w), \quad & \quad \w(0) & = \v_0 + \v_\star - a(\Pi_\eta^{-1}(P_\eta \v_0)) \v_\star =: \w_0, \label{wDGL}\\
	z_t & = r^{[z]} ( z, \w),& \quad z(0) & = \Pi_\eta^{-1} ( P_\eta \v_0)=: z_0. \label{gammaDGL}
\end{alignat}
Reversing the steps leading to \eqref{wDGL}, \eqref{gammaDGL} shows
that every local solution of this system leads to a solution of \eqref{CP}
close to $\v_{\star}$ via the transformation \eqref{decomp}.

\sect{Estimates of nonlinearities}
\label{sec6}
To study solutions of the system \eqref{wDGL}, \eqref{gammaDGL} we need to control the remaining nonlinearities $r^{[w]}, r^{[z]}$ from \eqref{wremainder} and \eqref{gammaremainder}. In this section we derive Lipschitz estimates with small Lipschitz constants for the nonlinearities in the space $X^1_\eta$. Of course the estimates will be guaranteed by the smoothness of $f$ from \eqref{perturbsys}. In particular, we can assume $f \in C^3$ by Assumption \ref{A1a}. However, our choice of the underlying space $X_\eta$ requires somewhat laborious calculations to derive the estimates. The main work is to take care of the offset which is hidden in the second component of elements in $X_\eta$.

\begin{lemma} \label{Lemma4.6}
  Let Assumption \ref{A1}, \ref{A2}, \ref{A4} and \ref{A5} be satisfied and let $\mu_1$ be given by Lemma \ref{lemma4.18}. Then for every $0 < \mu \le \mu_1$ there are constant $C = C(\mu) > 0$ and $\delta = \delta(\mu) > 0$ such that for
  all $z,z_1,z_2 \in B_{\delta}(0) \subset \R^2$ and $\w, \w_1, \w_2 \in B_{\delta}(0) \subset X^1_\eta$ the following holds:
\begin{small}
\begin{alignat*}{2}
	& i) & \, & \| r^{[f]}(z, \w_1) - r^{[f]}(z,\w_2) \|_{X^1_\eta} \le C \left( |z| + \max \left\{  \| \w_1 \|_{X^1_\eta}, \| \w_2 \|_{X^1_\eta} \right\} \right) \| \w_1 - \w_2 \|_{X^1_\eta}, \\
	& ii) & \, & \| r^{[f]}(z_1, \w) - r^{[f]}(z_2,\w) \|_{X^1_\eta} \le C|z_1 - z_2|, \\
	& iii)& \, &  \| r^{[w]}(z, \w_1) - r^{[w]}(z, \w_2) \|_{X^1_\eta} \le C \left( |z| + \max \left\{ \| \w_1 \|_{X^1_\eta}, \| \w_2 \|_{X^1_\eta} \right\} \right) \| \w_1 - \w_2 \|_{X^1_\eta}, \\
	& iv)& \, & \| r^{[w]}(z_1, \w_2) - r^{[w]}(z_2, \w_2) \|_{X_\eta^1} \le C \left( |z_1 - z_2| +  \| \w_1 - \w_2\|_{X^1_\eta} \right), \\
	& v) & \, & | r^{[z]}(z_1, \w_1) - r^{[z]}(z_2, \w_2) | \le C \left( |z_1 - z_2| +  \| \w_1 - \w_2 \|_{X^1_\eta} \right).
\end{alignat*}
\end{small}
\end{lemma}
\begin{remark}
  Note that  $r^{[f]}(z,0) = 0$ holds so that the estimates i) and iii)  imply
  linear bounds for the the nonlinearities $r^{[f]}$ and $r^{[w]}$ in $B_\delta(0)$.
\end{remark}
\begin{proof}
  Let $\delta$ be so small such that
  $B_{\delta}(0)\subset \chi(U)$ and $B_\delta(0) \subset V$ with $V$ from Lemma \ref{Lemma4.5}. Then the remainders $r^{[f]}, r^{[w]},r^{[z]}$ are well defined by
  Lemmas \ref{lemma4.18}, \ref{lemmatrafo}, and \ref{Lemma4.5}). Let us set $\gamma = \chi(z) = (\theta,\tau)$ as well as
  $\gamma_i = \chi(z_i)= (\theta_i,\tau_i), i=1,2$. Further we write $\w = (w,\zeta)^\top$ and $\w_i= (w_i,\zeta_i)^\top$ for $i=1,2$. For the sake of notation we also write $a(\gamma) v = R_\theta v(\cdot - \tau)$ for a function $v : \R \rightarrow \R^2$. \\
Throughout the proof, $C = C(\mu)$ denotes a universal constant depending on $\mu$. The smoothness of $f$ and Sobolev embeddings imply
\begin{align} \label{Lemma4.6proof1}
\begin{split}
	& \| Df(a(\gamma) v_\star) - Df(v_\star) \|_{L^\infty}  \le C \| a(\gamma)v_\star - v_\star \|_{L^\infty} \\
	& \le C \| a(\gamma) v_\star -  R_\theta v_\infty \hat{v}  - (v_\star - v_\infty \hat{v} ) \|_{L^\infty} + C |R_\theta v_\infty - v_\infty| \\
	& \le C \| a(\gamma) v_\star - R_\theta v_\infty \hat{v}  - (v_\star - v_\infty \hat{v} ) \|_{H^1} + C |R_\theta v_\infty - v_\infty| \\
	& \le C \| a(\chi^{-1}(z)) \v_\star - \v_\star \|_{X^1_\eta} \le C|z|.
\end{split}  
\end{align}
The last estimate follows from the smoothness of the group action; see Lemma \ref{lemmagroup}. Similarly, we find
\begin{align} \label{Lemma4.6proof1b}
	| Df(R_\theta v_\infty) - Df(v_\infty)| \le C |z|
\end{align}
and
\begin{align} \label{Lemma4.6proof1c}
\begin{split}
	& \| a(\gamma) v_\star - R_\theta v_\infty - (v_\star - v_\infty) \|_{L^2_\eta(\R_+)} \\
	& \le \| a(\gamma) v_\star - R_\theta v_\infty \hat{v} - (v_\star - v_\infty \hat{v}) \|_{L^2_\eta(\R_+)} + |R_\theta v_\infty - v_\infty| \| \hat{v} -1 \|_{L^2_\eta(\R_+)} \\
	& \le C \| a(\chi^{-1}(z))\v_\star - \v_\star \|_{X_\eta} \le C |z|.
\end{split}
\end{align}
By Theorem \ref{decay} we can also estimate
\begin{align} \label{Lemma4.6proof1d}
	\| v_\star - v_\infty \|_{L^2_\eta(\R_+)} \le C. 
\end{align}
In what follows these estimates will be used frequently. We start with \\
\textbf{i).} By definition and the triangle inequality we can split the left side of i) into
\begin{align*}
	& \| r^{[f]}(z, \w_1) - r^{[f]}(z, \w_2) \|_{X^1_\eta} \\
	& \le |f(R_\theta v_\infty + \zeta_1) - f(R_\theta v_\infty + \zeta_2) - Df(v_\infty)(\zeta_1 - \zeta_2)| \\
	& \quad + \| f(a(\gamma) v_\star + w_1) - f(a(\gamma) v_\star + w_2) - Df(v_\star)(w_1-w_2) \\
	& \quad \quad - \hat{v}[f(R_\theta v_\infty + \zeta_1) - f(R_\theta v_\infty + \zeta_2)- Df(v_\infty)(\zeta_1 - \zeta_2)] \|_{L^2_\eta} \\
	& \quad + \| \partial_x[f(a(\gamma) v_\star + w_1) - f(a(\gamma) v_\star + w_2) - Df(v_\star)(w_1-w_2)] \|_{L^2_\eta} \\
	& =: T_1 + T_2 + T_3.
\end{align*}
The first term $T_1$ is estimated by
\begin{align*}
	& T_1 = | f(R_\theta v_\infty + \zeta_1) - f(R_\theta v_\infty + \zeta_2) - Df(v_\infty)(\zeta_1 - \zeta_2)| \\
	& \le \int_{0}^1 |Df(R_\theta v_\infty + \zeta_2 + (\zeta_1 - \zeta_2)s) -Df(v_\infty)| ds |\zeta_1 - \zeta_2| \\
	& \le  \left( \int_0^1 |Df(R_\theta v_\infty + \zeta_2 + (\zeta_1 - \zeta_2)s) -Df(R_\theta v_\infty)| ds + |Df(R_\theta v_\infty) - Df(v_\infty)| \right) |\zeta_1 - \zeta_2|   \\
	& \le C \left( \int_0^1 |\zeta_2 - (\zeta_1 - \zeta_2)s| ds + |R_\theta v_\infty - v_\infty|\right) |\zeta_1 - \zeta_2| \\
	& \le C \left( |z| + \max \{ |\zeta_1|, |\zeta_2| \} \right) |\zeta_1 - \zeta_2| \le C \left( |z| + \max \left\{ \| \w_1 \|_{X^1_\eta}, \|\w_2\|_{X^1_\eta} \right\} \right) \|\w_1 - \w_2\|_{X^1_\eta}.
\end{align*}
For the second term $T_2$ we have
\begin{align*}
	& \| f(a(\gamma) v_\star + w_1) - f(a(\gamma) v_\star + w_2) - Df(v_\star)(w_1-w_2) \\
	& \qquad \qquad -  \hat{v} [f(R_\theta v_\infty + \zeta_1) - f(R_\theta v_\infty + \zeta_2)- Df(v_\infty)(\zeta_1 - \zeta_2)] \|_{L^2_\eta} \\
	& = \Big\| \int_0^1 Df(a(\gamma)v_\star + w_2 +(w_1-w_2)s) - Df(v_\star) ds (w_1 - w_2) \\
	& \qquad \qquad - \hat{v} \int_0^1 Df(R_\theta v_\infty + \zeta_2 + (\zeta_1 - \zeta_2)s) -Df(v_\infty) ds (\zeta_1 - \zeta_2) \Big\|_{L^2_\eta} \\
	& \le \Big\|  \int_0^1 Df(a(\gamma)v_\star + w_2 +(w_1-w_2)s) - Df(a(\gamma) v_\star) ds (w_1 - w_2) \\
	& \qquad \qquad - \hat{v} \int_0^1 Df(R_\theta v_\infty + \zeta_2 + (\zeta_1 - \zeta_2)s) -Df(R_\theta v_\infty) ds (\zeta_1 - \zeta_2) \Big\|_{L^2_\eta} \\
	& \qquad + \| [Df(a(\gamma) v_\star) - Df(v_\star)](w_1 - w_2) - \hat{v}[Df(R_\theta v_\infty) - Df(v_\infty)](\zeta_1 - \zeta_2) \|_{L^2_\eta} \\
	& =: T_4 + T_5.
\end{align*}
$T_5$ is bounded by another two terms
\begin{align*}
	T_5 & \le \| [Df(a(\gamma) v_\star) - Df(v_\star)](w_1 - \hat{v}\zeta_1 - w_2 + \hat{v} \zeta_2) \|_{L^2_\eta} \\
	& \quad + \| [Df(a(\gamma) v_\star) - Df(v_\star) - Df(R_\theta v_\infty) + Df(v_\infty)] (\zeta_1 - \zeta_2) \hat{v} \|_{L^2_\eta} =: T_6 + T_7.
\end{align*}
Using \eqref{Lemma4.6proof1} we have
\begin{align*}
	T_6  \le C |z| \| w_1 - \hat{v} \zeta_1 - w_2 + \hat{v} \zeta_2 \|_{L^2_\eta}  \le C |z| \| \w_1 - \w_2 \|_{X^1_\eta}.
\end{align*}
We bound $T_7$ by two terms, one for the negative and one for the positive half-line:
\begin{align*}
	T_7 & \le \| [Df(a(\gamma) v_\star) - Df(v_\star) - Df(R_\theta v_\infty) + Df(v_\infty)] (\zeta_1 \hat{v} - \zeta_2 \hat{v}) \|_{L^2_\eta(\R_-)} \\
	& \quad + \| [Df(a(\gamma) v_\star) - Df(v_\star) - Df(R_\theta v_\infty) + Df(v_\infty)] (\zeta_1 \hat{v} - \zeta_2 \hat{v}) \|_{L^2_\eta(\R_+)} =: T_8 + T_9
\end{align*}
Now, \eqref{Lemma4.6proof1}, \eqref{Lemma4.6proof1b} imply
\begin{align*}
	T_8 & \le \| Df(a(\gamma)v_\star) - Df(v_\star) - Df(R_\theta v_\infty) + Df(v_\infty) \|_{L^\infty} |\zeta_1 - \zeta_2| \| \hat{v}\|_{L^2_\eta(\R_-)} \\
	& \le C |z| |\zeta_1 - \zeta_2| \le C | z | \| \w_ 1- \w_2 \|_{X^1_\eta}.
\end{align*}
We use the abbreviations $\chi_1(s) := v_\star + s(a(\gamma)v_\star - v_\star)$, $\chi_2(s):=v_\infty + s(R_\theta v_\infty - v_\infty)$, $s \in [0,1]$ and \eqref{Lemma4.6proof1c}, \eqref{Lemma4.6proof1d} to obtain
\begin{align*}
	T_9  & = \|  [Df(a(\gamma)v_\star) - Df(v_\star) - Df(R_\theta v_\infty) + Df(v_\infty)](\zeta_1 - \zeta_2)\hat{v} \|_{L^2_\eta(\R_+)}   \\
	& \le \Big\| \int_0^1 D^2f(\chi_1(s))[a(\gamma)v_\star - v_\star,(\zeta_1 - \zeta_2)\hat{v}] ds  \\
	& \qquad - \int_0^1 D^2f(\chi_2(s))[R_\theta v_\infty - v_\infty,(\zeta_1 - \zeta_2)\hat{v}] ds \Big\|_{L^2_\eta(\R_+)}\\
	& \le \Big\| \int_0^1 D^2f(\chi_1(s))[a(\gamma)v_\star - v_\star - R_\theta v_\infty + v_\infty, (\zeta_1 - \zeta_2)\hat{v}] ds \Big\|_{L^2_\eta(\R_+)} \\
	& \quad \quad + \Big\| \int_0^1 [D^2f(\chi_1(s))-D^2f(\chi_2(s))][R_\theta v_\infty + v_\infty, (\zeta_1 - \zeta_2)\hat{v}] ds \Big\|_{L^2_\eta(\R_+)} \\
	& \le C \Big( \| a(\gamma)v_\star - R_\theta v_\infty - (v_\star - v_\infty) \|_{L^2_\eta(\R_+)} \\
	& \qquad +  \Big\| \int_0^1 \chi_1(s) - \chi_2(s) ds \Big\|_{L^2_\eta(\R_+)} |R_\theta v_\infty - v_\infty| \Big) |\zeta_1 - \zeta_2| \\
	& \le C |z| \| \w_1 - \w_2 \|_{X^1_\eta}.
\end{align*}
To estimate $T_4$ we use the abbreviations $w(s) := w_2 + (w_1-w_2)s$, $\zeta(s) := \zeta_2 + (\zeta_1-\zeta_2)s$, $s \in [0,1]$ and obtain
\begin{align*}
	T_4 & = \Big\|  \int_0^1 Df(a(\gamma)v_\star + w(s)) - Df(a(\gamma) v_\star) ds (w_1 - w_2) \\
	& \qquad \qquad - \hat{v} \int_0^1 Df(R_\theta v_\infty + \zeta(s)) -Df(R_\theta v_\infty) ds (\zeta_1 - \zeta_2) \Big\|_{L^2_\eta} \\
	& \le \Big\| \int_0^1 Df(a(\gamma) v_\star + w(s)) - Df(a(\gamma) v_\star) ds (w_1 - \zeta_1 \hat{v} - w_2 + \zeta_2 \hat{v}) \Big\|_{L^2_\eta} \\
	& \qquad \qquad + \Big\| \int_0^1 Df(a(\gamma) v_\star + w(s)) - Df(a(\gamma) v_\star)  \\
	& \qquad \qquad \qquad \qquad - Df(R_\theta v_\infty + \zeta(s)) + Df(R_\theta v_\infty) ds (\zeta_1 -\zeta_2) \hat{v}\Big\|_{L^2_\eta} \\
	& =: T_{10} + T_{11}.
\end{align*}
Now for every $s \in [0,1]$ we have
\begin{align} \label{Lemma4.6proof1e}
\begin{split}
	& \| Df(a(\gamma) v_\star + w(s)) - Df(a(\gamma) v_\star) \|_{L^\infty} \le C \| w_2 + s(w_1 - w_2) \|_{L^\infty} \\
	& \le C \max \left\{ \| w_1\|_{L^\infty}, \|w_2 \|_{L^\infty}\right\} \le C \max \left\{ \| \w_1\|_{X^1_\eta}, \|\w_2 \|_{X^1_\eta} \right\},
\end{split}
\end{align}
where we used that the Sobolev embedding implies for $i \in \{1,2\}$
\begin{align*}
	\| w_i \|_{L^\infty} \le \| w_i -\zeta_i \hat{v} \|_{L^\infty} + |\zeta_i| & \le C \| w_i -\zeta_i \hat{v} \|_{H^1} + |\zeta_i| \le C \| \w_i \|_{X^1_\eta}.
\end{align*}
Then \eqref{Lemma4.6proof1e} yields
\begin{align*}
	T_{10} & \le \int_0^1 \| Df(a(\gamma) v_\star + w(s)) -  Df(a(\gamma) v_\star) \|_{L^\infty} ds \| w_1 - \zeta_1 \hat{v} - w_2 + \zeta_2 \hat{v} \|_{L^2_\eta} \\
	& \le C \max \left\{ \| \w_1 \|_{X^1_\eta}, \| \w_2 \|_{X^1_\eta} \right\} \| \w_1 - \w_2 \|_{X^1_\eta}.
\end{align*}
Further,
\begin{align*}
	 T_{11} & \le \Big\| \int_0^1 Df(a(\gamma) v_\star + w(s)) - Df(a(\gamma) v_\star) \\
	 & \qquad \qquad - Df(R_\theta v_\infty + \zeta(s)) + Df(R_\theta v_\infty) ds (\zeta_1 -\zeta_2) \hat{v}\Big\|_{L^2_\eta(\R_-)} \\
	& + \Big\| \int_0^1 Df(a(\gamma) v_\star + w(s)) - Df(a(\gamma) v_\star) \\
	& \qquad \qquad - Df(R_\theta v_\infty + \zeta(s)) + Df(R_\theta v_\infty) ds (\zeta_1 -\zeta_2) \hat{v}\Big\|_{L^2_\eta(\R_+)} \\
	& =: T_{12} + T_{13}.
\end{align*}
We write $\kappa(s) := a(\gamma)v_\star + w(s) - R_\theta v_\infty - \zeta(s)$. Then for $s \in [0,1]$ there holds
\begin{align*}
	& \Big\| Df( a(\gamma) v_\star + w(s)) - Df(a(\gamma) v_\star) - Df(R_\theta v_\infty + \zeta(s)) + Df(R_\theta v_\infty) \Big\|_{L^\infty} \\
	& = \Big\| \int_0^1 D^2f(R_\theta v_\infty+ \zeta(s) + \kappa(s) \tau)[\kappa(s), \cdot] \\
	& \qquad \qquad - D^2f(R_\theta  v_\infty + (a(\gamma) v_\star - R_\theta v_\infty)\tau ) [a(\gamma) v_\star - R_\theta v_\infty, \cdot] d\tau \Big\|_{L^\infty} \\
	& \le \Big\| \int_0^1 D^2f(R_\theta v_\infty+ \zeta(s) + \kappa(s) \tau)[w(s)-\zeta(s), \cdot]  d\tau \Big\|_{L^\infty} \\
	& \qquad + \Big\| \int_0^1 \big( D^2f(R_\theta v_\infty+ \zeta(s) + \kappa(s) \tau) \\
	& \qquad \qquad - D^2f(R_\theta  v_\infty + (a(\gamma) v_\star - R_\theta v_\infty)\tau \big)[a(\gamma)v_\star - R_\theta v_\infty, \cdot] d\tau \Big\|_{L^\infty} \\
	& \le C \| w(s) - \zeta(s) \|_{L^\infty} + C  \int_0^1 \| \zeta(s) - (w(s) - \zeta(s)) \tau \|_{L^\infty} d\tau \le C \max \left\{ \| w_1 \|_{L^\infty}, \| w_2 \|_{L^\infty} \right\} \\
	& \le C \max \left\{ \| \w_1 \|_{X^1_\eta}, \| \w_2 \|_{X^1_\eta} \right\},
\end{align*}
where we used $|\zeta_i| \le \| w_i \|_{L^\infty}$, $i= 1,2$. So we conclude
\begin{align*}
	T_{12} \le C \max \left\{ \| \w_1 \|_{X^1_\eta}, \| \w_2 \|_{X^1_\eta} \right\} |\zeta_1 - \zeta_2|.
\end{align*}
Similarly, for every $s \in [0,1]$,
\begin{align*}
	& \Big\| Df( a(\gamma) v_\star + w(s)) - Df(a(\gamma) v_\star) - Df(R_\theta v_\infty + \zeta(s)) + Df(R_\theta v_\infty) \Big\|_{L^2_\eta(\R_+)} \\
	& \le C \| w(s) - \zeta(s) \|_{L^2_\eta(\R_+)} + C  \int_0^1 \| \zeta(s) - (w(s) - \zeta(s)) \tau \|_{L^\infty} d\tau \| a(\gamma)v_\star - R_\theta v_\infty \|_{L^2_\eta(\R_+)} \\
	& \le C \max \left\{ \| w_1- \zeta_1 \|_{L^2_\eta(\R_+)}, \| w_2- \zeta_2 \|_{L^2_\eta(\R_+)} \right\} + C \max \left\{ \| w_1 \|_{L^\infty}, \| w_2 \|_{L^\infty} \right\} \\
	& \le C \max \left\{ \| \w_1 \|_{X^1_\eta}, \| \w_2 \|_{X^1_\eta} \right\}.
\end{align*}
This yields the estimate for $T_{13}$
\begin{align*}
	T_{13} \le C \max \left\{ \| \w_1 \|_{X^1_\eta}, \| \w_2 \|_{X^1_\eta} \right\} |\zeta_1 - \zeta_2|.
\end{align*}
Summarizing, we have shown
\begin{align*}
	T_2 & = \| f(a(\gamma) v_\star + w_1) - f(a(\gamma) v_\star + w_2) - Df(v_\star)(w_1-w_2) \\
	& \quad \qquad - \hat{v}(R_\theta v_\infty + \zeta_1) - f(R_\theta v_\infty + \zeta_2)- Df(v_\infty)(\zeta_1 - \zeta_2)] \hat{v}  \|_{L^2_\eta} \\
	& \le C \left( |z| + \max \left\{  \| \w_1 \|_{X^1_\eta}, \| \w_2 \|_{X^1_\eta} \right\} \right) \| \w_1 - \w_2 \|_{X^1_\eta}.
\end{align*}
It remains to estimate the derivative given by $T_3$. We have
\begin{align*}
	T_3 & = \left\| \partial_x \Big[ f(a(\gamma) v_\star + w_1) - f(a(\gamma) v_\star + w_2) - Df(v_\star)(w_1 - w_2) \Big] \right\|_{L^2_\eta} \\
	& = \| Df(a(\gamma) v_\star + w_1)w_{1,x} + Df(a(\gamma) v_\star + w_1) a(\gamma)v_{\star,x} \\
	& \quad - Df(a(\gamma) v_\star + w_2)w_{2,x} - Df(a(\gamma) v_\star + w_2) a(\gamma)v_{\star,x} \\
	& \quad - D^2f(v_\star)[w_1 - w_2, v_{\star,x}] - Df(v_\star)(w_1-w_2)_x \|_{L^2_\eta} \\
	& = \left\| [ Df(a(\gamma)v_\star + w_1) - Df(a(\gamma) v_\star + w_2) ] w_{1,x} \right\|_{L^2_\eta} \\
	& \quad + \left\| [ Df(a(\gamma) v_\star + w_2) - Df(v_\star) ] (w_1 - w_2)_x \right\|_{L^2_\eta} \\
	& \quad + \left\| [ Df(a(\gamma) v_\star + w_1) - Df(a(\gamma) v_\star + w_2) ] (a(\gamma) v_\star - v_\star)_x \right\|_{L^2_\eta} \\
	& \quad +  \left\| [ Df(a(\gamma) v_\star + w_1) - Df(a(\gamma) v_\star + w_2) ] v_{\star,x} - D^2f(v_\star) [w_1 - w_2, v_{\star,x}] \right\|_{L^2_\eta} \\
	& = I_1 + I_2 + I_3 + \left\| \int_0^1 D^2f(a(\gamma) v_\star + w_2 + (w_1 - w_2)s) - D^2f(v_\star) ds [w_1 - w_2, v_{\star,x}] \right\|_{L^2_\eta} \\
	& = I_1 + I_2 + I_3 +I_4.
\end{align*}
Now
\begin{align*}
	I_1 & \le \| Df(a(\gamma) v_\star + w_1) - Df(a(\gamma) v_\star + w_2) \|_{L^\infty} \| w_{1,x} \|_{L^2_\eta} \\
	& \le C \| w_1 - w_2 \|_{L^\infty} \| w_{1,x} \|_{L^2_\eta} \le C \max \left\{ \| \w_1 \|_{X^1_\eta}, \| \w_2 \|_{X^1_\eta} \right\} \| \w_1 - \w_2 \|_{X^1_\eta}.
\end{align*}
In the same fashion we obtain
\begin{align*}
	I_2  & \le \| Df(a(\gamma) v_\star + w_2) - Df(v_\star) \|_{L^\infty} \| (w_1 - w_2)_x \|_{L^2_\eta}  \\
	& \le C \left( \| a(\gamma) v_\star - v_\star \|_{L^\infty} + \| w_2 \|_{L^\infty} \right)  \| \w_1 - \w_2 \|_{X^1_\eta} \\
	& \le C \left( |z| + \max \left\{ \| \w_1 \|_{X^1_\eta}, \| \w_2 \|_{X^1_\eta} \right\} \right) \| \w_1 - \w_2 \|_{X^1_\eta}
\end{align*}
and for $I_3$,
\begin{align*}
	I_3  & \le \| Df(a(\gamma) v_\star + w_1) - Df(a(\gamma)v_\star + w_2) \|_{L^\infty} \| a(\gamma) v_{\star,x} - v_{\star,x} \|_{L^2_\eta} \\
	& \le C \| w_1 - w_2 \|_{L^\infty} \| a(\gamma) v_{\star,x} - v_{\star,x} \|_{L^2_\eta} \le C |z| \| \w_1 - \w_2 \|_{X^1_\eta}. 
\end{align*}
For $I_4$ we have
\begin{align*}
	I_4 & \le C \left( \| a(\gamma) v_\star - v_\star\|_{L^\infty} + \max\{ \| w_1\|_{L^\infty}, \| w_2\|_{L^\infty} \}\right) \| w_1 - w_2 \|_{L^\infty} \| v_{\star,x} \|_{L^2_\eta} \\
	& \le  C \left( |z| + \max \left\{ \| \w_1 \|_{X^1_\eta}, \| \w_2 \|_{X^1_\eta} \right\} \right) \| \w_1 - \w_2 \|_{X^1_\eta}.
\end{align*}
Hence
\begin{align*}
	T_3 & = \left\| \Big[ f(a(\gamma) v_\star + w_1) - f(a(\gamma) v_\star + w_2) - Df(v_\star)(w_1 - w_2) \Big]_x \right\|_{L^2_\eta} \\
	& \le C \left( |z| + \max \left\{ \| \w_1 \|_{X^1_\eta}, \| \w_2 \|_{X^1_\eta} \right\} \right) \| \w_1 - \w_2 \|_{X^1_\eta}.
\end{align*}
Finally we have shown
\begin{align*}
	& \left\| r^{[f]}(z, \w_1) - r^{[f]}(z, \w_2) \right\|_{X^1_\eta} \le C \left( |z| + \max \left\{ \| \w_1 \|_{X^1_\eta}, \| \w_2 \|_{X^1_\eta} \right\} \right) \| \w_1 - \w_2 \|_{X^1_\eta}.
\end{align*}
\textbf{ii).} As in i) we frequently use the mean value theorem and the smoothness of $f$ which follows from Assumption \ref{A1}. First, we estimate
\begin{align*}
	 & \| r^{[f]}(z_1, \w) - r^{[f]}(z_2, \w) \|_{X^1_\eta} \\
	 & = \left\| \vek{f(a(\gamma_1)v_\star + w) - f(a(\gamma_1) v_\star) - f(a(\gamma_2) v_\star + w) - f(a(\gamma_2)v_\star)}{f(R_{\theta_1}v_\infty + \zeta) - f(R_{\theta_1} v_\infty) - f(R_{\theta_2} v_\infty + \zeta) - f(R_{\theta_2}v_\infty)} \right\|_{X^1_\eta} \\
	 & \le | f(R_{\theta_1} v_\infty + \zeta) - f(R_{\theta_2} v_\infty + \zeta) | + | f(R_{\theta_1} v_\infty) - f(R_{\theta_2} v_\infty) | \\
	 & \quad + \| f(a(\gamma_1)v_\star + w) - f(a(\gamma_2)v_\star + w) - \hat{v} [f(R_{\theta_1}v_\infty+\zeta) - f(R_{\theta_2}v_\infty + \zeta)] \|_{L^2_\eta}\\
	 & \quad + \| f(a(\gamma_1)v_\star) - f(a(\gamma_2)v_\star) - \hat{v}(f(R_{\theta_1} v_\infty) - f(R_{\theta_2}v_\infty)) \|_{L^2_\eta} \\
	 & \quad + \| \partial_x [f(a(\gamma_1) v_\star + w) - f(a(\gamma_2)v_\star + w)]\|_{L^2_\eta} + \| \partial_x[f(a(\gamma_1) v_\star) - f(a(\gamma_2)v_\star)]\|_{L^2_\eta} \\
	 & =: J_1 + J_2 + J_3 + J_4 + J_5 + J_6.
\end{align*}
The smoothness of $f$ implies
\begin{align*}
	J_1  & = | f(R_{\theta_1}v_\infty +\zeta) - f(R_{\theta_2} v_\infty + \zeta)| \le C |R_{\theta_1}v_\infty - R_{\theta_2}v_\infty| \le C |z_1 - z_2|.  
\end{align*}
The same holds true for $\zeta = 0$ so that $J_2 \le C |z_1 - z_2|$. Write $\kappa_1(s) := a(\gamma_2) v_\star + w +(a(\gamma_1)v_\star - a(\gamma_2)v_\star)s$ and $\kappa_2(s) := R_{\theta_2}v_\infty + \zeta + (R_{\theta_1} v_\infty - R_{\theta_2}v_\infty)s$, $s \in [0,1]$ and obtain for $J_3$,  
\begin{align*}
	J_3 & = \| f(a(\gamma_1)v_\star + w) - f(a(\gamma_2)v_\star + w) - \hat
	v[f(R_{\theta_1}v_\infty + \zeta) + f(R_{\theta_2}v_\infty + \zeta) ] \|_{L^2_\eta} \\
	&  = \Big\| \int_0^1 Df(a(\gamma_2) v_\star + w +(a(\gamma_1)v_\star - a(\gamma_2)v_\star)s) (a(\gamma_1) v_\star - a(\gamma_2) v_\star) ds \\
	& \quad - \hat{v} \int_0^1 Df(R_{\theta_2} v_\infty + \zeta +(R_{\theta_1}v_\infty - R_{\theta_2}v_\infty)s)(R_{\theta_1} v_\infty - R_{\theta_2} v_\infty ) ds \Big\|_{L^2_\eta} \\
	& \le \Big\| \int_0^1 Df(\kappa_1(s)) (a(\gamma_1)v_\star - R_{\theta_1} v_\infty \hat{v} - a(\gamma_2)v_\star + R_{\theta_2} v_\infty \hat{v}) ds \Big\|_{L^2_\eta} \\
	& \quad + \Big\| \int_0^1 [ Df(\kappa_1(s)) - Df(\kappa_2(s)) ](R_{\theta_1} v_\infty \hat{v} - R_{\theta_2}v_\infty \hat{v})ds \Big\|_{L^2_\eta} =: J_7 + J_8.
\end{align*}
We estimate $J_7$ by
\begin{align*}
	J_7 & \le C \| a(\gamma_1) v_\star - R_{\theta_1} v_\infty \hat{v} - a(\gamma_2) v_\star + R_{\theta_2} v_\infty \hat{v} \|_{L^2_\eta} \\
	& \le C \| a(\chi^{-1}(z_1)) \v_\star - a(\chi^{-1}(z_2)) \v_\star \|_{X_\eta} \\
	& \le C |z_1 - z_2|
\end{align*}
and bound $J_8$ by two terms
\begin{align*}
	J_8 & \le  \Big\| \int_0^1 [ Df(\kappa_1(s)) - Df(\kappa_2(s)) ](R_{\theta_1} v_\infty \hat{v} - R_{\theta_2}v_\infty \hat{v})ds \Big\|_{L^2_\eta(\R_-)} \\
	& \quad + \Big\| \int_0^1 [ Df(\kappa_1(s)) - Df(\kappa_2(s)) ](R_{\theta_1} v_\infty \hat{v} - R_{\theta_2}v_\infty \hat{v})ds \Big\|_{L^2_\eta(\R_+)} = J_9 + J_{10}.
\end{align*}
Then
\begin{align*}
	J_9 \le C \| \hat{v} \|_{L^2_\eta(\R_-)} |R_{\theta_1} v_\infty - R_{\theta_2} v_\infty | \le C | z_1 - z_2|
\end{align*}
and for $J_{10}$
\begin{align*}
	J_{10} & \le C |R_{\theta_1} v_\infty - R_{\theta_2} v_\infty| \int_0^1 \|\kappa_1(s)  - \kappa_2(s) \|_{L^2_\eta} ds \le C |z_1 - z_2|.
\end{align*}
Thus we have shown $J_3 \le C |z_1 - z_2|$. In particular the estimates hold for $w = 0$, $\zeta = 0$. Therefore we also have shown $J_4 \le C |z_1 - z_2|$ and it remains to estimate the spatial derivatives $J_5$ and $J_6$. We note that for arbitrary $u \in L^2_\eta$ we have by Sobolev embedding
\begin{align*}
	& \| [Df(a(\gamma_1) v_\star + w) - Df(a(\gamma_2) v_\star + w)] u \|_{L^2_\eta} \le C \| a(\gamma_1) v_\star - a(\gamma_2)v_\star \|_{L^\infty} \|u \|_{L^2_\eta} \\
	& \le C \| u \|_{L^2_\eta} \left( \| a(\gamma_1)v_\star - R_{\theta_1} v_\infty \hat{v} - a(\gamma_2)v_\star + R_{\theta_2 }v_\infty \hat{v} \|_{L^\infty} + \| R_{\theta_1} v_\infty \hat{v} - R_{\theta_2} v_\infty \hat{v}\|_{L^\infty} \right) \\
	& \le C \| u \|_{L^2_\eta} |z_1 - z_2|.
\end{align*}
This implies
\begin{align*}
	J_5 & \le \| [Df(a(\gamma_1) v_\star + w) - Df(a(\gamma_2) v_\star + w)] w_x \|_{L^2_\eta} \\
	& \quad + \| Df(a(\gamma_1) v_\star + w)a(\gamma_1)v_{\star,x} - Df(a(\gamma_2) v_\star + w) a(\gamma_2)v_{\star,x} \|_{L^2_\eta} \\
	& \le C \| w_x \|_{L^2_\eta}|z_1 -z_2| + \| [Df(a(\gamma_1) v_\star + w)- Df(a(\gamma_2) v_\star + w)] a(\gamma_1)v_{\star,x} \|_{L^2_\eta} \\
	& \quad+ C \| a(\gamma_1)v_{\star,x} - a(\gamma_2)v_{\star,x} \|_{L^2_\eta} \\
	& \le C \left( \| w_x \|_{L^2_\eta} + \| a(\gamma_1) v_{\star,x} \|_{L^2_\eta} \right) |z_1 - z_2| + C \| a(\gamma_1)v_{\star,x} - a(\gamma_2) v_{\star,x} \|_{L^2_\eta} \le C |z_1 - z_2|.
\end{align*}
In particular the same holds true for $w = 0$ and we observe $J_6 \le C |z_1 - z_2|$. Summarizing, we have shown
\begin{align*}
	\| r^{[f]}(z_1, \w) - r^{[f]}(z_2, \w) \|_{X^1_\eta} \le C_1 |z_1 - z_2|.
\end{align*}
\textbf{iii).} Since the group action is smooth and since $P_\eta$ from \eqref{projector} is a projector we have
\begin{align*}
	\Big\| \Big((I-P_\eta) - (I-P_\eta) \big(a(\cdot)\v_\star \circ \chi^{-1}\big)(z) S_\eta(z)^{-1}P_\eta \Big) \u \Big\|_{X^1_\eta} \le C \| \u \|_{X^1_\eta} \quad \forall \u \in X^1_\eta.
\end{align*}
Now the claim follows from i). \\
\textbf{iv).} By the smoothness of the group action and Lemma \ref{Lemma4.5} the
function $\big(a(\cdot) \v_\star \circ \chi^{-1}\big)(z) S_\eta (z)^{-1}$ is continuously differentiable in $z$. Therefore, we obtain the local Lipschitz estimate
\begin{align} \label{Lemma4.6proof1f}
	\| \big(a(\cdot) \v_\star \circ \chi^{-1} \big)(z_1)S_\eta(z_1)^{-1}\w - \big(a(\cdot) \v_\star \circ \chi^{-1} \big)(z_2) S_\eta(\gamma_2)^{-1}\w \|_{X^1_\eta} \le C |z_1 - z_2 | \| \w \|_{X^1_\eta}.
\end{align}
Then we use \eqref{Lemma4.6proof1f} and i) to see
\begin{align*}
	& \| r^{[w]}(z_1,\w) - r^{[w]}(z_2,\w) \|_{X^1_\eta} \\
	& \le C \| r^{[f]}(z_1,\w) - r^{[f]}(z_2,\w) \|_{X^1_\eta} \\
	& \quad + \| \big(a(\cdot) \v_\star \circ \chi^{-1}\big)(z_1)S_\eta(z_1)^{-1}P_\eta r^{[f]}(z_1,\w) - \big(a(\cdot) \v_\star \circ \chi^{-1}\big)(z_2)S_\eta(z_2)^{-1}P_\eta r^{[f]}(z_1,\w) \|_{X^1_\eta}  \\
	& \le C |z_1 - z_2|.
\end{align*}
Now we obtain using ii) and iii)
\begin{align*}
	& \| r^{[w]}(z_1,\w_1) - r^{[w]}(z_2,\w_2) \|_{X^1_\eta} \\
	& \le \| r^{[w]}(z_1,\w_1) - r^{[w]}(z_2,\w_1) \|_{X^1_\eta} + \| r^{[w]}(z_2,\w_1) - r^{[w]}(z_2,\w_2) \|_{X^1_\eta} \\
	& \le C \left( |z_1 - z_2| + \| \w_1 - \w_2 \|_{X^1_\eta} \right).
\end{align*}
\textbf{v).} Similar to iv) we have by Lemma \ref{Lemma4.5} that $S_\eta(z)^{-1}$ is locally Lipschitz w.r.t. $z$. Then we obtain
\begin{align*}
	& \left| r^{[z]}(z_1, \w_1) - r^{[z]}(z_2, \w_2) \right| = \left| S_\eta(z_1)^{-1} P_\eta r^{[f]}(z_1, \w_1) - S_\eta(z_2)^{-1} P_\eta r^{[f]}(z_2, \w_2) \right| \\
	& \le  C \left\| r^{[f]}(z_1,\w_1) - r^{[f]}(z_2,\w_2) \right\|_{X^1_\eta} + \left| (S_\eta(z_1)^{-1} - S_\eta(z_2)^{-1}) P_\eta r^{[f]}(z_2,\w_2) \right| \\
	& \le C_4 \left( |z_1 - z_2| +  \| \w_1 - \w_2 \|_{X^1_\eta} \right).
\end{align*}
\end{proof}

\sect{Nonlinear stability}
\label{sec7}
In this section we complete the proof of the main Theorem \ref{Theorem4.10} according to the following strategy. For sufficiently small initial perturbation
$\v_0$ in \eqref{CP} we show existence of a local mild solution of the
corresponding integral equations of the decomposed system \eqref{wDGL}, \eqref{gammaDGL} which reads as
\begin{align} 
	\w(t) & = e^{t\L_\eta} \w_0 + \int_0^t e^{(t-s)\L_\eta} r^{[w]}(z(s),\w(s)) ds, \label{integralwDGL}\\
	z(t) & = z_0 + \int_0^t r^{[z]} ( z(s), \w(s)) ds. \label{integralgammaDGL}
\end{align}
A Gronwall estimate then shows that the solution exists for all times, that
the perturbation $\w$ decays exponentially and that $z$ converges to the
coordinates of an asymptotic phase. Combining the results with the regularity
theory for mild solutions we infer
Theorem \ref{Theorem4.10} and thus nonlinear stability of traveling oscillating fronts. 

\begin{lemma} \label{Lemma4.7}
Let Assumption \ref{A1}, \ref{A2}, \ref{A4} and \ref{A5} be satisfied and let $0 < \mu \le \min(\mu_0,\mu_1)$ with $\mu_0$ from Theorem \ref{thm4.17} and $\mu_1$ from Lemma \ref{lemma4.18}. Then for every $0 < \varepsilon_1 < \delta$ and $0 < 2K \varepsilon_0 \le \delta$ with $K$ from Theorem \ref{semigroup} and $\delta$ from Lemma \ref{Lemma4.6}, there exists $t_\star = t_\star(\varepsilon_0,\varepsilon_1, \mu) > 0$ such that for all initial values $(z_0,\w_0) \in \R^2 \times V^1_\eta$ with
\begin{align*}
	\| \w_0 \|_{X^1_\eta} < \varepsilon_0, \quad |z_0| < \varepsilon_1
\end{align*}
there exists a unique solution $(z,\w) \in C([0,t_\star), \R^2 \times V^1_\eta)$ of \eqref{integralwDGL}, \eqref{integralgammaDGL} with
\begin{align*}
	\| \w(t) \|_{X^1_\eta} \le 2K\varepsilon_0, \quad |z(t)| \le 2 \varepsilon_1, \quad t \in [0,t_\star).
\end{align*} 
In particular, $t_\star$ is independent of  $(z_0,\w_0) \in B_{\varepsilon_1}(0) \times B_{\varepsilon_0}(0)$.
\end{lemma}

\begin{proof}
Take $\nu = \nu(\mu) > 0$ from Theorem \ref{semigroup}, $C = C(\mu) > 0$ from Lemma \ref{Lemma4.6} and let $t_\star$ be so small such that the following conditions are satisfied:
\begin{align} \label{condtstar}
t_\star < \frac{\varepsilon_1}{2C ( \varepsilon_1 + K \varepsilon_0 )}, \quad  t_\star + \frac{2K}{\nu} (1- e^{-\nu t_\star}) < \frac{1}{C}.
\end{align}
The proof employs a contraction argument in the space $Z := C([0,t_\star),\R^2 \times V^1_\eta)$ equipped with the supremums norm $\| (z,\w) \|_{Z} := \sup_{t \in [0,t_\star)} \{ |z(t)| + \| \w(t) \|_{X^1_\eta} \}$. Define the map $\Gamma: Z \rightarrow Z$ given by the right hand side of \eqref{integralwDGL}, \eqref{integralgammaDGL}. We show that $\Gamma$ is a contraction on the closed set
\begin{align*}
	B := \{ (z,\w) \in Z: \| \w(t) \|_{X^1_\eta} \le 2K\varepsilon_0,\, |z(t)| \le 2\varepsilon_1,\, t\in [0,t_\star) \} \subset Z.
\end{align*}
Let $(z,\w) \in B$. By using the estimates from Theorem \ref{semigroup}, Lemma \ref{Lemma4.6} and \eqref{condtstar} we obtain for all $0 \le t < t_\star$
\begin{align*}
	& \left\| e^{t\L_\eta} \w_0 + \int_0^t e^{(t-s)\L_\eta} r^{[w]}(z(s),\w(s)) ds \right\|_{X^1_\eta} \\
	& \le K e^{-\nu t} \varepsilon_0 + K \int_0^t e^{-\nu (t-s)} \| r^{[w]}(z(s),\w(s)) \|_{X^1_\eta} ds  \\
	& \le K e^{-\nu t} \varepsilon_0 + K C \int_0^t e^{-\nu (t-s)} \| \w(s) \|_{X^1_\eta} ds \\
	& \le K\varepsilon_0 + \frac{2K^2 C \varepsilon_0}{\nu}(1 - e^{-\nu t_\star}) \le 2K\varepsilon_0
\end{align*}
and
\begin{align*}
	\left| z_0 + \int_0^t r^{[z]}(z(s),\w(s))ds \right| & \le \varepsilon_1 + \int_0^t |r^{[z]}(z(s),\w(s))| ds \\
	& \le \varepsilon_1 + C \int_0^t |z(s)| + \| \w(s) \|_{X^1_\eta} ds  \\
	& \le \varepsilon_1 + 2C(\varepsilon_1 + K\varepsilon_0) t_\star \le 2 \varepsilon_1.
\end{align*}
Hence $\Gamma$ maps $B$ into itself. Further, for $(z_1,\w_1),(z_2,\w_2) \in B$ and $0 \le t < t_\star$ we can estimate
\begin{align*}
	& \| \Gamma(z_1,\w_1) - \Gamma(z_2,\w_2) \|_{Z} \\
	& \le \sup_{t \in [0,t_\star)} \Big\{ \int_0^t | r^{[z]}(z_1(s),\w_1(s)) - r^{[z]}(z_2(s),\w_2(s))|ds \\
	& \qquad + \int_0^t  K e^{-\nu (t-s)} \| r^{[w]}(z_1(s),\w_1(s)) - r^{[w]}(z_2(s),\w_2(s)) \|_{X^1_\eta} ds \Big\} \\
	& \le \Big( C t_\star + \frac{KC}{\nu } (1- e^{-\nu t_\star}) \Big) \| (z_1-z_2, \w_1-\w_2)\|_{Z}.  
\end{align*}
By condition \eqref{condtstar}, the map $\Gamma$ is a contraction on $B$
and the assertion follows from the contraction mapping theorem.
\end{proof}

We use the following Gronwall lemma from \cite[Lemma 6.3]{BeynLorenz}.

\begin{lemma} \label{Gronwall}
Suppose $\varepsilon, \nu, C,\tilde{C} > 0$ such that
\begin{align*}
	C \ge 1, \quad \varepsilon \le \frac{\nu}{16 \tilde{C} C}
\end{align*}
and let $\varphi \in C([0,t_\infty),[0,\infty))$ for some $0 < t_\infty \le \infty$ satisfying
\begin{align*}
	\varphi(t) \le C \varepsilon e^{-\nu t} + \tilde{C} \int_0^t e^{-\nu(t-s)} \left( \varphi(s)^2 + \varepsilon \varphi(s) \right) ds, \quad \forall t \in [0,t_\infty).
\end{align*}
Then for all $0 \le t < t_\infty$ there holds
\begin{align*}
	\varphi(t) \le 2C\varepsilon e^{-\frac{3}{4} \nu t}.
\end{align*}
\end{lemma}

Next we prove the stability result for the $(z,\w)$-systems
\eqref{integralwDGL}, \eqref{integralgammaDGL} and \eqref{wDGL}, \eqref{gammaDGL}.
The Gronwall estimate ensures that the solution from Lemma \ref{Lemma4.7} can not reach the boundary of the region of existence and therefore exists for all times. Moreover, the perturbation $\w$ of the TOF decays exponentially.
Regularity of the solution will follow by standard results from \cite{Amann} and \cite{Henry}. As in \cite{Amann}, we denote by $C^\alpha$, $\alpha \in (0,1)$ the space of H\"older continuous functions and by $C^{1 + \alpha}$ the space of differentiable functions with H\"older continuous derivative. 

\begin{theorem} \label{Theorem4.9}
Let Assumption \ref{A1}, \ref{A2}, \ref{A4} and \ref{A5} be satisfied and let $0 < \mu \le \min(\mu_0,\mu_1)$ with $\mu_0$ from Theorem \ref{thm4.17} and $\mu_1$ from Lemma \ref{lemma4.18}. Then there are $\varepsilon(\mu), \beta(\mu) > 0$ such that for all initial values $(z_0, \w_0) \in \R^2 \times V^2_\eta$ with $\| (z_0, \w_0)\|_{\R^2 \times X_\eta^1} < \varepsilon$ the following statements hold:
\begin{enumerate}[i)]
\item There are unique
\begin{align*}
	\w  \in C^{\alpha}((0,\infty),V^2_\eta) \cap C^{1+\alpha}((0,\infty),V_\eta) \cap C^1([0,\infty),V_\eta), \quad z \in C^1([0,\infty),\R^2), 
\end{align*}
for arbitrary $\alpha \in (0,1)$, satisfying \eqref{wDGL} in $X_\eta$ and \eqref{gammaDGL} in $\R^2$.
\item There exist $z_\infty = z_\infty(z_0,\w_0) \in \R^2$ and $K_0 = K_0(\mu) \ge 1$ such that for all $t \ge 0$
\begin{align*}
	\| \w(t) \|_{X_\eta^1} + |z(t) - z_\infty| \le K_0 e^{-\beta t} \|(z_0,\w_0) \|_{\R^2 \times X^1_\eta}, \quad |z_\infty| \le (K_0+1) \| (z_0,\w_0)  \|_{\R^2 \times X^1_\eta}.
\end{align*}
\end{enumerate}
\end{theorem}

\begin{proof}
Recall the constants $K = K(\mu), \nu = \nu(\mu)$ from Theorem \ref{semigroup} and $C = C(\mu), \delta = \delta(\mu)$ from Lemma \ref{Lemma4.6}. We choose $C_0, \varepsilon, \tilde{\varepsilon}> 0$ such that $0 < 2K \tilde{\varepsilon} < \delta$ and 
\begin{align} \label{condeps}
	\varepsilon < \min \left( \frac{\delta}{C_0}, \frac{\tilde{\varepsilon}}{4K}, \frac{\nu}{16K^2CC_0} \right), \quad C_0 > 2 + \frac{16C K}{3\nu}.
\end{align}
Let us  abbreviate $\xi_0 :=\| (z_0,w_0) \|_{\R^2 \times X^1_\eta} < \varepsilon$ and set
\begin{align*}
	t_\infty : = \sup \Big\{ T>0:\, \exists! (z,\w) \in C([0,T),& \R^2 \times V_\eta) \text{ satisfying \eqref{integralwDGL}, \eqref{integralgammaDGL} on } [0,T) \\
	&  \text{and } \| \w(t) \|_{X^1_\eta} \le K\tilde{\varepsilon},\, |z(t)| \le C_0 \xi_0,\, t \in [0,T) \Big\}.
\end{align*}
Then Lemma \ref{Lemma4.7} with $\varepsilon_0 = \tilde{\varepsilon}$ and $\varepsilon_1 = \frac{C_0 \xi_0}{2} < \delta$ implies $t_\infty \ge t_\star = t_\star(\varepsilon_0,\varepsilon_1,\mu)$ and we denote the unique solution by $(z,\w)$. Using Theorem \ref{semigroup} and Lemma \ref{Lemma4.6} we estimate for all $0 \le t < t_\infty$
\begin{align*}
	\| \w(t) \|_{X^1_\eta} & \le \| e^{t\L_\eta} \w_0 \|_{X^1_\eta} + \int_0^t \| e^{(t-s)\L_\eta} r^{[w]}(z(s),\w(s)) \|_{X^1_\eta} ds \\
	& \le Ke^{-\nu t} \| \w_0 \|_{X^1_\eta} + \int_0^t e^{-\nu (t-s)} \| r^{[w]}(z(s),\w(s) \|_{X^1_\eta} ds \\
	& \le Ke^{-\nu t} \| \w_0 \|_{X^1_\eta} + KC \int_0^t e^{-\nu(t-s)} \left( |z(s)| + \| \w(s) \|_{X^1_\eta} \right) \| \w(s) \|_{X^1_\eta} ds \\
	& \le Ke^{-\nu t} \xi_0 + KC C_0 \int_0^t e^{-\nu(t-s)} \left( \xi_0 + \| \w(s) \|_{X^1_\eta} \right) \| \w(s) \|_{X^1_\eta} ds.
\end{align*}
Then the Gronwall estimate in Lemma \ref{Gronwall} implies due to \eqref{condeps}
\begin{align} \label{stability:proof1}
	\| \w(t) \|_{X^1_\eta} \le 2K e^{-\frac{3}{4}\nu t} \xi_0 < 2K e^{-\frac{3}{4}\nu t} \varepsilon < \frac{\tilde{\varepsilon}}{2}, \quad t \in [0,t_\infty).
\end{align}
This yields
\begin{align}
\begin{split} \label{stability:proof2}
	|z(t)| & \le |z_0| + \int_0^t |r^{[z]}(z(s),\w(s))|ds \le \xi_0 + C \int_0^t \| \w(s) \|_{X^1_\eta} ds \\
	& \le \xi_0 + 2KC\xi_0 \int_0^t e^{-\frac{3}{4}\nu s} ds \le \xi_0 + \frac{8CK}{3\nu} \xi_0 < \frac{C_0 \xi_0}{2}, \quad t \in [0,t_\infty).
\end{split}
\end{align}
We show that $t_\infty < \infty$ leads to a contradiction. The estimates \eqref{stability:proof1}, \eqref{stability:proof2} imply
\begin{align*}
	\| \w(t_\infty - \tfrac{1}{2}t_\star) \|_{X^1_\eta} < \frac{\tilde{\varepsilon}}{2} =\varepsilon_0, \quad |z(t_\infty - \tfrac{1}{2}t_\star)| < \frac{C_0 \xi_0}{2} = \varepsilon_1.  
\end{align*}
Now we can apply Lemma \ref{Lemma4.7} once again to the integral equations \eqref{integralwDGL}, \eqref{integralgammaDGL} with $\w_0 = \w(t_\infty - \tfrac{1}{2}t_\star)$ and $z_0 = z(t_\infty - \tfrac{1}{2}t_\star)$ and obtain a solution $(\tilde{z}, \tilde{\w})$ on $[0,t_\star)$ with
\begin{alignat*}{3}
	& \tilde{\w}(0) = \w(t_\infty - \tfrac{1}{2}t_\star), & \quad & \| \w(t)\|_{X^1_\eta} \le K\tilde{\varepsilon},\, & \quad &  t\in[0,t_\star) \\
	& \tilde{z}(0) = z(t_\infty - \tfrac{1}{2}t_\star), & \quad & |z(t)| \le C_0 \xi_0,\, & \quad & t \in [0,t_\star).
\end{alignat*}
Define
\begin{align*}
	(\bar{z},\bar{\w})(t) := \begin{cases} (z,\w)(t), & t \in [0,t_\infty - \tfrac{1}{2} t_\star ] \\
	(\tilde{z}, \tilde{\w})(t - t_\infty + \tfrac{1}{2}t_\star), & t \in (t_\infty - \tfrac{1}{2}t_\star,t_\infty + \tfrac{1}{2}t_\star).
	\end{cases}
\end{align*}
Then $(\bar{z}, \bar{\w})$ is a solution on $[0,t_\infty + \tfrac{1}{2}t_\star)$ with $\| \bar{\w}(t) \|_{X^1_\eta} \le K \tilde{\varepsilon}$ and $|\bar{z}(t)| \le C_0 \xi_0$.
  This contradicts the definition of $t_\infty$. Hence $t_\infty = \infty$ and \eqref{stability:proof1} holds on $[0,\infty)$. Further, we see that the integral
\begin{align*}
	z_\infty := z_0 + \int_0^\infty r^{[z]}(z(s),\w(s))ds
\end{align*}
exists since
\begin{align*}
	|z(t) - z_\infty| & \le \int_t^\infty |r^{[z]}(z(s),\w(s))|ds  \\
	& \le C \int_t^{\infty} \| \w(s)\|_{X^1_\eta} \le 2KC\xi_0 \int_t^\infty e^{-\frac{3}{4} \nu s} ds = \frac{8KC}{3 \nu} e^{-\frac{3}{4} \nu t } \xi_0.
\end{align*}
Thus the first estimate in ii) is proven with $K_0 = 2K + \frac{8KC}{3 \nu}$ and $\tilde{\beta} = \frac{3}{4} \nu$. The second estimate is obtained by
\begin{align*}
	|z_\infty| \le |z(0) - z_\infty| + |z_0| \le (K_0 + 1) \xi_0.
\end{align*}
It remains to show the regularity of $(z,\w)$. By Lemma \ref{Lemma4.7} one infers $r^{[z]}(z(\cdot), \w(\cdot)) \in C([0,\infty), \R^2)$ and thus $z \in C^1([0,\infty),\R^2)$. Furthermore, let $r(t) := r^{[w]}(z(t),\w(t))$. Suppose $0 \le s \le t < \infty$. Then by Lemma \ref{Lemma4.7} we find some $C_r > 0$ such that
\begin{align*}
	\| r(t) - r(s) \|_{X_\eta} & = \| r^{[w]}(z(t),\w(t)) - r^{[w]}(z(s),\w(s)) \|_{X_\eta} \\
	& \le C \left( |z(t) - z(s)| + \| \w(t) - \w(s) \|_{X^1_\eta} \right) \\
	& \le C \left( \int_s^t | r^{[z]}(z(\sigma), \w(\sigma)) | d\sigma + \int_s^t \| r^{[w]}(z(\sigma), \w(\sigma)) \|_{X^1_\eta} d\sigma \right) \\
	& \le C \bigg( C \int_s^t \| \w(\sigma) \|_{X^1_\eta}  d\sigma + C \int_s^t |z(\sigma)| + \| \w(\sigma) \|_{X^1_\eta} d\sigma \bigg) \le C_r (t-s).  
\end{align*} 
This implies $r \in C^\alpha ([0,\infty), X_\eta)$ for every $\alpha \in (0,1)$ and
  for arbitrary $s > 0$, 
\begin{align*}
	\int_0^s \| r(t) \|_{X_\eta} dt & = \int_0^s \| r^{[w]}(z(t), \w(t)) \|_{X_\eta} dt \le C\int_0^s \| \w(t) \|_{X^1_\eta} dt < \infty.
\end{align*}
Now the regularity of $\w$ is a consequence of the well known theory of semilinear parabolic equations and can be concluded, for instance, using \cite[Thm. 1.2.1]{Amann} \cite[Thm. 3.2.2]{Henry}.
\end{proof}

We conclude with the

\begin{proof}[Proof of Theorem \ref{Theorem4.10}]
We choose $\mu_0$ from Theorem \ref{thm4.17} and possibly decrease it further such that $\mu_0 \le \mu_1$ with $\mu_1$ from Lemma \ref{lemma4.18}. We take the sets $V,W$ from Lemma \ref{lemmatrafo} and let $\delta > 0$ be so small such that the ball $B_\delta = \{ \u \in X_\eta: \| \u \|_{X_\eta} \le \delta\}$ is contained in the image of  $V$ under $T_\eta$ and its projection $P_\eta(B_\delta)$ in the image of $W$ under $\Pi_\eta$, i.e. $B_\delta \subset T_\eta(V)$ and $P_\eta(B_\delta) \subset \Pi_\eta(W)$. Then the inverse maps $T_\eta^{-1}$, $\Pi_\eta^{-1}$ exist on $B_\delta$, respectively $P_\eta(B_\delta)$, and are diffeomorphic. Moreover, let
\begin{align*}
	C_\Pi := \sup_{\v \in B_\delta} \frac{| \Pi_\eta^{-1}(P_\eta\v) |}{\| \v \|_{X_\eta}}
\end{align*}
and, since the group action is smooth, we find $C \ge 1$ such that
\begin{align*}
	\| a(\chi^{-1}(z_1))\v_\star - a(\chi^{-1}(z_2))\v_\star \|_{X^1_\eta} \le C |z_1 - z_2| \quad \forall z_1,z_2 \in \Pi_\eta^{-1}(P_\eta (B_\delta)).
\end{align*}
Decrease $\varepsilon > 0$ from Theorem \ref{Theorem4.9} such that the  solution $(z,\w)$ of \eqref{wDGL}, \eqref{gammaDGL}  for initial values smaller than $\varepsilon$ satisfy $\w(t) \in T_\eta^{-1}(B_\delta)$ and $z(t) \in \Pi_\eta^{-1} (P_\eta(B_\delta))$ for all $t \in [0,\infty)$. \\
We restrict the size of the initial perturbation $\v_0$ by the condition
\begin{align*}
	\varepsilon_0 < \min \left( \frac{\varepsilon}{C_\Pi (1+C) + 1}, \frac{\pi}{2K_0 +1}, \frac{\delta}{2K_0(3C+1)} \right)
\end{align*}
with $K_0$ from Theorem \ref{Theorem4.9}. The initial values for the $(z,\w)$-system
are defined by
\begin{align*}
	(z_0, \w_0) := T_\eta^{-1} ( \v_0 ) = \big(\Pi_\eta^{-1}(P_\eta \v_0), \v_0 + \v_\star - a(\chi^{-1}(z_0)) \v_\star \big).
\end{align*}
Then $|z_0| \le C_\Pi \| \v_0 \|_{X_\eta}$ holds and
\begin{align*}
\begin{split}
	\| (z_0, \w_0) \|_{\R^2 \times X_\eta^1} \le |z_0| + \| a(\chi^{-1}(z_0))\v_\star - \v_\star \|_{X^1_\eta} + \| \v_0 \|_{X^1_\eta} \le C_\Pi (1+C)\varepsilon_0 + \varepsilon_0 < \varepsilon.
\end{split}
\end{align*}
Thus, by Theorem \ref{Theorem4.9}, there are $z \in C^1([0,\infty),\R^2)$ and $\w \in C((0,\infty),V^2_\eta) \cap C^1((0,\infty),V_\eta)$ such that $(z,\w)$ solves \eqref{wDGL}, \eqref{gammaDGL} with $z(0) = z_0$, $\w(0) = \w_0$ and
\begin{align*}
	\| \w(t) \|_{X^1_\eta} \le K_0 \varepsilon_0, \quad |z(t)| \le |z(t) - z_\infty| + |z_\infty| \le (2K_0 + 1)\varepsilon_0 < \pi, \quad t \in [0,\infty).
\end{align*}
Hence, $z(t)$ lies in the chart $(U,\chi)$ for all $t \in [0,\infty)$ and we can define $\gamma = \chi^{-1}(z) \in C^1([0,\infty),\G)$. Set
\begin{align*}
	\u(t) = a(\gamma(t)) \v_\star + \w(t), \quad t \in [0,\infty).
\end{align*}
Then $\u \in  C((0,\infty), Y_\eta) \cap C^1([0,\infty),X_\eta)$ and by Lemma \ref{lemmatrafo} and the construction of the decomposition in section \ref{sec5}, we conclude $\u_t = \F(\u)$ and $\u(0) = \v_\star + \v_0$. \\
With $\gamma_\infty = \chi^{-1}(z_\infty)$ we have by Theorem \ref{Theorem4.9},
\begin{align*}
	\| \w(t) \|_{X^1_\eta}& + |\gamma(t) - \gamma_\infty|_G  = \| \w(t) \|_{X^1_\eta} + |z(t) - z_\infty| \\
	& \le K_0 e^{-\beta t}\| (z_0, \w_0) \|_{\R^2 \times X^1_\eta}
	 \le K e^{-\beta t} \| \v_0 \|_{X^1_\eta}, 
\end{align*}
where $K = C_\Pi (1 + C)K_0+K_0$. We further estimate the asymptotic phase,
\begin{align*}
	|\gamma_\infty|_\G & \le |\gamma_0|_\G + |\gamma_0 - \gamma_\infty|_\G 
	 = |z_0| + |z_0 - z_\infty| \\
	& \le C_\Pi \| \v_0\|_{X_\eta^1} + K_0 \| (z_0,\w_0) \|_{\R^2 \times X^1_\eta} \le C_\infty \| \v_0 \|_{X^1_\eta}
\end{align*}
with $C_\infty = C_\Pi(1 + K_0) + K_0(1 + CC_\Pi)$. Finally, we show uniqueness of $\u$.
First note
\begin{align*}
	\| \u(t) - \v_\star \|_{X_\eta} \le C |z(t) - z_\infty| + \| \w(t) \|_{X_\eta} + C |z_\infty| \le (3C + 1) K_0 \varepsilon_0 \le \frac{\delta}{2}.
\end{align*}
Assume there is another solution $\tilde{\u}$ of \eqref{CP} on $[0,T)$ for some $T > 0$. Let
\begin{align*}
	\tau := \sup \{ t \in [0,T): \| \tilde{\u} - \v_\star \|_{X_\eta} \le \delta \text{ on } [0,t) \}.
\end{align*}
Then there is a solution $(\tilde{z}, \tilde{\w})$ of \eqref{wDGL}, \eqref{gammaDGL} on $[0,\tau)$ such that $T_\eta(\tilde{z}(t), \tilde{\w}(t)) = \tilde{\u}(t) - \v_\star$ and, therefore, $\tilde{\u}(t) = a(\tilde{\gamma}(t))\v_\star + \tilde{\w}(t)$, $\tilde{\gamma}(t) = \chi^{-1}(\tilde{z}(t))$. But since $(z,\w)$ is unique we conclude $(\tilde{z}, \tilde{\w}) = (z,\w)$ and $\u(t) = \tilde{\u}(t)$ on $[0,\tau)$. Now assume
    $\tau < T$. Then we  have
\begin{align*}
	\frac{\delta}{2} \ge \| \u(t) - \v_\star\|_{X_\eta} = \| \tilde{\u}(t) - \v_\star\|_{X_\eta} \quad \text{for all} \; t \in [0,\tau).
\end{align*}
Since the right-hand side converges to $\delta$ as $t \rightarrow \tau$, we arrive at a contradiction.
\end{proof}

\sect{Appendix} \label{secA}
Consider the differential operator
\begin{align*}
  L_0 u = A u'' + c u', 
  \end{align*}
where $c >0$ and $ A \in \R^{m,m}$ satisfies $\Re(\lambda)>0$ for all
$\lambda \in \sigma(A)$. 
\begin{lemma}[Limits of solutions] \label{lemmaA1}
  Let $r \in C(\R,\R^{m})$ have limits $\lim_{x \to \pm \infty}r(x)$ and
  let $v \in C^2(\R,\R^m)$ be a bounded solution of $L_0 v = r$.
  Then the following limits exist and vanish
  \begin{equation*}
    \lim_{x \to \pm \infty} r(x) = 0 = \lim_{x \to \pm \infty}v'(x)=
    \lim_{x \to \pm \infty}v''(x).
  \end{equation*}
\end{lemma}
\begin{proof}Consider first $x \ge 0$.
 Then we can write $v$ for some $a_1,a_2 \in \R^m$ as
  \begin{align} \label{eqA:decompv} v(x) = Y_1(x) a_1 + Y_2(x) a_2 + v_3(x),
  \end{align}
  where $Y_1(x)=I$ and $Y_2(x)=\exp(- c A^{-1}x)$ form a fundamental
  system for $L_0$  and $v_3$ solves $L_0v_3=r$, $v_3(0)=v_3'(0)=$, i.e.
  \begin{equation} \label{eqA:solformula}
    v_3(x) = \int_0^{\infty} G(x,\xi) r(\xi) d\xi, \quad
    G(x,\xi)= \begin{cases} \frac{1}{c} (I -Y_2(x-\xi)), & 0 \le \xi \le x,\\            0,  & 0 \le x < \xi .
    \end{cases}
  \end{equation}
  By the positivity of $A$ and $c$ we have $|Y_2(x)| \le C \exp(- b x), x \ge 0$
  for some $b >0$. Since $v,Y_1,Y_2$ are bounded on $\R_+$, so is $v_3$.
  If $r_+=\lim_{x \to \infty} r(x) \neq 0$ then we have the following lower bound for $0 < x_0< x$
  \begin{align*}
    \big| v_3(x)\big|  \ge&\, \Big| \int_{x_0}^x  G(x,\xi)d\xi r_+\Big|  - \Big| \int_0^{x_0} G(x,\xi) r(\xi) d\xi\Big| - \Big|  \int_{x_0}^x G(x,\xi)(r(\xi)-r_+) d\xi \Big| \\
    \ge & c^{-1} \big((x - x_0)|r_+| - 2 C |A|c^{-1}|r_+| - (1+C)\|r\|_{L^{\infty}} x_0 - (x-x_0)(1+C) \sup_{\xi \ge x_0}|r(\xi)- r_+| \big).
  \end{align*}
  The last term can be absorbed into the first term by taking $x_0$ large,
  and the resulting term dominates the middle terms as $x \to \infty$. Hence $v_3$
  is unbounded and we arrive at a contradiction.
  For the derivative we find
  \begin{align*}
    v'(x) = Y_2'(x) a_2 - \frac{1}{c} \int_0^x Y_2'(x-\xi) r(\xi) d\xi,
  \end{align*}
  which together with $r_+=0$ and the exponential decay of $Y_2'$ yields
  $\lim_{x\to \infty}v'(x)=0$. 

  Instead of considering $L_0$ on $\R_-$ we reflect domains and consider
  $L_0$ on $\R_+$ but now with $c <0$. Formulas \eqref{eqA:decompv}
  and \eqref{eqA:solformula} still hold but with the Green's function given
  by
  \begin{align*}
    G(x,\xi)= \frac{1}{c} \begin{cases}
      I- \exp(c A^{-1} \xi) , & 0 \le \xi \le x, \\
      \exp(cA^{-1}(\xi-x))- \exp(c A^{-1}\xi), & 0 \le x < \xi.
    \end{cases}
    \end{align*}
    Note that $c <0$ implies an estimate
    \begin{align*}
      |G(x,\xi)| \le C \begin{cases} 1, &0 \le \xi \le x,\\
        \exp(-b(\xi - x)),& 0 \le x < \xi.
      \end{cases}
      \end{align*}
    Hence the integral in \eqref{eqA:solformula} converges and provides a linear
    upper bound for $v_3(x)$. Since $Y_2(x) a_2$ grows exponentially if $a_2 \neq 0$, we obtain $a_2=0$ from the boundedness of $v$. As in case $c >0$
    we then derive a linear lower bound for $|v_3(x)|$ if $r_+ \neq 0$.
    In this way, we find again $r_+=0$ and then $\lim_{x \to \infty}v'(x)=0$
    from
    \begin{align*}
      v_3'(x) = -A^{-1} \int_x^{\infty} \exp(cA^{-1}(\xi - x))r(\xi) d\xi.
    \end{align*}
\end{proof}

\begin{proof}[Proof of Lemma \ref{lem:asym}] The TOF $v_{\star}$ satisfies
  \begin{align*} L_0 v_{\star} = - S_{\omega} v_{\star}-f(v_{\star}) =:r,
  \end{align*}
  hence Lemma \ref{lemmaA1} shows
  $\lim_{x \to \pm \infty}v_{\star}'(x)=\lim_{x \to \pm \infty}v_{\star}''(x)=0$
  as well as 
  \begin{align*}
    0 =\lim_{x \to \infty} r(x) = - S_{\omega}v_{\infty} - f(v_{\infty})
    = - (S_{\omega}+g(|v_{\infty}|^2))v_{\infty}.
  \end{align*}
  This is the real version of the complex equation $(i \omega + G(|V_{\infty}|^2))V_{\infty}=0$, so that $S_{\omega}+g(|v_{\infty}|^2)=0$ follows.
  
\end{proof}

\begin{proof}[Proof of Theorem \ref{decay}.]
The profile $v_\star$ is a solution of \eqref{statcomovsys} and $f \in C^3$ by Assumption \ref{A1}. Therefore $v_\star \in C^5_b(\R,\R^2)$. For the estimate on $\R_-$ we transform \eqref{statcomovsys} into a $4$-dimensional first order system with $w = (w_1,w_2)^\top$, $w_1 = v_\star$, $w_2 = v_\star'$. Then $w$ solves
\begin{align} \label{systemminus}
	w' = \mathcal{H}(w), \quad \mathcal{H}(w) = \vek{w_2}{-A^{-1}(cw_2 + S_\omega w_1 + f(w_1))}
\end{align}
and $w = (v_\star, v_\star')^\top \rightarrow 0$ as $x \rightarrow -\infty$ (cf. Lemma \ref{lem:asym}). Now zero is an equilibrium of \eqref{systemminus} with
\begin{align*}
	D\mathcal{H}(0) = \begin{pmatrix} 0 & I_2 \\ -A^{-1} (S_\omega + Df(0)) & -cA^{-1} \end{pmatrix}, \quad
	Df(0) = \begin{pmatrix} g_1(0) & -g_2(0) \\ g_2(0) & g_1(0) \end{pmatrix}.
\end{align*}
One can show that Assumption \ref{A1} implies zero to be a hyperbolic equilibrium of \eqref{systemminus} with local stable and unstable manifolds of dimension $2$. Since convergence to hyperbolic equilibria is known to be exponentially fast (cf. \cite[Theorem 7.6]{Sideris}), we conclude the desired estimate on $\R_-$. \\
For the estimate on $\R_+$ we use an ansatz from  \cite{Saarloos} with polar coordinates,
\begin{align} \label{polar}
	v_\star(x) = r(x) \vek{\cos \phi(x)}{\sin \phi(x)},
\end{align}
and introduce the new variables $q := \phi$ and $\kappa := \frac{r'}{r}$. Plugging the ansatz \eqref{polar} into \eqref{statcomovsys} then gives the equation for $(r,q,\kappa)$,
\begin{align} \label{polarsystem}
	\begin{pmatrix}
	r \\ q \\ \kappa	
	\end{pmatrix} ' =
	\begin{pmatrix}
	r \kappa \\ 
	\vek{q^2 - \kappa^2}{-2\kappa q} - A^{-1} \vek{c\kappa + g_1(|r|^2)}{ cq + \omega + g_2(|r|^2) }
	\end{pmatrix} =: \Gamma(r,\kappa, q). 
\end{align}
We define $r_\infty := |v_\infty|$ and $\phi_\infty := \arg(v_\infty)$. Then we have $r \rightarrow r_\infty$ as $x \rightarrow \infty$, since $v_\star \rightarrow v_\infty$ as $x \rightarrow \infty$. In addition, $v_\star' \rightarrow 0$ as $x \rightarrow \infty$, by Lemma \ref{lem:asym}, which implies $r' \rightarrow 0$ 
as $x \rightarrow \infty$. Therefore we obtain $\kappa = \frac{r'}{r} \rightarrow 0$ as $x\to \infty$ and further
\begin{align*}
	r' \vek{\cos \phi}{\sin \phi} + r q \vek{-\sin \phi}{\cos \phi} = v_\star' \rightarrow 0, \quad x \rightarrow \infty.
\end{align*}
This shows $q \rightarrow 0$. Summarizing we have $(r,\kappa,q) \rightarrow (r_\infty,0,0)$ as $x \rightarrow \infty$. Now one verifies that $(r_\infty,0,0)$ is a hyperbolic equilibrium of \eqref{polarsystem} with stable manifold of dimension equal to $2$ and unstable manifold of dimension equal to $1$. Again since convergence to hyperbolic equilibria is known to be exponentially fast (cf. \cite[Theorem 7.6]{Sideris}), we find $K_0,\mu_\star > 0$ such that for $x \ge 0$,
\begin{align*}
	|(r',\kappa',q')| = |\Gamma(r,\kappa,q) - \Gamma(r_\infty,0,0) | \le C |(r,\kappa,q) - (r_\infty,0,0)| \le K_0 e^{-\mu_\star x}
\end{align*}
where we use the fact that $\Gamma \in C^1$ by Assumption \ref{A1}. Finally we find $K > 0$ such that
\begin{align*}
	|v_\star(x) - v_\infty| + |v_\star'(x)| & \le \Big| r(x) \vek{\cos \phi(x)}{ \sin \phi(x)} - r_\infty \vek{\cos \phi_\infty}{ \sin \phi_\infty} \Big| + |r'(x)| + |r(x) q(x)| \\
	& \le |r(x) - r_\infty| + |r_\infty | |\phi(x) - \phi_\infty| + |r'(x)| + \| r \|_{L^\infty} |q(x)| \\
	& \le |r(x) - r_\infty| + |r_\infty | \int_{x}^\infty |q(x)| dx + |r'(x)| + \| r \|_{L^\infty} |q(x)| \\
	& \le K e^{-\mu_\star x}
\end{align*}
for all $x \ge 0$. Since $f \in C^3$, the estimates for $v_\star''$ and $v_\star'''$ then follow by differentiating \eqref{comovsys}.
\end{proof}

\begin{proof}[Proof of Lemma \ref{lemmagroup}]
We note that translations on $L^2_\eta$ are continuous and the estimate $\| v (\cdot - \tau) \|_{L^2_\eta} \le e^{\mu |\tau|}\| v \|_{L^2_\eta}$ for all $v \in L^2_{\eta}$ holds. Further, if $v \in H^1_\eta$ it is straightforward to show $\| v(\cdot - \tau )- v \|_{L^2_\eta} \le |\tau| e^{\mu |\tau|} \| v_x \|_{L^2_\eta}$ and the same holds true if $v$ is replaced by the tremplate function $\hat{v}$. Using these facts and invariance under rotation of the norms we obtain continuity of the group action on $X_\eta$ by
\begin{align*}
 \| a(\gamma) \v \|_{X_\eta} & \le |\rho| + \| v(\cdot - \tau) - \rho \hat{v} \|_{L^2_\eta} \le |\rho| + \| v(\cdot - \tau) - \rho \hat{v} (\cdot - \tau) \|_{L^2_\eta} + |\rho| \| \hat{v}(\cdot - \tau) - \hat{v} \|_{L^2_\eta} \\
 & \le |\rho| + e^{\mu |\tau|} ( \| v - \rho \hat{v} \|_{L^2_\eta} + |\rho| |\tau| \| \hat{v}_x \|_{L^2_\eta} ) \le C \| \v \|_{X_\eta}.
\end{align*}
Using the continuity of translations on $L^2_\eta$ once again yields $\| a(\gamma) \v \|_{Y_\eta} \le C \| \v \|_{Y_\eta}$. It is easy to verify the properties 
$a(\gamma_1)a(\gamma_2) = a(\gamma_1 \circ \gamma_2)$ and 
$a(\gamma)^{-1} = a(\gamma^{-1})$ so that $a(\cdot) \in GL[X_\eta]$ is a homomorphism. The continuity of the group action in $\G$ for $\v \in X_\eta$ follows by
\begin{align*}
 & \| a(\gamma) \v - \v \|_{X_\eta}  \\
 & \le | R_{\theta}\rho -  \rho| + \| R_{\theta} v(\cdot - \tau) - R_{\theta}\rho \hat{v} - ( v - \rho \hat{v} )\|_{L_\eta^2} \\
	& \le | R_{\theta}\rho - \rho | + \| R_{\theta} ( v(\cdot - \tau) -\rho \hat{v} ) - R_{\theta} ( v - \rho \hat{v} ) \|_{L_\eta^2}+ \| R_{\theta} (v - \rho \hat{v} ) - ( v - \rho \hat{v} ) \|_{L_\eta^2} \\
	& \le |R_{\theta} - I | \left( |\rho| + \| v- \rho \hat{v}\|_{L_\eta^2} \right) + \| v(\cdot - \tau) - v\|_{L_\eta^2} \rightarrow 0 \quad \text{as} \quad (\theta,\tau) \rightarrow 0.
\end{align*}
Similarly, for $v \in Y_\eta$ we have
\begin{align*}
& \| a(\gamma) \v - \v \|_{Y_\eta}^2 = \| a(\gamma) \v - \v \|_{X_\eta}^2 + \sum_{\alpha = 1}^2 \|  R_{\theta} \partial^\alpha v(\cdot - \tau) - \partial^\alpha v \|_{L_\eta^2}^2 \rightarrow 0 \quad \text{as} \quad (\theta,\tau) \rightarrow 0.
\end{align*}
It is left to show that $a(\cdot)\v$ is of class $C^1$ for $v \in Y_\eta$ and to
compute its derivative. For this purpose it suffices to prove the assertion at $\gamma = \one = (0,0)$. Let us take $h = (h_1,h_2) \in \R^2$ small such that $\chi^{-1}(h) \in U$. Then
\begin{align*}
	& \| a(\chi^{-1}(h)) \v - \v - h_1 \mathbf{S_1}\v + h_2 \v_x \|_{X_\eta} \\
 &  \le |R_{-h_1} \rho - \rho + h_1 S_1 \rho| + \| R_{-h_1} (v(\cdot - h_2) - \rho \hat{v}) - (v - \rho \hat{v}) + h_1 S_1 (v- \rho \hat{v}) + h_2 v_x \|_{L_\eta^2}.
\end{align*}
Since $\partial_\theta R_\theta \rho_{|\theta = 0} = S_1 \rho$, the first term is $o(|h|)$. The second term is less obvious. We frequently add zero and split  into serveral terms
\begin{align*}
& \| R_{-h_1} (v(\cdot - h_2) - \rho \hat{v}) - (v - \rho \hat{v}) + h_1 S_1 (v- \rho \hat{v}) + h_2 v_x \|_{L_\eta^2} \\
	& \le \| R_{-h_1} (v-\rho \hat{v})(\cdot - h_2) - (v - \rho \hat{v})(\cdot - h_2) + h_1 S_1 (v- \rho \hat{v})(\cdot - h_2)\|_{L_\eta^2} \\
	& \quad + \| R_{-h_1} \rho \hat{v}(\cdot - h_2) - R_{-h_1} \rho \hat{v} +(v-\rho \hat{v})(\cdot - h_2) - h_1S_1(v-\rho \hat{v})(\cdot - h_2) \\
	& \quad \quad  - (v-\rho \hat{v}) + h_1S_1 (v-\rho \hat{v}) + h_2 v_x \|_{L_\eta^2} \\
	& \le T_1 + \| (v-\rho \hat{v})(\cdot - h_2) - (v-\rho \hat{v}) + h_2(v_x - \rho \hat{v}_x) \|_{L_\eta^2} \\
	& \quad + \| h_2 \rho \hat{v}_x + R_{-h_1} \rho \hat{v}(\cdot - h_2) - R_{-h_1} \rho \hat{v} - h_1 S_1 (v-\rho \hat{v})(\cdot - h_2) + h_1S_1(v- \rho \hat{v}) \|_{L_\eta^2} \\
	& \le T_1 + T_2 + \| R_{-h_1} [ \rho \hat{v} (\cdot - h_2) - \rho \hat{v} + h_2 \rho \hat{v}_x] \|_{L_\eta^2} \\
	& \quad + \| h_2 \rho \hat{v}_x - h_2 R_{-h_1} \rho \hat{v}_x - h_1S_1 (v-\rho \hat{v})(\cdot - h_2) + h_1S_1(v- \rho \hat{v}) \|_{L_\eta^2} \\
	& \le T_1 + T_2 + T_3 + \| -h_1 S_1 (v- \rho \hat{v})(\cdot - h_2) + h_1 S_1(v- \rho \hat{v}) - h_2(h_1S_1 v_x - h_1 S_1 \rho \hat{v}_x) \|_{L_\eta^2} \\
	& \quad + \| h_2 \rho \hat{v}_x - h_2 R_{-h_1} \rho \hat{v}_x + h_2h_1S_1 v_x - h_2 h_1 S_1 \rho \hat{v}_x \|_{L_\eta^2} \\
	& \le T_1 + T_2 + T_3 + T_4 + \| R_{-h_1} h_2 \rho \hat{v}_x - h_2 \rho \hat{v}_x + h_1 S_1 h_2 \rho \hat{v}_x \|_{L_\eta^2} + \| h_2h_1 S_1 v_x\|_{L^2_\eta} \\
	& = T_1 + T_2 + T_3 + T_4 + T_5 + T_6.
\end{align*}
Now $T_1, T_5 = o(|h|)$ holds since rotations are smooth and $\partial_\theta R_\theta \rho_{|\theta = 0} = S_1 \rho$. Further, $T_6 = o(|h|)$ is obvious. 
Finally $T_2, T_3, T_4 = o(|h|)$ hold, since translations on $H^1_\eta$ are smooth and therefore $\| v(\cdot - \tau) - v + h v_x \| = o(|h|)$ for $v \in H^1_\eta$. This completes the proof.
\end{proof}

\section*{Acknowledgment} Both authors thank the CRC 1283  
 `Taming uncertainty and
profiting from randomness and low regularity in analysis, stochastics and their applications’ at Bielefeld University for support during preparation of
this paper and of the thesis \cite{Doeding}. 

\end{document}